\documentclass[leqno,12pt]{amsart} 

\usepackage{amssymb} 
\usepackage{amsmath}
\usepackage{amsthm}
\usepackage{amscd}
\usepackage{bm}
\usepackage{graphicx,psfrag,wrapfig}
\theoremstyle{plain} 
\newtheorem{theorem}{Theorem}[section]  
\newtheorem{lemma}[theorem]{Lemma}
\newtheorem{corollary}[theorem]{Corollary}
\newtheorem{proposition}[theorem]{Proposition}

\theoremstyle{definition} 

\newtheorem{remark}[theorem]{Remark}
\newtheorem{example}[theorem]{Example}

\newcommand{\const}{\mathrm{const}}
\newcommand{\norm}[1]{\left|\!\left|#1\right|\!\right|}
\newcommand{\ad}{\mathrm{ad}}
\newcommand{\supp}{\mathrm{supp}}
\newcommand{\vol}{\mathrm{vol}}
\newcommand{\Diam}{\mathrm{Diam}}
\newcommand{\moduli}{\mathcal{M}}
\newcommand{\widim}{\mathrm{Widim}}
\newcommand{\dist}{\mathrm{dist}}

\begin{document}

\title[Instanton approximation, periodic ASD connections, and mean dimension]
{Instanton approximation, periodic ASD connections, and mean dimension} 

\author[S. Matsuo and M. Tsukamoto]{Shinichiroh Matsuo and Masaki Tsukamoto}



\thanks{Shinichiroh Matsuo was supported by Grant-in-Aid for JSPS fellows (19$\cdot$5618), and
Masaki Tsukamoto was supported by Grant-in-Aid for Young Scientists (B) (21740048).}

\date{\today}

\maketitle

\begin{abstract}
We study a moduli space of ASD connections over $S^3\times \mathbb{R}$.
We consider not only finite energy ASD connections but also infinite energy ones.
So the moduli space is infinite dimensional in general.
We study the (local) mean dimension of this infinite dimensional moduli space.
We show the upper bound on the mean dimension by using a ``Runge-approximation" for ASD connections,
and we prove its lower bound by constructing an infinite dimensional deformation theory 
of periodic ASD connections.
\end{abstract}


\section{introduction}  \label{section: introduction}
Since Donaldson \cite{Donaldson-application} discovered his revolutionary theory, 
many mathematicians have intensively studied the Yang-Mills gauge theory.
There are several astonishing results on the structures of the ASD moduli spaces 
and their applications.
But most of them study only finite energy ASD connections and their finite dimensional 
moduli spaces.
Almost nothing is known about infinite energy ASD connections and their infinite dimensional 
moduli spaces.
(One of the authors struggled to open the way to this direction in \cite{Tsukamoto, Tsukamoto-2}.)
This paper studies an infinite dimensional moduli space coming from the Yang-Mills theory 
over $S^3\times \mathbb{R}$.
Our main purposes are to prove estimates on its ``mean dimension'' (Gromov \cite{Gromov}) and to show that 
there certainly exists a non-trivial structure in this infinite dimensional moduli space.
(Mean dimension is a ``dimension of an infinite dimensional space averaged by a group action''.)

The reason why we consider $S^3\times \mathbb{R}$ is that it is one of the simplest non-compact anti-self-dual 
4-manifolds of (uniformly) positive scalar curvature.
(Indeed it is conformally flat.) 
These metrical conditions are used via the Weitzenb\"{o}ck formula (see Section \ref{subsection: construction}).
Recall that one of the important results of the pioneering work of Atiyah-Hitchin-Singer \cite[Theorem 6.1]{A-H-S}
is the calculation of the dimension of the moduli space of (irreducible) self-dual connections over a compact self-dual 4-manifold
of positive scalar curvature.
So our work is an attempt to develop an infinite dimensional analogue of 
\cite[Theorem 6.1]{A-H-S}.

Of course, the study of the mean dimension is just a one-step toward the full understanding 
of the structures of the infinite dimensional moduli space. 
(But the authors believe that ``dimension'' is one of the most fundamental invariants of spaces and 
that the study of mean dimension is a crucial step toward the full understanding.)
So we need much more studies,
and the authors hope that this paper becomes a stimulus to a further study of 
infinite dimensional moduli spaces in the Yang-Mills gauge theory.

Set $X:= S^3\times \mathbb{R}$.
Throughout the paper, the variable $t$ means the variable of the $\mathbb{R}$-factor of $X=S^3\times \mathbb{R}$.
(That is, $t:X\to \mathbb{R}$ is the natural projection.)
$S^3\times \mathbb{R}$ is endowed with the product metric of 
a positive constant curvature metric on $S^3$ and the standard metric on $\mathbb{R}$.
(Therefore $X$ is $S^3(r)\times \mathbb{R}$ for some $r>0$ as a Riemannian manifold,
where $S^3(r) =\{x\in \mathbb{R}^4|\, |x|=r\}$.)
Let $\bm{E} := X\times SU(2)$ be the product principal $SU(2)$-bundle over $X$.
The additive Lie group $\mathbb{R}$ acts on $X$ by 
$X\times \mathbb{R}\ni ((\theta, t), s) \mapsto (\theta, t+s)\in X$.
This action trivially lifts to the action on $\bm{E}$ by
$\bm{E}\times \mathbb{R}\ni ((\theta, t, u), s) \mapsto (\theta, t+s, u)\in \bm{E}$.

Let $d\geq 0$. We define $\moduli_d$
as the set of all gauge equivalence classes of ASD connections $\bm{A}$ on $\bm{E}$ satisfying 
\begin{equation} \label{eq: Brody condition}
 \norm{F(\bm{A})}_{L^\infty(X)}\leq d.
\end{equation}
Here $F(\bm{A})$ is the curvature of $\bm{A}$.
$\moduli_d$ is equipped with the topology of $\mathcal{C}^\infty$-convergence on compact subsets:
a sequence $[\bm{A}_n]$ $(n\geq 1)$ converges to $[\bm{A}]$ in $\moduli_d$
if there exists a sequence of gauge transformations $g_n$ of $\bm{E}$ such that 
$g_n(\bm{A}_n)$ converges to $\bm{A}$ as $n\to \infty$ in the $\mathcal{C}^\infty$-topology over every compact subset in $X$.
$\moduli_d$ becomes a compact metrizable space
by the Uhlenbeck compactness (\cite{Uhlenbeck, Wehrheim}).
The additive Lie group $\mathbb{R}$ continuously acts on $\moduli_d$ by 
\[ \moduli_d \times \mathbb{R} \to \moduli_d, \quad 
([\bm{A}], s) \mapsto [s^*\bm{A}],\]
where $s^*$ is the pull-back by $s :\bm{E}\to \bm{E}$.
Then we can consider the mean dimension $\dim(\moduli_d:\mathbb{R})$.
Intuitively, 
\[ \dim(\moduli_d:\mathbb{R}) = \frac{\dim\moduli_d}{\vol(\mathbb{R})}.\]
(This is $\infty/\infty$ in general. 
The precise definition will be given in Section \ref{Section: Mean dimension and local mean dimension}.)
Our first main result is the following estimate on the mean dimension.
\begin{theorem}\label{thm: main theorem}
\[ \dim(\moduli_d:\mathbb{R}) < +\infty.\]
Moreover, $\dim(\moduli_d:\mathbb{R})\to +\infty$ as $d\to +\infty$.
\end{theorem}
For an ASD connection $\bm{A}$ on $\bm{E}$, we define $\rho(\bm{A})$ by setting
\begin{equation} \label{eq: definition of rho(A)}
\rho(\bm{A}) := \lim_{T\to +\infty}\frac{1}{8\pi^2 T}
\sup_{t\in \mathbb{R}}\int_{S^3\times[t,t+T]} |F(\bm{A})|^2d\vol.
\end{equation}
This limit always exists because we have the following subadditivity.
\begin{equation*}
 \begin{split}
 \sup_{t\in \mathbb{R}}&\int_{S^3\times[t,t+T_1+T_2]}|F(\bm{A})|^2 d\vol \\
 &\leq \sup_{t\in \mathbb{R}}\int_{S^3\times [t,t+T_1]}|F(\bm{A})|^2d\vol
 + \sup_{t\in \mathbb{R}}\int_{S^3\times [t,t+T_2]}|F(\bm{A})|^2d\vol.
 \end{split}
\end{equation*}
$\rho(\bm{A})$ is translation invariant; For $s\in \mathbb{R}$, 
we have $\rho(s^*\bm{A})=\rho(\bm{A})$, where $s^*\bm{A}$ is the 
pull-back of $\bm{A}$ by the map $s:\bm{E}=S^3\times \mathbb{R}\times SU(2)\to \bm{E}$, 
$(\theta,t,u)\mapsto (\theta,t+s,u)$.
We define $\rho(d)$ as the supremum of $\rho(\bm{A})$ over all ASD connections $\bm{A}$ on $\bm{E}$
satisfying $\norm{F(\bm{A})}_{L^\infty}\leq d$.

Let $\bm{A}$ be an ASD connection on $\bm{E}$.
We call $\bm{A}$ a periodic ASD connection if 
there exist $T>0$, a principal $SU(2)$-bundle $\underbar{E}$ over $S^3\times (\mathbb{R}/T\mathbb{Z})$, 
and an ASD connection $\underbar{A}$ on $\underbar{E}$ such that $(\bm{E}, \bm{A})$ is gauge equivalent to 
$(\pi^*(\underbar{E}), \pi^*(\underbar{A}))$ 
where $\pi:S^3\times \mathbb{R}\to S^3\times (\mathbb{R}/T\mathbb{Z})$
is the natural projection.
(Here $S^3\times (\mathbb{R}/T\mathbb{Z})$ is equipped with the metric induced by the covering map 
$\pi$.)
Then we have 
\begin{equation} \label{eq: rho(A) for periodic ASD}
\rho(\bm{A}) = \frac{1}{8\pi^2 T}\int_{S^3\times [0,T]}|F(\bm{A})|^2d\vol = \frac{c_2(\underbar{E})}{T}.
\end{equation}
We define $\rho_{peri}(d)$ as the supremum of $\rho(\bm{A})$ over all periodic ASD connections $\bm{A}$ on 
$\bm{E}$ satisfying $\norm{F(\bm{A})}_{L^\infty} <d$.
(Note that we impose the strict inequality condition here.)
If $d=0$, then such an $\bm{A}$ does not exist.  Hence we set $\rho_{peri}(0) :=0$.
(If $d>0$, then the product connection $\bm{A}$ is a periodic ASD connection satisfying 
$\norm{F(\bm{A})}_{L^\infty}=0 < d$.)
Obviously we have $\rho_{peri}(d)\leq \rho(d)$.
Our second main result is the following estimates on the ``local mean dimensions''.
\begin{theorem}\label{thm: main theorem on the local mean dimension}
For any $[\bm{A}]\in \moduli_d$, 
\[ \dim_{[\bm{A}]} (\moduli_d:\mathbb{R}) \leq 8\rho(\bm{A}).\]
Moreover, if $\bm{A}$ is a periodic ASD connection satisfying $\norm{F(\bm{A})}_{L^\infty} <d$,
then 
\[ \dim_{[\bm{A}]}(\moduli_d:\mathbb{R}) = 8\rho(\bm{A}).\]
Therefore,
\[ 8\rho_{peri}(d)\leq \dim_{loc}(\moduli_d:\mathbb{R})\leq 8\rho(d).\]
\end{theorem}
Here $\dim_{[\bm{A}]}(\moduli_d:\mathbb{R})$ is the ``local mean dimension'' of $\moduli_d$
at $[\bm{A}]$, and 
$\dim_{loc}(\moduli_d:\mathbb{R}) 
:= \sup_{[\bm{A}]\in \moduli_d} \dim_{[\bm{A}]}(\moduli_d:\mathbb{R})$ 
is the ``local mean dimension'' of $\moduli_d$.
These notions will be
defined in Section \ref{subsection: local mean dimension}.

Note that
\begin{equation*} 
 \lim_{d\to +\infty} \rho_{peri}(d)=+\infty.
\end{equation*}
This obviously follows from the fact that for any integer $n\geq 0$ there exists 
an ASD connection on $S^3\times (\mathbb{R}/\mathbb{Z})$ whose second Chern number is equal to $n$.
This is a special case of the famous theorem of Taubes \cite{Taubes}.
(Note that the intersection form of $S^3\times S^1$ is zero.)
We have $\dim(\moduli_d:\mathbb{R})\geq \dim_{loc}(\moduli_d:\mathbb{R})$ 
(see (\ref{eq: local mean dim. leq mean dim.}) in Section \ref{subsection: local mean dimension}).
Hence the statement that $\dim(\moduli_d:\mathbb{R})\to +\infty$
$(d\to +\infty)$ in Theorem \ref{thm: main theorem}
follows from the inequality
$\dim_{loc}(\moduli_d:\mathbb{R})\geq 8\rho_{peri}(d)$
in Theorem \ref{thm: main theorem on the local mean dimension}.
\begin{remark} \label{remark: another description of moduli_d}
All principal $SU(2)$-bundles over $S^3\times \mathbb{R}$ are gauge equivalent to the product bundle $\bm{E}$.
Hence the moduli space $\moduli_d$ is equal to the space of all gauge equivalence classes $[E,A]$
such that
$E$ is a principal $SU(2)$-bundle over $X$, and that $A$ is an ASD connection on $E$ satisfying 
$|F(A)|\leq d$.
We have $[E_1,A_1]=[E_2,A_2]$ if and only if there exists a bundle map $g:E_1\to E_2$
satisfying $g(A_1)=A_2$.
In this description, the topology of $\moduli_d$ is described as follows.
A sequence $[E_n, A_n]$ $(n\geq 1)$ in $\moduli_d$ converges to $[E, A]$ if 
and only if there exist gauge transformations $g_n:E_n\to E$ $(n\gg 1)$ such that 
$g_n(A_n)$ converges to $A$ as $n\to \infty$ in $\mathcal{C}^\infty$ over every compact subset in $X$.
\end{remark}
\begin{remark}
An ASD connection satisfying the condition (\ref{eq: Brody condition}) is a Yang-Mills
analogue of a ``Brody curve''
(cf. Brody \cite{Brody}) in the entire holomorphic curve theory (Nevanlinna theory).
It is widely known that there exist several similarities between the Yang-Mills gauge theory and 
the theory of (pseudo-)holomorphic curves (e.g. Donaldson invariant vs. Gromov-Witten invariant).
On the holomorphic curve side, several researchers in the Nevanlinna theory have systematically studied
the value distributions of holomorphic curves (of infinite energy) from the complex plane $\mathbb{C}$.
They have found several deep structures of such infinite energy holomorphic curves.
Therefore the authors hope that infinite energy ASD connections also have deep structures.
\end{remark}
The rough ideas of the proofs of the main theorems are as follows. 
(For more about the outline of the proofs, see Section \ref{section: outline of the proofs of the main theorems}.)
The upper bounds on the (local) mean dimension are proved by using the Runge-type approximation 
of ASD connections (originally due to Donaldson \cite{Donaldson}).
This ``instanton approximation'' technique gives a method to approximate infinite energy ASD connections 
by finite energy ones (instantons).
Then we can construct ``finite dimensional approximations'' of $\moduli_d$ by moduli spaces
of instantons.
This gives a upper bound on $\dim(\moduli_d:\mathbb{R})$.
The lower bound on the local mean dimension is proved by constructing an infinite dimensional 
deformation theory of periodic ASD connections.
This method is a Yang-Mills analogue of the deformation theory of ``elliptic Brody curves"
developed in Tsukamoto \cite{Tsukamoto-deformation}.

A big technical difficulty in the study of $\moduli_d$ comes from the point that
ASD equation is not elliptic.
When we study the Yang-Mills theory over compact manifolds, this point can be easily overcome by 
using the Coulomb gauge. But in our situation (perhaps) there is no such good way to recover 
the ellipticity.
So we will consider some ``partial gauge fixings'' in this paper.
In the proof of the upper bound, we will consider the Coulomb gauge over $S^3$ instead of $S^3\times \mathbb{R}$
(see Propositions \ref{prop: gauge fixing on S^3} and \ref{prop: gauge fixing on S^3: flat connection}).
In the proof of the lower bound, we will consider the Coulomb gauge over $S^3\times \mathbb{R}$,
but it is less powerful and more technical than the usual Coulomb gauges over compact manifolds
(see Proposition \ref{prop: Coulomb gauge condition}).

\textbf{Organization of the paper:}
In Section \ref{Section: Mean dimension and local mean dimension} we review the definition of mean dimension and define local mean dimension.
In Section \ref{section: outline of the proofs of the main theorems}
we explain the outline of the proofs of Theorem \ref{thm: main theorem} 
and \ref{thm: main theorem on the local mean dimension}.
Sections \ref{section: solving ASD equation}, \ref{section: continuity of the perturbation} 
and \ref{section: cut-off constructions}
are preparations for the proof of the upper bounds on the (local) mean dimension.
Section \ref{section: non-flat implies irreducible} is a preparation for both proofs of
the upper and lower bounds.
In Section \ref{Section: proof of the upper bound} we prove the upper bounds.
Section \ref{section: Analytic preliminaries for the lower bound} 
is a preparation for the proof of the lower bound.
In Section \ref{section: Proof of the lower bound} we develop the deformation theory of periodic 
ASD connections and prove the lower bound on the local mean dimension.
In Appendix \ref{appendix: Green kernel}
we prepare some basic results on the Green kernel of $\Delta +a$ ($a>0$).

\textbf{Acknowledgement.}
The authors wish to thank Professors Kenji Nakanishi and Yoshio Tsutsumi.
When the authors studied the lower bound on the local mean dimension,
they gave the authors helpful advice.
Their advice was very useful especially for preparing the arguments in Section
\ref{section: Analytic preliminaries for the lower bound}.
The authors also wish to thank Professors Kenji Fukaya and Mikio Furuta for 
their encouragement.

\section{Mean dimension and local mean dimension} \label{Section: Mean dimension and local mean dimension} 
\subsection{Review of mean dimension} \label{subsection: Review of mean dimension}
We review the definitions and basic properties of mean dimension in this subsection.
For the detail, see Gromov \cite{Gromov} and Lindenstrauss-Weiss \cite{Lindenstrauss-Weiss}.

Let $(X, d)$ be a compact metric space, $Y$ be a topological space, and $f:X\to Y$ be a 
continuous map.
For $\varepsilon>0$, $f$ is called an $\varepsilon$-embedding if we have $\Diam f^{-1}(y)\leq \varepsilon$ for all
$y\in Y$.
We define $\widim_\varepsilon(X,d)$ as the minimum integer $n\geq 0$ 
such that there exist a polyhedron $P$ of dimension $n$ and an $\varepsilon$-embedding $f:X\to P$.
We have 
\[ \lim_{\varepsilon\to 0}\widim_\varepsilon (X, d) =\dim X,\]
where $\dim X$ denotes the topological covering dimension of $X$.
For example, consider $[0,1]\times [0,\varepsilon]$ with the Euclidean distance.
Then the natural projection $\pi:[0,1]\times [0,\varepsilon]\to [0,1]$ is an $\varepsilon$-embedding.
Hence $\widim_\varepsilon([0, 1]\times [0,\varepsilon], \mathrm{Euclidean})\leq 1$.
The following is given in Gromov \cite[p. 333]{Gromov}.
(For the detailed proof, see also Gournay \cite[Lemma 2.5]{Gournay} and Tsukamoto \cite[Appendix]{Tsukamoto-deformation}.)
\begin{lemma}\label{lemma: widim of Banach ball}
Let $(V, \norm{\cdot})$ be a finite dimensional normed linear space over $\mathbb{R}$.
Let $B_r(V)$ be the closed ball of radius $r>0$ in $V$. Then
\[ \widim_\varepsilon(B_r(V), \norm{\cdot}) = \dim V \quad (\varepsilon <r).\]
\end{lemma}
$\widim_\varepsilon(X,d)$ satisfies the following subadditivity. (The proof is obvious.)
\begin{lemma} \label{lemma: subadditivity of widim}
For compact metric spaces $(X,d_X)$, $(Y,d_Y)$, we set 
$(X,d_X)\times(Y,d_Y):=(X\times Y, d_{X\times Y})$ with 
$d_{X\times Y}((x_1,y_1),(x_2,y_2)) := \max (d_X(x_1,x_2),d_Y(y_1,y_2))$.
Then we have
\[ \widim_\varepsilon((X,d_X)\times (Y,d_Y)) \leq \widim_\varepsilon(X,d_X)+\widim_\varepsilon(Y,d_Y).\]
\end{lemma}
The following will be used in Section \ref{subsection: proof of mean dim. < infty}
\begin{lemma} \label{lemma: widim and space sum}
Let $(X, d)$ be a compact metric space and suppose $X=X_1\cup X_2$ with 
closed sets $X_1$ and $X_2$. Then 
\[ \widim_\varepsilon(X, d)\leq \widim_\varepsilon(X_1, d)+\widim_\varepsilon(X_2,d)+1.\]
In general, if $X=X_1\cup X_2\cup\cdots \cup X_n$ ($X_i$: closed), then 
\[ \widim_\varepsilon(X,d)\leq \sum_{i=1}^n\widim_\varepsilon(X_i,d)+n-1.\]
\end{lemma}
\begin{proof}
There exist a finite polyhedron $P_i$ ($i=1,2$) with $\dim P_i=\widim_\varepsilon (X_i,d)$ and 
an $\varepsilon$-embedding $f_i:(X_i,d)\to P_i$.
Let $P_1*P_2 =\{tx\oplus (1-t)y|\, x\in X_1, y\in X_2, 0\leq t\leq 1\}$ be the join of $P_1$ and $P_2$.
($P_1*P_2 = [0,1]\times P_1\times P_2/\sim$, where $(0, x, y)\sim (0,x',y)$ for any $x,x'\in X$
and $(1,x,y)\sim (1,x,y')$ for any $y,y'\in Y$. $tx\oplus(1-t)y$ is the equivalence class of $(t,x,y)$.)
$P_1*P_2$ is a finite polyhedron of dimension $\widim_\varepsilon(X_1,d)+\widim_\varepsilon(X_2,d)+1$.
Since a finite polyhedron is ANR, there exists a open set $U_i\supset X_i$ over which the map $f_i$
continuously extends.
Let $\rho$ be a cut-off function such that $0\leq \rho\leq 1$, $\supp \rho \subset U_1$
and $\rho(x)=1$ if and only if $x\in X_1$.
Then $\supp(1-\rho) = \overline{X\setminus X_1} \subset X_2\subset U_2$.
We define a continuous map $F:X\to P_1*P_2$ by setting $F(x):=\rho(x)f_1(x)\oplus (1-\rho(x))f_2(x)$.
$F$ becomes an $\varepsilon$-embedding; Suppose $F(x)=F(y)$.
If $\rho(x)=\rho(y)=1$, then $x, y\in X_1$ and $f_1(x)=f_1(y)$. Then $d(x,y)\leq \varepsilon$.
If $\rho(x)=\rho(y)<1$, then $x,y\in X_2$ and $f_2(x)=f_2(y)$. Then $d(x,y)\leq \varepsilon$.
Thus $\widim_\varepsilon(X,d)\leq \dim P_1*P_2=\widim_\varepsilon(X_1,d)+\widim_\varepsilon(X_2,d)+1$.
\end{proof}

Let $\Gamma$ be a locally compact Hausdorff unimodular group with a bi-invariant Haar measure $|\cdot|$.
We suppose that $\Gamma$ is endowed with a left-invariant proper distance.
(Properness means that every bounded closed set is compact.)
In Section \ref{Section: Mean dimension and local mean dimension} 
we always assume that $\Gamma$ satisfies these conditions.
When $\Gamma$ is discrete, we always assume that the Haar measure $|\cdot|$ is the counting measure.
(That is, $|\Omega|$ is equal to the cardinality of $\Omega$.)

Let $\Omega \subset \Gamma$ be a subset and $r>0$.
The $r$-boundary $\partial_r\Omega$ is the set of points $\gamma\in \Gamma$ such that 
the closed $r$-ball $B_r(\gamma)$ centered at $\gamma$ 
has non-empty intersection with both $\Omega$ and $\Gamma\setminus \Omega$.
A sequence of bounded Borel sets $\{\Omega_n\}_{n\geq 1}$ in $\Gamma$ is called amenable (or F{\o}lner)
if for any $r>0$ the following is satisfied:
\[ \lim_{n\to \infty}|\partial_r\Omega_n|/|\Omega_n| = 0 .\]
$\Gamma$ is called amenable group if it admits an amenable sequence.
\begin{example} \label{example: integers}
$\Gamma = \mathbb{Z}$ with the counting measure $|\cdot|$ and the standard distance $|x-y|$.
Then the sequence of sets $\{0,1,2,\cdots, n\}$ $(n\geq 1)$ is amenable.
The sequence of sets $\{-n,-n+1, \cdots, -1,0,1,\cdots, n-1, n\}$ $(n\geq 1)$ is also amenable.
\end{example}
\begin{example} \label{example: real line}
$\Gamma = \mathbb{R}$ with the Lebesgue measure $|\cdot|$ and the standard distance $|x-y|$.
In this paper we always assume that $\mathbb{R}$ has these standard measure and distance.
Then the sequence of sets $\{x\in \mathbb{R}|\, 0\leq x \leq n\}$ $(n\geq 1)$ is amenable. 
The sequence of sets $\{x\in \mathbb{R}|\, -n\leq x\leq n\}$ $(n\geq 1)$ is also amenable.
\end{example}
We need the following ``Ornstein-Weiss Lemma'' (\cite[pp. 336-338]{Gromov} and \cite[Appendix]{Lindenstrauss-Weiss}).
\begin{lemma} \label{lemma: Ornstein-Weiss lemma}
Suppose $\Gamma$ is amenable.
Let $h: \{\text{bounded sets in }\Gamma\} \to \mathbb{R}_{\geq 0}$ be a map satisfying the following conditions.

\noindent 
(i) If $\Omega_1\subset \Omega_2$, then $h(\Omega_1)\leq h(\Omega_2)$.

\noindent 
(ii) $h(\Omega_1\cup \Omega_2)\leq h(\Omega_1)+h(\Omega_2)$.

\noindent 
(iii) For any $\gamma\in \Gamma$ and any bounded set $\Omega \subset \Gamma$, 
$h(\gamma\Omega) = h(\Omega)$. Here $\gamma\Omega := \{\gamma x\in \Gamma |\, x\in \Omega\}$.

Then for any amenable sequence $\{\Omega_n\}_{n\geq 1}$ in $\Gamma$, the limit $\lim_{n\to \infty}h(\Omega_n)/|\Omega_n|$
always exists and is independent of the choice of an amenable sequence $\{\Omega_n\}_{n\geq 1}$.
\end{lemma}
Let $(X,d)$ be a compact metric space with a continuous action of $\Gamma$.
We suppose that the action is a right-action.
For a subset $\Omega\subset \Gamma$, we define a new distance $d_{\Omega}(\cdot, \cdot)$ on $X$ by 
\[ d_{\Omega}(x,y) := \sup_{\gamma\in \Omega}d(x.\gamma, y.\gamma) \quad (x,y\in X).\]
\begin{lemma}\label{lemma: widim satisfies ornstein-weiss}
The map $\Omega\mapsto \widim_\varepsilon(X,d_\Omega)$ satisfies the conditions (i), (ii), (iii) 
in Lemma \ref{lemma: Ornstein-Weiss lemma}.
\end{lemma}
\begin{proof}
If $\Omega_1\subset \Omega_2$, then the identity map $(X, d_{\Omega_1})\to (X, d_{\Omega_2})$
is distance non-decreasing. Hence $\widim_\varepsilon(X,d_{\Omega_1})\leq \widim_\varepsilon(X, d_{\Omega_2})$.
The map $(X, d_{\Omega_1\cup\Omega_2})\to (X,d_{\Omega_1})\times (X,d_{\Omega_2})$, $x\to (x,x)$, is distance 
preserving. Hence, by using Lemma \ref{lemma: subadditivity of widim},
$\widim_{\varepsilon}(X,d_{\Omega_1\cup\Omega_2})\leq \widim_\varepsilon(X,d_{\Omega_1})
+\widim_\varepsilon(X,d_{\Omega_2})$.
The map $(X,d_{\gamma\Omega})\to (X,d_{\Omega})$, $x\mapsto x.\gamma$, is an isometry.
Hence $\widim_\varepsilon(X,d_{\gamma\Omega})=\widim_\varepsilon(X,d_\Omega)$.
\end{proof}
Suppose that $\Gamma$ is an amenable group and that an amenable sequence $\{\Omega_n\}_{n\geq 1}$ is given.
For $\varepsilon>0$, we set 
\[ \widim_\varepsilon(X:\Gamma):=\lim_{n\to \infty}\frac{1}{|\Omega_n|}\widim_\varepsilon (X, d_{\Omega_n}).\]
This limit exists and is independent of the choice of an amenable sequence $\{\Omega_n\}_{n\geq 1}$.
The value of $\widim_\varepsilon(X:\Gamma)$ depends on the distance $d$. Hence, strictly speaking, 
we should use the notation $\widim_\varepsilon ((X, d):\Gamma)$.
But we use the above notation for simplicity.
We define $\dim(X:\Gamma)$ (the mean dimension of $(X, \Gamma)$) by 
\[ \dim(X:\Gamma):=\lim_{\varepsilon\to 0}\widim_\varepsilon(X:\Gamma).\]
This becomes a topological invariant, i.e., the value of $\dim(X:\Gamma)$ does not depend on 
the choice of a distance $d$ on $X$ compatible with the topology of $X$.
\begin{example}
Let $\Gamma$ be a finitely generated (discrete) amenable group.
Let $B\subset \mathbb{R}^N$ be the closed ball. 
$\Gamma$ acts on $B^\Gamma$ by the shift.
Then 
\[ \dim(B^\Gamma:\Gamma) = N.\]
For the proof, see Lindenstrauss-Weiss \cite[Proposition 3.1, 3.3]{Lindenstrauss-Weiss} or Tsukamoto
\cite[Example 9.6]{Tsukamoto-2}.
\end{example}

\subsection{Local mean dimension} \label{subsection: local mean dimension}
Let $(X, d)$ be a compact metric space.
The usual topological dimension $\dim X$ is a ``local notion" as follows:
For each point $p\in X$, we define the ``local dimension" $\dim_p X$ at $p$ by 
$\dim_p X := \lim_{r \to 0} \dim B_r (p)$.
(Here $B_r(p)$ is the closed $r$-ball centered at $p$.)
Then we have $\dim X=\sup_{p\in X}\dim_p X$.
The authors don't know whether a similar description of the mean dimension is possible or not. 
Instead, in this subsection we will introduce a new notion ``local mean dimension".

Suppose that an amenable group $\Gamma$ continuously acts on $X$ from the right. 
($(X,d)$ is a compact metric space.)
Let $Y\subset X$ be a closed subset.
Then the map $\Omega\mapsto \sup_{\gamma\in\Gamma}\widim_\varepsilon(Y, d_{\gamma\Omega})$
satisfies the conditions in Lemma \ref{lemma: Ornstein-Weiss lemma}.
Hence we can set
\[ \widim_\varepsilon(Y\subset X:\Gamma) 
:= \lim_{n\to \infty}\left(\frac{1}{|\Omega_n|}\sup_{\gamma\in \Gamma}
\widim_\varepsilon(Y, d_{\gamma\Omega_n})\right), \]
where $\{\Omega_n\}_{n\geq 1}$ is an amenable sequence. We define
\[ \dim(Y\subset X:\Gamma) := \lim_{\varepsilon\to 0}\widim_\varepsilon(Y\subset X:\Gamma).\]
This does not depend on the choice of a distance on $X$ compatible with the topology of $X$.
If $Y_1$ and $Y_2$ are closed subsets in $X$ with $Y_1\subset Y_2$, then 
\[ \dim(Y_1\subset X:\Gamma)\leq \dim(Y_2\subset X:\Gamma).\]
If $Y\subset X$ is a $\Gamma$-invariant closed subset, then 
$\widim_\varepsilon (Y,d_{\gamma\Omega_n}) = \widim_\varepsilon(Y,d_{\Omega_n})$ because
$(Y,d_{\gamma\Omega_n})\to (Y,d_{\Omega_n})$, $x\to x.\gamma$, is an isometry. Hence
\[ \dim(Y\subset X:\Gamma) = \dim(Y:\Gamma),\]
where the right-hand-side is the ordinary mean dimension of $(Y,\Gamma)$.
In particular, $\dim(X\subset X:\Gamma) = \dim(X:\Gamma)$, and hence for any closed subset $Y\subset X$
(not necessarily $\Gamma$-invariant) 
\[ \dim(Y\subset X:\Gamma) \leq \dim(X\subset X:\Gamma) = \dim(X:\Gamma) .\]

Let $X_1$ and $X_2$ be compact metric spaces with continuous $\Gamma$-actions.
Let $Y_1\subset X_1$ and $Y_2\subset X_2$ be closed subsets.
If there exists a $\Gamma$-equivariant topological embedding 
$f:X_1\to X_2$ satisfying $f(Y_1)\subset Y_2$, then 
\begin{equation} \label{eq: monotonicity of relative mean dimension}
\dim(Y_1\subset X_1:\Gamma) \leq \dim(Y_2\subset X_2:\Gamma).
\end{equation}

For each point $p\in X$ and $r>0$ we denote $B_r(p)_\Gamma$ (or $B_r(p;X)_\Gamma$)
as the closed $r$-ball centered at $p$ with respect to the 
distance $d_\Gamma(\cdot,\cdot)$:
\[ B_r(p)_{\Gamma} := \{x\in X|\, d_{\Gamma}(x,p)\leq r\}.\]
Note that $d_\Gamma(x,p)\leq r \Leftrightarrow d(x.\gamma,p.\gamma)\leq r$ for all $\gamma\in \Gamma$.
$B_r(p)_\Gamma$ is a closed set in $X$.
We define the local mean dimension of $X$ at $p$ by 
\[ \dim_p(X:\Gamma) := \lim_{r\to 0}\dim(B_r(p)_\Gamma\subset X:\Gamma).\]
This is independent of the choice of a distance compatible with the topology of $X$.
We define the local mean dimension of $X$ by 
\[ \dim_{loc}(X:\Gamma) :=\sup_{p\in X}\dim_p(X:\Gamma).\]
Obviously we have 
\begin{equation} \label{eq: local mean dim. leq mean dim.}
 \dim_{loc}(X:\Gamma)\leq \dim(X:\Gamma).
\end{equation}
We will use the following formula in Section \ref{subsection: upper bound on the local mean dimension}.
Since $(B_r(p)_{\Gamma}).\gamma = B_r(p.\gamma)_{\Gamma}$, we have 
\[ \widim_\varepsilon(B_r(p)_{\Gamma}, d_{\gamma\Omega})
 =\widim_\varepsilon((B_r(p)_{\Gamma}).\gamma, d_{\Omega})
 = \widim_\varepsilon(B_r(p.\gamma)_{\Gamma}, d_\Omega),\]
and hence
\begin{equation} \label{eq: one description of local mean widim}
\widim_\varepsilon(B_r(p)_{\Gamma}\subset X:\Gamma)
= \lim_{n\to \infty}\left( \frac{1}{|\Omega_n|}
\sup_{\gamma\in \Gamma}\widim_\varepsilon(B_r(p.\gamma)_{\Gamma},d_{\Omega_n})\right).
\end{equation}

Let $X, Y$ be compact metric spaces with continuous $\Gamma$-actions.
If there exists a $\Gamma$-equivariant topological embedding $f:X\to Y$,
then, from (\ref{eq: monotonicity of relative mean dimension}), for all $p\in X$ 
\[ \dim_p(X:\Gamma)\leq \dim_{f(p)}(Y:\Gamma).\]
\begin{example}
Let $B\subset \mathbb{R}^N$ be the closed ball centered at the origin. Then we have 
\[ \dim_{\bm{0}}(B^{\Gamma}:\Gamma) 
  = \dim_{loc}(B^{\Gamma}:\Gamma)  = \dim(B^{\Gamma}:\Gamma) = N,\]
where $\bm{0} = (x_\gamma)_{\gamma\in \Gamma}$ with $x_\gamma=0$ for all $\gamma\in \Gamma$.
\end{example}
\begin{proof}
Fix a distance on $B^{\Gamma}$. 
Then it is easy to see that for any $r>0$ there exists $s>0$
such that $B_s^{\Gamma} \subset B_r(\bm{0})_{\Gamma}$, where 
$B_s$ is the $s$-ball in $\mathbb{R}^N$. Then 
\[ N = \dim(B_s^{\Gamma}:\Gamma) \leq \dim(B_r(\bm{0})_{\Gamma}\subset B^{\Gamma}:\Gamma) 
  \leq \dim(B^{\Gamma}:\Gamma) = N.\]
Hence $\dim_{\bm{0}}(B^{\Gamma}:\Gamma)=N$.
\end{proof}
\begin{remark}
We have so far supposed that $\Gamma$ has a bi-invariant Haar measure and a proper left-invariant distance.
The values of mean dimension and local mean dimension depend on the choice of a Haar measure.
But they are independent of the choice of a proper left-invariant distance on $\Gamma$.
(We need the \textit{existence} of a proper left-invariant distance on $\Gamma$ for defining the notion ``amenable sequence''.
But this notion is independent of the choice of a proper left-invariant distance on $\Gamma$.)
\end{remark}

\subsection{The case of $\Gamma = \mathbb{R}$}
Let $\Gamma = \mathbb{R}$ with the Lebesgue measure and the standard distance.
Suppose that $\mathbb{R}$ continuously acts on a compact metric space $(X,d)$.
For $T>0$, consider the discrete subgroup $T\mathbb{Z} := \{Tn\in \mathbb{R}|\, n\in \mathbb{Z}\}$ in $\mathbb{R}$.
$T\mathbb{Z}$ also acts on $X$.
We want to compare the mean dimensions of $(X,\mathbb{R})$ and $(X,T\mathbb{Z})$.
Here $T\mathbb{Z}$ is equipped with the counting measure.
\begin{proposition} \label{prop: mean dimension of flow and shift}
\[ \dim(X:T\mathbb{Z}) = T\dim(X:\mathbb{R}).\]
This result is given in \cite[p. 329]{Gromov} and \cite[Proposition 2.7]{Lindenstrauss-Weiss}.
For any point $p\in X$,
\[ \dim_{p}(X:T\mathbb{Z}) = T\dim_{p}(X:\mathbb{R}).\]
In particular, $\dim_{loc}(X:T\mathbb{Z}) = T\dim_{loc}(X:\mathbb{R})$.
\end{proposition}
\begin{proof}
Set $\Omega_n := \{\gamma\in \mathbb{R}|\, 0\leq \gamma < Tn\}$ and $\Omega_n':=\Omega_n\cap T\mathbb{Z}$.
$\{\Omega_n\}_{n\geq 1}$ is an amenable sequence for $\mathbb{R}$, and $\{\Omega'_n\}_{n\geq 1}$ is 
an amenable sequence for $T\mathbb{Z}$.
Let $Y\subset X$ be a closed subset.
For $\gamma\in T\mathbb{Z}$,
$d_{\gamma+\Omega_n'}(\cdot,\cdot)\leq d_{\gamma+\Omega_n}(\cdot,\cdot)$.
Hence, for any $\varepsilon>0$, 
$\widim_\varepsilon (Y,d_{\gamma + \Omega'_n})\leq \widim_\varepsilon(Y,d_{\gamma+\Omega_n})$.
Therefore 
\[ \dim(Y\subset X: T\mathbb{Z})\leq T\dim(Y\subset X:\mathbb{R}).\]

For any $\varepsilon>0$ there exists $\delta>0$ such that 
if $d(x,y)\leq \delta$ then $d_{[0,2T)}(x,y)\leq \varepsilon$.
Let $a\in \mathbb{R}$ and set $k:=[a]$ (the maximum integer $\leq a$).
If $d_{kT+\Omega'_n}(x,y)\leq \delta$, then $d_{aT+\Omega_n}(x,y)\leq \varepsilon$.
Hence $\widim_\varepsilon(Y,d_{aT+\Omega_n})\leq \widim_\delta(Y,d_{kT+\Omega_n'})$.
This implies 
\[ \sup_{\gamma\in \mathbb{R}}\widim_\varepsilon(Y,d_{\gamma+\Omega_n})
   \leq \sup_{\gamma\in T\mathbb{Z}} \widim_\delta(Y,d_{\gamma+\Omega'_n}).\]
Therefore $T\dim(Y\subset X: \mathbb{R})\leq \dim(Y\subset X: T\mathbb{Z})$.
Thus 
\begin{equation} \label{eq: relative mean dimension of flow and shift}
 T\dim(Y\subset X:\mathbb{R}) = \dim(Y\subset X: T\mathbb{Z}).
\end{equation}
In particular, if $Y=X$, then $\dim(X: T\mathbb{Z}) = T\dim(X:\mathbb{R})$.

For any $r>0$ there exists $r'>0$ such that if $d(x,y)\leq r'$ then $d_{[0,T)}(x,y)\leq r$.
Then if $d_{T\mathbb{Z}}(x,y)\leq r'$, we have $d_{\mathbb{R}}(x,y)\leq r$.
Hence $B_{r'}(p)_{T\mathbb{Z}}\subset B_r(p)_{\mathbb{R}}\subset B_r(p)_{T\mathbb{Z}}$.
Therefore, by using the above (\ref{eq: relative mean dimension of flow and shift}),
\begin{equation*}
 \begin{split}
 \dim(B_{r'}(p)_{T\mathbb{Z}}\subset X:T\mathbb{Z}) &= T\dim(B_{r'}(p)_{T\mathbb{Z}}\subset X:\mathbb{R})
 \leq T\dim(B_r(p)_{\mathbb{R}}\subset X:\mathbb{R}) \\
 &\leq T\dim(B_r(p)_{T\mathbb{Z}}\subset X:\mathbb{R}) = \dim(B_r(p)_{T\mathbb{Z}}\subset X:T\mathbb{Z}).
 \end{split}
\end{equation*}
Thus $\dim_p(X:T\mathbb{Z}) = T\dim_p(X:\mathbb{R})$.
\end{proof}

\section{Outline of the proofs of the main theorems} \label{section: outline of the proofs of the main theorems}
The ideas of the proofs of Theorem \ref{thm: main theorem} and \ref{thm: main theorem on the local mean 
dimension} are simple.
But the completion of the proofs needs lengthy technical arguments.
So we want to describe the outline of the proofs in this section.
Here we don't pursue the accuracy of the arguments for simplicity of the explanation.
Some of the arguments will be replaced with different ones in the later sections. 

First we explain how to get the upper bound on the mean dimension of $\moduli_d$.
We define a distance on $\moduli_d$ by setting 
\begin{equation*}
 \dist([\bm{A}], [\bm{B}]) 
 := \inf_{g:\bm{E}\to \bm{E}} 
 \left\{\sum_{n\geq 1} 2^{-n}\frac{\norm{g(\bm{A})-\bm{B}}_{L^\infty(|t|\leq n)}}
 {1+ \norm{g(\bm{A})-\bm{B}}_{L^\infty(|t|\leq n)}}\right\},
\end{equation*}
where $g$ runs over all gauge transformations of $\bm{E}$, and $|t|\leq n$ means the region 
$\{(\theta, t)\in S^3\times \mathbb{R}|\, |t|\leq n\}$.
For $R = 1,2,3,\cdots$, we define $\Omega_R \subset \mathbb{R}$ by
$\Omega_R :=\{s \in \mathbb{R}| -R \leq s\leq R \}$.
$\{\Omega_R\}_{R\geq 1}$ is an amenable sequence in $\mathbb{R}$.

Let $\varepsilon >0$ be a positive number, and define a positive integer 
$L=L(\varepsilon)$ so that 
\begin{equation} \label{eq: definition of L in the outline}
 \sum_{n>L}2^{-n} < \varepsilon/2.
\end{equation}
Let $D= D(\varepsilon)$ be a large positive number which depends on $\varepsilon$ but is 
independent of $R$, and set $T := R+L+D$. 
($D$ is chosen so that the condition (\ref{eq: conditions of A'' in the outline}) below 
is satisfied.
Here we don't explain how to define $D$ precisely.)

For $c\geq 0$ we define $M(c)$ as the space of the gauge equivalence classes $[A]$ where 
$A$ is an ASD connection on $\bm{E}$ satisfying 
\[ \frac{1}{8\pi^2}\int_{X}|F_A|^2d\vol \leq c.  \]  
The index theorem gives the estimate:
\[ \dim M(c) \leq 8c.\]
We want to construct an $\varepsilon$-embedding from $(\moduli_d, \dist_{\Omega_R})$ to $M(c)$
for an appropriate $c\geq 0$.

Let $\bm{A}$ be an ASD connection on $\bm{E}$ with $[\bm{A}]\in \moduli_d$.
We ``cut-off'' $\bm{A}$ over the region $T < |t| <T+1$
and construct a new connection $\bm{A}'$ satisfying the following conditions.
$\bm{A}'$ is a (not necessarily ASD) connection on $\bm{E}$ satisfying $\bm{A}'|_{|t|\leq T}=\bm{A}|_{|t|\leq T}$, 
$F(\bm{A}') =0$ over $|t|\geq T+1$, and 
\[ \frac{1}{8\pi^2}\int_{X} tr(F(\bm{A}')^2) \leq \frac{1}{8\pi^2} \int_{|t|\leq T}|F(\bm{A})|^2d\vol + \const
  \leq \frac{2Td^2\vol(S^3)}{8\pi^2} + \const,\]
where $\const$ is a positive constant independent of $\varepsilon$ and $R$.
Next we ``perturb'' $\bm{A}'$ and construct an ASD connection $\bm{A}''$ on $\bm{E}$ satisfying
\begin{equation}\label{eq: conditions of A'' in the outline}
 |\bm{A}-\bm{A}''| = |\bm{A}'-\bm{A}''| \leq \varepsilon/4 \quad (|t|\leq T-D = R+L),
\end{equation}
\[ \frac{1}{8\pi^2}\int_X |F(\bm{A}'')|^2 d\vol = \frac{1}{8\pi^2}\int_X tr(F(\bm{A}')^2)
 \leq \frac{2Td^2\vol(S^3)}{8\pi^2} + \const.\]
Then we can define the map 
\[ \moduli_d\to M\left(\frac{2Td^2\vol(S^3)}{8\pi^2} + \const\right), 
\quad [\bm{A}]\mapsto [\bm{A}''].\]
The conditions (\ref{eq: definition of L in the outline}) and (\ref{eq: conditions of A'' in the outline})
imply that this map is an $\varepsilon$-embedding with respect to the 
distance $\dist_{\Omega_R}$. Hence 
\[ \widim_\varepsilon(\moduli_d, \dist_{\Omega_R}) 
\leq \dim M\left(\frac{2Td^2\vol(S^3)}{8\pi^2} + \const\right)
 \leq \frac{2Td^2\vol(S^3)}{\pi^2} + 8\cdot \const.\]
(Caution! This estimate will \textit{not} be proved in this paper. The above argument contains a gap.)
Recall $T=R+L+D$. Since $L$, $D$ and $\const$ are independent of $R$, we get 
\[ \widim_\varepsilon(\moduli_d:\mathbb{R}) = \lim_{R \to \infty}\frac{\widim_{\varepsilon}(\moduli_d,\dist_{\Omega_R})}{2R} 
\leq \frac{d^2\vol(S^3)}{\pi^2}.\] 
Hence we get 
\begin{equation} \label{eq: upper bound on the mean dimension in the outline}
 \dim(\moduli_d:\mathbb{R}) \leq \frac{d^2\vol(S^3)}{\pi^2} < +\infty.
\end{equation}

This is the outline of the proof of the upper bound on the mean dimension.
(The upper bound on the local mean dimension can be proved by investigating the above 
procedure more precisely.)
Strictly speaking, the above argument contains a gap.
Actually we have not so far succeeded to prove the estimate 
$\dim(\moduli_d:\mathbb{R}) \leq d^2\vol(S^3)/\pi^2$.
In this paper we prove only $\dim(\moduli_d:\mathbb{R}) < +\infty$.
A problem occurs in the cut-off construction.
Indeed (we think that) there exists no canonical way to cut-off connections compatible with the gauge symmetry.
Therefore we cannot define a suitable cut-off construction all over $\moduli_d$.
Instead we will decompose $\moduli_d$ as $\moduli_d = \bigcup_{0\leq i,j\leq N}\moduli_{d,T}(i,j)$ 
($N$ is independent of $\varepsilon$ and $R$) and 
define a cut-off construction for each piece $\moduli_{d,T}(i,j)$ independently.
Then we will get an upper bound worse than (\ref{eq: upper bound on the mean dimension in the outline})
(cf. Lemma \ref{lemma: widim and space sum}).
We study the cut-off construction (the procedure $[\bm{A}]\mapsto [\bm{A}']$) 
in Section \ref{section: cut-off constructions}.
In Section \ref{section: solving ASD equation} and \ref{section: continuity of the perturbation}
we study the perturbation procedure ($\bm{A}' \mapsto \bm{A}''$).
The upper bounds on the (local) mean dimension are proved in Section \ref{Section: proof of the upper bound}.

Next we explain how to prove the lower bound on the local mean dimension.
Let $T>0$, $\underbar{E}$ be a principal $SU(2)$-bundle over $S^3\times (\mathbb{R}/T\mathbb{Z})$,
and $\underbar{A}$ be a non-flat ASD connection on $\underbar{E}$ satisfying 
$|F(\underbar{A})| < d$.

Let $\pi:S^3\times \mathbb{R}\to S^3\times (\mathbb{R}/T\mathbb{Z})$ be the natural projection, and set 
$E:=\pi^*(\underbar{E})$ and $A:=\pi^*(\underbar{A})$.
We define the infinite dimensional Banach space $H^1_A$ by 
\[ H^1_A :=\{a\in \Omega^1(\ad E)|\, (d_A^*+d_A^+)a=0,\, \norm{a}_{L^\infty} <\infty\}.\]
There exists a natural $T\mathbb{Z}$-action on $H^1_A$.
Let $r>0$ be a sufficiently small number.
For each $a\in H^1_A$ with $\norm{a}_{L^\infty}\leq r$ we can construct 
$\tilde{a} \in \Omega^1(\ad E)$ (a small perturbation of $a$) satisfying $F^+(A+\tilde{a})=0$
and $|F(A+\tilde{a})|\leq d$.
If $a=0$, then $\tilde{a}=0$.

For $n\geq 1$, let $\pi_n:S^3\times(\mathbb{R}/nT\mathbb{Z})\to S^3\times (\mathbb{R}/T\mathbb{Z})$
be the natural projection, and set $E_n:=\pi_n^*(\underbar{E})$ and $A_n := \pi_n^*(\underbar{A})$.
We define $H^1_{A_n}$ as the space of $a \in \Omega(\ad E_n)$ satisfying $(d^*_{A_n}+d^+_{A_n})a=0$.
We can identify $H^1_{A_n}$ with the subspace of $H^1_A$ consisting of $nT\mathbb{Z}$-invariant elements.
The index theorem gives 
\[ \dim H^1_{A_n} = 8nc_2(\underbar{E}).\]

We define the map from $B_r(H^1_A)$ (the $r$-ball of $H^1_A$ centered at the origin) to $\moduli_d$ by
\[ B_r(H^1_A) \to \moduli_d, \quad 
  a \mapsto [E, A+\tilde{a}].\]
(cf. the description of $\moduli_d$ in Remark \ref{remark: another description of moduli_d}.)
This map becomes a $T\mathbb{Z}$-equivariant topological embedding for $r\ll 1$.
(Here $B_r(H^1_A)$ is endowed with the following topology.
A sequence $\{a_n\}_{n\geq 1}$ in $B_r(H^1_A)$ converges to $a$
in $B_r(H^1_A)$ if and only if $a_n$ uniformly converges to $a$ over every compact subset.)
Then we have
\[ \dim_{[E,A]}(\moduli_d:T\mathbb{Z})\geq \dim_0(B_r(V):T\mathbb{Z}).\]
The right-hand-side is the local mean dimension of $B_r(H^1_A)$ at the origin.
We can prove that $\dim_0(B_r(H^1_A):T\mathbb{Z})$ can be 
estimated from below by ``the growth of periodic points":
\[ \dim_0(B_r(H^1_A):T\mathbb{Z})\geq \lim_{n\to \infty}\dim H^1_{A_n}/n = 8c_2(\underbar{E}).\]
(This is not difficult to prove. This is just an application of Lemma \ref{lemma: widim of Banach ball}.)
Therefore 
\[ \dim_{[E,A]}(\moduli_d:\mathbb{R}) = \dim_{[E,A]}(\moduli_d:T\mathbb{Z})/T \geq 8c_2(\underbar{E})/T = 8\rho(A).\]
This is the outline of the proof of the lower bound.

\section{Perturbation} \label{section: solving ASD equation}
In this section we construct the method of constructing ASD connections from 
``approximately ASD" connections over $X = S^3\times \mathbb{R}$.
We basically follow the argument of Donaldson \cite{Donaldson}.
As we promised in the introduction, the variable $t$ means the variable of the 
$\mathbb{R}$-factor of $S^3\times \mathbb{R}$.
\subsection{Construction of the perturbation} \label{subsection: construction}
Let $T$ be a positive number, and $d, d'$ be two non-negative real numbers.
Set $\varepsilon_0 = 1/(1000)$. (The value $1/(1000)$ itself has no meaning.
The point is that it is an explicit number which satisfies 
(\ref{eq: fixing condition of varepsilon_0}) below.)
Let $E$ be a principal $SU(2)$-bundle over $X$, and $A$ be a connection on $E$ satisfies
the following conditions (i), (ii), (iii).

\noindent
(i) $F_A = 0$ over $|t|>T+1$.

\noindent
(ii) $F_A^+$ is supported in $\{(\theta,t)\in S^3\times \mathbb{R}|\, T<|t|<T+1\}$, and 
$\norm{F_A^+}_{\mathrm{T}}\leq \varepsilon_0$.
Here $\norm{\cdot}_{\mathrm{T}}$ is the ``Taubes norm'' defined below 
((\ref{eq: def. of pointwise Taubes norm}) and (\ref{eq: def. of Taubes norm})).
(``$\mathrm{T}$'' of the norm $\norm{\cdot}_{\mathrm{T}}$ comes from ``Taubes'', and it has no 
relation with the above positive number $T$.)

\noindent
(iii) $|F_A|\leq d$ on $|t|\leq T$ and $\norm{F^+_A}_{L^\infty(X)}\leq d'$.
(The condition (iii) is not used in Section \ref{subsection: construction}, 
\ref{subsection: regularity and the behavior at the end}, \ref{subsection: conclusion of the construction}. 
It will be used in Section \ref{subsection: interior estimate}.)

Let $\Omega^+(\ad E)$ be the set of smooth self-dual 2-forms valued in $\ad E$
(not necessarily compact supported).
The first main purpose of this section is to solve the equation $F^+(A+d_A^*\phi) = 0$ for $\phi\in \Omega^+(\ad E)$.
We have $F^+(A+d_A^*\phi) = F^+_A + d_A^+d_A^*\phi + (d_A^*\phi\wedge d_A^*\phi)^+$.
The Weitzenb\"{o}ck formula gives (\cite[Chapter 6]{Freed-Uhlenbeck})
\begin{equation}\label{eq: Weitzenbock formula}
 d_A^+d_A^*\phi = \frac{1}{2}\nabla_A^*\nabla_A \phi + \left(\frac{S}{6}- W^+\right)\phi + F_A^+\cdot\phi ,
\end{equation}
where $S$ is the scalar curvature of $X$ and $W^+$ is the self-dual part of the Weyl curvature.
Since $X$ is conformally flat, we have $W^+=0$.
The scalar curvature $S$ is a positive constant.
Then the equation $F^+(A+d_A^*\phi) =0$ becomes 
\begin{equation}\label{eq: ASD equation}
 (\nabla^*_A \nabla_A +S/3)\phi + 2F^+_A\cdot\phi + 2(d_A^*\phi \wedge d_A^*\phi)^+ = -2F_A^+.
\end{equation}
Set $c_0 = 10$. Then
\begin{equation}\label{eq: universal constant c_0}
 |F^+_A\cdot\phi|\leq c_0|F_A^+|\cdot |\phi|, \quad 
 |(d_A^*\phi_1 \wedge d_A^*\phi_2)^+|\leq c_0|\nabla_A\phi_1|\cdot |\nabla_A\phi_2|.
\end{equation}
(These are not best possible.\footnote{Strictly speaking, the choice of $c_0$ depends on the 
convention of the metric (inner product) on $su(2)$.
Our convention is: $\langle A, B\rangle = -tr(AB)$ for $A, B\in su(2)$.})
The positive constant $\varepsilon_0 = 1/1000$ in the above satisfies
\begin{equation} \label{eq: fixing condition of varepsilon_0}
 50 c_0 \varepsilon_0 <1.
\end{equation}

Let $\Delta = \nabla^*\nabla$ be the Laplacian on functions over $X$, and 
$g(x, y)$ be the Green kernel of $\Delta + S/3$. 
We prepare basic facts on $g(x,y)$ in Appendix \ref{appendix: Green kernel}.
Here we state some of them without the proofs.  
For the proofs, see Appendix \ref{appendix: Green kernel}.
$g(x,y)$ satisfies
\[ (\Delta_y + S/3)g(x,y) = \delta_x(y).\]
This equation means that, for any compact supported smooth function $\varphi$,
\[ \varphi(x) = \int_X g(x,y)(\Delta_y+S/3)\varphi(y) d\vol(y),\]
where $d\vol(y)$ denotes the volume form of $X$.
$g(x,y)$ is smooth outside the diagonal and it has a singularity of order $1/d(x,y)^2$ along the 
diagonal:
\begin{equation} \label{eq: singulairy of g(x,y) along the diagonal}
 \const_1/d(x,y)^2 \leq g(x,y) \leq \const_2/d(x,y)^2, \quad (d(x,y) \leq \const_3) ,
\end{equation}
where $d(x,y)$ is the distance on $X$, and $\const_1$, $\const_2$, $\const_3$ are positive constants.
$g(x,y) >0$ for $x\neq y$ (Lemma \ref{lemma: positivity of the Green function}), 
and it has an exponential decay (Lemma \ref{lemma: exponential decay of the Green function}):
\begin{equation} \label{eq: exponential decay of g(x,y)}
 0< g(x,y) < \const_4\cdot e^{-\sqrt{S/3}d(x,y)} \quad (d(x,y)\geq 1).
\end{equation}
Since $S^3\times \mathbb{R} = SU(2)\times \mathbb{R}$ is a Lie group and its Riemannian 
metric is two-sided invariant, we have $g(zx,zy) = g(xz,yz) = g(x,y)$.
In particular, for $x= (\theta_1, t_1)$ and $y = (\theta_2, t_2)$, we have 
$g((\theta_1,t_1-t_0),(\theta_2,t_2-t_0)) = g((\theta_1,t_1),(\theta_2,t_2))$ $(t_0\in \mathbb{R})$.
That is, $g(x,y)$ is invariant under the translation $t\mapsto t-t_0$.

For $\phi\in \Omega^+(\ad E)$, we define the pointwise Taubes norm $|\phi|_{\mathrm{T}}(x)$ by setting 
\begin{equation} \label{eq: def. of pointwise Taubes norm}
 |\phi|_{\mathrm{T}}(x):= \int_X g(x, y)|\phi(y)|d\vol(y)   \quad (x\in X) ,
\end{equation}
(Recall $g(x,y)>0$ for $x\neq y$.) This may be infinity.
We define the Taubes norm $\norm{\phi}_{\mathrm{T}}$ by 
\begin{equation}\label{eq: def. of Taubes norm}
 \norm{\phi}_{\mathrm{T}} := \sup_{x\in X}|\phi|_{\mathrm{T}}(x) .
\end{equation}
Set 
\[ K := \int_X g(x, y)d\vol(y)  \quad \text{(this is independent of $x \in X$)}.\]
(This is finite by (\ref{eq: singulairy of g(x,y) along the diagonal})
 and (\ref{eq: exponential decay of g(x,y)}).)
We have 
\[ \norm{\phi}_{\mathrm{T}} \leq K\norm{\phi}_{L^\infty}.\]

We define $\Omega^+(\ad E)_0$ as the set of $\phi\in \Omega^+(\ad E)$
which vanish at infinity: $\lim_{x\to \infty}|\phi(x)|=0$.
(Here $x = (\theta, t)\to \infty$ means $|t|\to +\infty$.)
If $\phi\in\Omega^+(\ad E)_0$, then $\norm{\phi}_{\mathrm{T}}<\infty$ and 
$\lim_{x\to \infty}|\phi|_{\mathrm{T}}(x) =0$. 
(See the proof of Proposition \ref{prop: solve the equation in a general case}.)

Let $\eta\in \Omega^+(\ad E)_0$.
There uniquely exists $\phi\in \Omega^+(\ad E)_0$ satisfying
$(\nabla_A^*\nabla_A+S/3)\phi=\eta$. (See Proposition \ref{prop: solve the equation in a general case}.)
We set $(\nabla_A^*\nabla_A+S/3)^{-1}\eta :=\phi$.
This satisfies
\begin{equation}\label{eq: phi leq eta}
 |\phi(x)|\leq |\eta|_{\mathrm{T}}(x),\, 
 \text{and hence $\norm{\phi}_{L^\infty}\leq \norm{\eta}_{\mathrm{T}}$}.
\end{equation}
\begin{lemma}\label{lemma: key lemma for the construction}
$\lim_{x\to \infty}|\nabla_A\phi (x)| = 0$.
\end{lemma}
\begin{proof}
From the condition (i) in the beginning of this section, $A$ is flat over $|t|>T+1$.
Therefore there exists a bundle map $g:E|_{|t|>T+1}\to X_{|t|>T+1}\times SU(2)$
such that $g(A)$ is the product connection.
Here $X_{|t|>T+1} = \{(\theta, t)\in S^3\times \mathbb{R}|\, |t|>T+1\}$ and $E|_{|t|>T+1}$
is the restriction of $E$ to $X_{|t|>T+1}$.
We sometimes use similar notations in this paper.
Set $\phi' := g(\phi)$ and $\eta' := g(\eta)$. They satisfy 
$(\nabla^*\nabla+S/3)\phi' = \eta'$. 
(Here $\nabla$ is defined by the product connection on $X|_{|t|>T+1}\times SU(2)$ and 
the Levi-Civita connection.)

For $|t|>T+2$, we set $B_t := S^3\times (t-1, t+1)$.
From the elliptic estimates, for any $\theta \in S^3$,
\[ |\nabla \phi'(\theta, t)| \leq C(\norm{\phi'}_{L^\infty(B_t)} + \norm{\eta'}_{L^\infty(B_t)}),\]
where $C$ is a constant independent of $t$.
This means 
\[ |\nabla_A\phi(\theta, t)|\leq C(\norm{\phi}_{L^\infty(B_t)} + \norm{\eta}_{L^\infty(B_t)}).\]
The right-hand-side goes to $0$ as $|t|$ goes to infinity.
\end{proof}
The following Lemma shows a power of the Taubes norm.
(Here $\eta\in \Omega^+(\ad E)_0$ and $\phi = (\nabla_A^*\nabla_A+S/3)^{-1}\eta\in \Omega^+(\ad E)_0$.)
\begin{lemma} \label{lemma: magical lemma of the Taubes norm}
\[ \left| |\nabla_A\phi|^2\right|_{\mathrm{T}}(x) := \int_X g(x,y)|\nabla_A\phi(y)|^2d\vol(y)
\leq \norm{\eta}_{\mathrm{T}} |\eta|_{\mathrm{T}}(x) .\]
In particular, 
$\norm{|\nabla_A\phi|^2}_{\mathrm{T}} 
:= \sup_{x\in X}\left| |\nabla_A\phi|^2\right|_{\mathrm{T}}(x)\leq \norm{\eta}_{\mathrm{T}}^2$ 
and $\norm{(d_A^*\phi\wedge d^*_A\phi)^+}_{\mathrm{T}}\leq c_0 \norm{\eta}_{\mathrm{T}}^2$.
\end{lemma}
\begin{proof}
$\nabla |\phi|^2 = 2(\nabla_A\phi,\phi)$ vanishes at infinity 
(Lemma \ref{lemma: key lemma for the construction}).
\[ (\Delta + 2S/3)|\phi|^2 = 2(\nabla_A^*\nabla_A\phi + (S/3)\phi, \phi)-2|\nabla_A\phi|^2 
   = 2(\eta, \phi)-2|\nabla_A\phi|^2.\]
In particular, $(\Delta +S/3)|\phi|^2$ vanishes at infinity (Lemma \ref{lemma: key lemma for the construction}).
Hence $|\phi|^2, \nabla|\phi|^2, (\Delta +S/3)|\phi|^2$ vanish at infinity (in particular, 
they are contained in $L^\infty$). 
Then we can apply Lemma \ref{lemma: Green kernel representation} in Appendix \ref{appendix: Green kernel}
to $|\phi|^2$ and get 
\[ \int_Xg(x,y)(\Delta_y+S/3)|\phi(y)|^2 d\vol(y) = |\phi(x)|^2.\]
We have
\begin{equation*}
 \begin{split}
 |\nabla_A\phi|^2 &= (\eta, \phi) -\frac{1}{2}(\Delta+S/3)|\phi|^2 -\frac{S}{6}|\phi|^2 ,\\
              &\leq (\eta, \phi) -\frac{1}{2}(\Delta+S/3)|\phi|^2. 
 \end{split}
\end{equation*}
Therefore
\begin{equation*}
 \begin{split}
 \int_X g(x,y)|\nabla_A\phi(y)|^2 d\vol(y) &\leq \int_X g(x, y)(\eta(y), \phi(y))d\vol(y) 
 - \frac{1}{2}|\phi(x)|^2,\\ 
 &\leq \int_X g(x, y)(\eta(y), \phi(y))d\vol(y) ,\\
 &\leq \norm{\phi}_{L^\infty}\int_X g(x, y)|\eta(y)| d\vol(y) 
 \leq \norm{\eta}_{\mathrm{T}} |\eta|_{\mathrm{T}}(x). 
 \end{split}
\end{equation*}
In the last line we have used (\ref{eq: phi leq eta}).
\end{proof}
For $\eta_1, \eta_2\in \Omega^+(\ad E)_0$, set 
$\phi_i := (\nabla_A^*\nabla_A+S/3)^{-1}\eta_i \in \Omega^+(\ad E)_0$ $(i=1,2)$ and 
\begin{equation} \label{eq: definition of beta}
 \beta(\eta_1, \eta_2) := (d_A^*\phi_1\wedge d_A^*\phi_2)^+ + (d_A^*\phi_2\wedge d^*_A\phi_1)^+.
\end{equation}
$\beta$ is symmetric and 
$|\beta(\eta_1, \eta_2)| \leq 2c_0|\nabla_A\phi_1|\cdot |\nabla_A\phi_2|$.
In particular, $\beta(\eta_1, \eta_2)\in \Omega^+(\ad E)_0$ (Lemma \ref{lemma: key lemma for the construction}).
\begin{lemma} \label{lemma: estimate of beta}
$\norm{\beta(\eta_1, \eta_2)}_{\mathrm{T}}\leq 4c_0 \norm{\eta_1}_{\mathrm{T}}\norm{\eta_2}_{\mathrm{T}}$.
\end{lemma}
\begin{proof}
From Lemma \ref{lemma: magical lemma of the Taubes norm}, 
$\norm{\beta(\eta, \eta)}_{\mathrm{T}}\leq 2c_0\norm{\eta}_{\mathrm{T}}^2$.
Suppose $\norm{\eta_1}_{\mathrm{T}} = \norm{\eta_2}_{\mathrm{T}} =1$.
Since $4\beta(\eta_1, \eta_2) = \beta(\eta_1+\eta_2, \eta_1+\eta_2) -\beta(\eta_1-\eta_2, \eta_1-\eta_2)$,
\[ 4\norm{\beta(\eta_1, \eta_2)}_{\mathrm{T}}
 \leq 2c_0\norm{\eta_1+\eta_2}_{\mathrm{T}}^2 + 2c_0\norm{\eta_1-\eta_2}_{\mathrm{T}}^2\leq 16c_0.\]
Hence $\norm{\beta(\eta_1, \eta_2)}_{\mathrm{T}}\leq 4c_0$. The general case follows from this.
\end{proof}
For $\eta\in \Omega^+(\ad E)_0$, we set 
$\phi := (\nabla_A^*\nabla_A+S/3)^{-1}\eta \in \Omega^+(\ad E)_0$ and define 
\[ \Phi(\eta) := -2F^+_A\cdot \phi - \beta(\eta, \eta) -2F_A^+ \in \Omega^+(\ad E)_0 .\]
If $\eta$ satisfies $\eta = \Phi(\eta)$, then $\phi$ satisfies the ASD equation (\ref{eq: ASD equation}).
\begin{lemma}\label{lemma: contraction}
For $\eta_1, \eta_2\in \Omega^+(\ad E)_0$, 
\[ \norm{\Phi(\eta_1)-\Phi(\eta_2)}_{\mathrm{T}} 
 \leq  2c_0(\norm{F_A^+}_{\mathrm{T}}+2\norm{\eta_1+\eta_2}_{\mathrm{T}})\norm{\eta_1-\eta_2}_{\mathrm{T}}.  \]
\end{lemma}
\begin{proof}
\[ \Phi(\eta_1)-\Phi(\eta_2) = -2F_A^+\cdot (\phi_1-\phi_2) + \beta(\eta_1+\eta_2, \eta_2-\eta_1).\]
From Lemma \ref{lemma: estimate of beta} and 
$\norm{\phi_1-\phi_2}_{L^\infty}\leq \norm{\eta_1-\eta_2}_{\mathrm{T}}$ (see (\ref{eq: phi leq eta})),
\begin{equation*}
 \begin{split}
  \norm{\Phi(\eta_1)-\Phi(\eta_2)}_{\mathrm{T}} &\leq 2c_0\norm{F_A^+}_{\mathrm{T}}\norm{\phi_1-\phi_2}_{L^\infty}
                                +4c_0\norm{\eta_1+\eta_2}_{\mathrm{T}}\norm{\eta_1-\eta_2}_{\mathrm{T}} ,\\
  &\leq 2c_0(\norm{F_A^+}_{\mathrm{T}}+2\norm{\eta_1+\eta_2}_{\mathrm{T}})\norm{\eta_1-\eta_2}_{\mathrm{T}}.
 \end{split}
\end{equation*}
\end{proof}
\begin{proposition}  \label{prop: eta_n becomes Cauchy}
The sequence $\{\eta_n\}_{n\geq 0}$ in $\Omega^+(\ad E)_0$ defined by 
\[ \eta_0 =0, \quad \eta_{n+1} = \Phi(\eta_n), \]
becomes a Cauchy sequence with respect to the Taubes norm $\norm{\cdot}_{\mathrm{T}}$ and satisfies
\[ \norm{\eta_n}_{\mathrm{T}} \leq 3\varepsilon_0,\]
for all $n\geq 0$.
\end{proposition}
\begin{proof}
Set $B:=\{\eta\in \Omega^+(\ad E)_0| \norm{\eta}_{\mathrm{T}}\leq 3\varepsilon_0\}$.
For $\eta\in B$ (recall: $\norm{F_A^+}_{\mathrm{T}}\leq \varepsilon_0$),
\begin{equation*}
 \begin{split}
 \norm{\Phi(\eta)}_{\mathrm{T}}&
 \leq 2c_0\norm{F_A^+}_{\mathrm{T}}\norm{\phi}_{L^\infty} 
 + 2c_0\norm{\eta}_{\mathrm{T}}^2 + 2\norm{F_A^+}_{\mathrm{T}},\\
 &\leq 2c_0\varepsilon_0 \norm{\eta}_{\mathrm{T}} + 2c_0\norm{\eta}_{\mathrm{T}}^2 + 2\varepsilon_0 ,\\
 &\leq (24c_0\varepsilon_0 + 2)\varepsilon_0 \leq 3\varepsilon_0.
 \end{split}
\end{equation*}
Here we have used (\ref{eq: fixing condition of varepsilon_0}).
Hence $\Phi(\eta)\in B$.
Lemma \ref{lemma: contraction} implies (for $\eta_1, \eta_2\in B$)
\[ \norm{\Phi(\eta_1)-\Phi(\eta_2)}_{\mathrm{T}}
 \leq 2c_0(\norm{F_A^+}_{\mathrm{T}} + 2\norm{\eta_1+\eta_2}_{\mathrm{T}})\norm{\eta_1-\eta_2}_{\mathrm{T}}
                                 \leq 26c_0\varepsilon_0 \norm{\eta_1-\eta_2}_{\mathrm{T}}.\] 
$26c_0\varepsilon_0<1$ by (\ref{eq: fixing condition of varepsilon_0}). 
Hence $\Phi:B\to B$ becomes a contraction map with respect to the norm $\norm{\cdot}_{\mathrm{T}}$.
Thus $\eta_{n+1} = \Phi(\eta_n)$ $(\eta_0 =0)$ becomes a Cauchy sequence.
\end{proof}
The sequence $\phi_n \in \Omega^+(\ad E)_0$ $(n\geq 0)$ defined by 
$\phi_n := (\nabla_A^*\nabla_A+S/3)^{-1}\eta_n$
satisfies $\norm{\phi_n-\phi_m}_{L^\infty}\leq \norm{\eta_n-\eta_m}_{\mathrm{T}}$.
Hence it becomes a Cauchy sequence in $L^\infty(\Lambda^+(\ad E))$.
Therefore $\phi_n$ converges to some $\phi_A$ in $L^\infty(\Lambda^+(\ad E))$.
$\phi_A$ is continuous since every $\phi_n$ is continuous. 
Indeed we will see later that $\phi_A$ is smooth and satisfies the ASD equation 
$F^+(A+d_A^*\phi_A) =0$.

We have $\eta_{n+1} = \Phi(\eta_n) = -2F_A^+\cdot \phi_n -2(d^*_A\phi_n\wedge d_A^*\phi_n)^+ -2F_A^+$.
\begin{equation*}
 \begin{split}
 |2F_A^+\cdot \phi_n|_{\mathrm{T}}(x) &\leq 2c_0 \int g(x, y) |F_A^+(y)||\phi_n(y)|d\vol(y) ,\\
 &\leq 2c_0|F_A^+|_{\mathrm{T}}(x) \norm{\phi_n}_{L^\infty} 
 \leq 2c_0|F^+_A|_{\mathrm{T}}(x)\norm{\eta_n}_{\mathrm{T}}.
 \end{split}
\end{equation*}
\[ |2(d_A^*\phi_n\wedge d_A^*\phi_n)^+|_{\mathrm{T}}(x)
 \leq 2c_0\norm{\eta_n}_{\mathrm{T}}|\eta_n|_{\mathrm{T}}(x)  \quad 
 \text{(Lemma \ref{lemma: magical lemma of the Taubes norm})}.\]
Hence
\[ |\eta_{n+1}|_{\mathrm{T}}(x)
 \leq 2c_0\norm{\eta_n}_{\mathrm{T}}|F_A^+|_{\mathrm{T}}(x) 
 + 2c_0\norm{\eta_n}_{\mathrm{T}} |\eta_n|_{\mathrm{T}}(x) + 2|F_A^+|_{\mathrm{T}}(x).\]
Since $\norm{\eta_n}_{\mathrm{T}}\leq 3\varepsilon_0$, 
\[ |\eta_{n+1}|_{\mathrm{T}}(x) 
    \leq 6c_0\varepsilon_0 |\eta_n|_{\mathrm{T}}(x) + (6c_0\varepsilon_0 + 2)|F_A^+|_{\mathrm{T}}(x).\]
By (\ref{eq: fixing condition of varepsilon_0}),
\[ |\eta_n|_{\mathrm{T}}(x)
  \leq \frac{(6c_0\varepsilon_0 +2)|F_A^+|_{\mathrm{T}}(x)}{1-6c_0\varepsilon_0} 
  \leq 3|F_A^+|_{\mathrm{T}}(x).\]

Recall that $F_A^+$ is supported in $\{T<|t|<T+1\}$ and that $g(x,y)>0$ for $x\neq y$.
Set
\[ \delta(x) := \int_{T<|t|<T+1} g(x, y)d\vol(y) \quad (x\in X).\]
Then $|F_A^+|_{\mathrm{T}}(x) \leq \delta(x)\norm{F_A^+}_{L^\infty}$.
Note that $\delta(x)$ vanishes at infinity because 
$g(x,y) \leq \const\cdot e^{-\sqrt{S/3}d(x,y)}$ for $d(x,y)\geq 1$. 
(See (\ref{eq: exponential decay of g(x,y)}).)
We get the following decay estimate.
\begin{proposition}\label{prop: decay estimate}
$ |\phi_n(x)|\leq |\eta_n|_{\mathrm{T}}(x) \leq 3\delta(x) \norm{F_A^+}_{L^\infty}$.
Hence $|\phi_A(x)|\leq 3\delta(x)\norm{F_A^+}_{L^\infty}$. In particular,
$\phi_A$ vanishes at infinity.
\end{proposition}

\subsection{Regularity and the behavior at the end} \label{subsection: regularity and the behavior at the end}
From the definition of $\phi_n$, we have 
\begin{equation}\label{eq: differential equation for phi_n}
 (\nabla_A^*\nabla_A +S/3)\phi_{n+1} = \eta_{n+1} 
 = -2F_A^+\cdot \phi_n -2(d_A^*\phi_n\wedge d_A^*\phi_n)^+ -2F_A^+ .
\end{equation}
\begin{lemma} \label{lemma: gradient estimate}
$\sup_{n\geq 1}\norm{\nabla_A\phi_n}_{L^\infty} < +\infty$.
\end{lemma}
\begin{proof}
We use the rescaling argument of Donaldson \cite[Section 2.4]{Donaldson}.
Recall that $\phi_n$ are uniformly bounded and uniformly go to zero at infinity (Proposition \ref{prop: decay estimate}).
Moreover $\norm{\nabla_A\phi_n}_{L^\infty}<\infty$ for each $n\geq 1$ 
by Lemma \ref{lemma: key lemma for the construction}.
Suppose $\sup_{n\geq 1}\norm{\nabla_A\phi_n}_{L^\infty}=+\infty$.
Then there exists a sequence $n_1<n_2<n_3<\cdots$ such that $R_k :=\norm{\nabla_A\phi_{n_k}}_{L^\infty}$
go to infinity and $R_k \geq \max_{1\leq n\leq n_k}\norm{\nabla_A \phi_n}_{L^\infty}$.
Since $|\nabla_A\phi_n|$ vanishes at infinity (see Lemma \ref{lemma: key lemma for the construction}),
we can take $x_k \in X$ satisfying $R_k = |\nabla_A\phi_{n_k}(x_k)|$.
From the equation (\ref{eq: differential equation for phi_n}),
$|\nabla_A^*\nabla_A\phi_{n_k}|\leq \const_A \cdot R_k^2$.
Here ``$\const_A$'' means a positive constant depending on $A$ (but independent of $k\geq 1$).
Let $r_0>0$ be a positive number less than the injectivity radius of $X$.
We consider the geodesic coordinate centered at $x_k$ for each $k\geq 1$, and
we take a bundle trivialization of $E$ over each geodesic ball $B(x_k,r_0)$ by the exponential gauge
centered at $x_k$.
Then we can consider $\phi_{n_k}$ as a vector-valued function in the ball $B(x_k, r_0)$. 
Under this setting, $\phi_{n_k}$ satisfies 
\begin{equation} \label{eq: bound on Delta phi_{n_k}}
 \left|\sum_{i,j} g_{(k)}^{ij}\partial_i\partial_j \phi_{n_k}\right|\leq \const_A \cdot R_k^2 \quad \text{on $B(x_k,r_0)$},
\end{equation}
where $(g_{(k)}^{ij}) = (g_{(k),ij})^{-1}$ and $g_{(k),ij}$ is the Riemannian metric tensor 
in the geodesic coordinate centered at $x_k$.
(Indeed $S^3\times \mathbb{R}=SU(2)\times \mathbb{R}$ is a Lie group. 
Hence we can take the geodesic coordinates 
so that $g_{(k),ij}$ are independent of $k$.)
Set $\tilde{\phi}_k(x):=\phi_{n_k}(x/R_k)$.
$\tilde{\phi}_k(x)$ is a vector-valued function defined over the $r_0R_k$-ball in $\mathbb{R}^4$
centered at the origin.
$\tilde{\phi}_k$ $(k\geq 1)$ satisfy $|\nabla\tilde{\phi}_k(0)|=1$, and they are uniformly bounded.
From (\ref{eq: bound on Delta phi_{n_k}}), they satisfy
\[ \left|\sum_{i,j} \tilde{g}^{ij}_{(k)}\partial_i\partial_j\tilde{\phi}_k\right| \leq \const_A ,\]
where $\tilde{g}^{ij}_{(k)}(x) = g^{ij}_{(k)}(x/R_{k})$.
$\{\tilde{g}^{ij}_{(k)}\}_{k\geq 1}$ converges to $\delta^{ij}$
(the Kronecker delta) as $k\to +\infty$ 
in the $\mathcal{C}^\infty$-topology over compact subsets in $\mathbb{R}^4$.
Hence there exists a subsequence $\{\tilde{\phi}_{k_l}\}_{l\geq 1}$ which converges to some 
$\tilde{\phi}$ in the $\mathcal{C}^1$-topology over compact subsets in $\mathbb{R}^4$.
Since $|\nabla\tilde{\phi}_k(0)|=1$, we have $|\nabla\tilde{\phi}(0)|=1$.

If $\{x_{k_l}\}_{l\geq 1}$ is a bounded sequence, then $\{\tilde{\phi}_{k_l}\}$ has a subsequence
which converges to a constant function uniformly over every compact subset
because $\phi_n$ converges to $\phi_A$ (a continuous section) in the $\mathcal{C}^0$-topology ($=L^\infty\text{-topology}$)
and $R_k \to \infty$. 
But this contradicts the above $|\nabla\tilde{\phi}(0)|=1$.
Hence $\{x_{k_l}\}$ is an unbounded sequence. Since $\phi_n$ uniformly go to zero at infinity, 
$\{\tilde{\phi}_{k_l}\}$ has a subsequence which converges to $0$  uniformly over every compact subset.
Then this also contradicts $|\nabla\tilde{\phi}(0)|=1$.
\end{proof}
From Lemma \ref{lemma: gradient estimate} and the equation (\ref{eq: differential equation for phi_n}), the elliptic estimates show
that $\phi_n$ converges to $\phi_A$ in the $\mathcal{C}^\infty$-topology over every compact subset in $X$.
In particular, $\phi_A$ is smooth. 
(Indeed $\phi_A\in \Omega^+(\ad E)_0$ from Proposition \ref{prop: decay estimate}.) 
From the equation (\ref{eq: differential equation for phi_n}), 
\begin{equation}\label{eq: ASD equation for phi_A}
 (\nabla_A^*\nabla_A +S/3)\phi_A = -2F_A^+\cdot \phi_A -2(d_A^*\phi_A\wedge d_A^*\phi_A)^+ -2F_A^+.
\end{equation}
This implies that $A+d_A^*\phi_A$ is an ASD connection.

Lemma \ref{lemma: key lemma for the construction} shows $\lim_{x\to \infty}|\nabla_A\phi_n(x)| =0$
for each $n$.
Indeed we can prove a stronger result:
\begin{lemma}\label{lemma: decay estimate for the gradient}
For each $\varepsilon >0$, there exists a compact set $K\subset X$ such that for all $n$
\[ |\nabla_A\phi_n(x)|\leq \varepsilon \quad (x\in X\setminus K).\]
Therefore, $\lim_{x\to \infty}|\nabla_A\phi_A(x)|=0$.
\end{lemma}
\begin{proof}
Suppose the statement is false.
Then there are $\delta>0$, a sequence $n_1<n_2<n_3<\cdots$, and a sequence of points 
$x_1, x_2, x_3, \cdots$ in $X$ which goes to infinity such that 
\[ |\nabla_A\phi_{n_k}(x_k)|\geq \delta \quad (k =1, 2, 3, \cdots).\]
Let $x_k = (\theta_k ,t_k)\in S^3\times \mathbb{R} =X$. $|t_k|$ goes to infinity.
We can suppose $|t_k|> T+2$. 

Since $A$ is flat in $|t| > T+1$, there exists a bundle trivialization 
$g:E|_{|t|> T+1} \to X_{|t|> T+1} \times SU(2)$ such that $g(A)$ is equal to 
the product connection. (Here $X_{|t|>T+1} = \{(\theta, t)\in S^3\times \mathbb{R}|\, |t|>T+1\}$.)
Set $\phi_n':= g(\phi_n)$. We have 
\[ (\nabla^*\nabla +S/3)\phi'_{n} = -2(d^*\phi'_{n-1}\wedge d^*\phi'_{n-1})^+ \quad 
   (|t|> T+1) ,\]
where $\nabla$ is defined by using the product connection on $X_{|t|>T+1} \times SU(2)$.
From this equation and Lemma \ref{lemma: gradient estimate},
\[ |(\nabla^*\nabla +S/3)\phi'_{n}|\leq \const \quad (|t|> T+1),\]
where $\const$ is independent of $n$.
We define $\varphi_k\in \Gamma(S^3\times (-1,1), \Lambda^+\otimes su(2))$ by 
$\varphi_k(\theta, t) := \phi'_{n_k}(\theta, t_k +t)$.
We have $|(\nabla^*\nabla+S/3)\varphi_k|\leq \const$.
Since $|\phi'_n(x)|\leq 3\delta(x)\norm{F^+_A}_{L^\infty}$ and $|t_k|\to +\infty$, 
the sequence $\varphi_k$ converges to $0$ in $L^\infty(S^3\times (-1, 1))$.
Using the elliptic estimate, we get $\varphi_k\to 0$ in $\mathcal{C}^1(S^3\times [-1/2, 1/2])$.
On the other hand, $|\nabla \varphi_k(\theta_k, 0)| = |\nabla_A\phi_{n_k}(\theta_k, t_k)|\geq \delta>0$.
This is a contradiction.
\end{proof}
Set 
\begin{equation}\label{eq: eta_A and phi_A}
 \eta_A:= (\nabla_A^*\nabla_A+S/3)\phi_A 
 = -2F_A^+\cdot \phi_A -2(d_A^*\phi_A\wedge d_A^*\phi_A)^+ -2F_A^+.
\end{equation}
This is contained in $\Omega^+(\ad E)_0$ (Lemma \ref{lemma: decay estimate for the gradient}).
The sequence $\eta_n$ defined in Proposition \ref{prop: eta_n becomes Cauchy}
satisfies
\[ \eta_{n+1} = -2F_A^+\cdot \phi_n -2(d_A^*\phi_n\wedge d_A^*\phi_n)^+ -2F_A^+.\]
\begin{corollary} \label{cor: convergence of eta_n to eta_A}
The sequence $\eta_n$ converges to $\eta_A$ in $L^\infty$. 
In particular, $\norm{\eta_n-\eta_A}_{\mathrm{T}} \to 0$ as $n\to \infty$.
Hence $\norm{\eta_A}_{\mathrm{T}}\leq 3\varepsilon_0$. (Proposition \ref{prop: eta_n becomes Cauchy}.)
\end{corollary}
\begin{proof}
\[ \eta_{n+1}-\eta_A = -2F_A^+\cdot(\phi_n-\phi_A)
   +2\{d_A^*(\phi_A-\phi_n)\wedge d_A^*\phi_A + d_A^*\phi_n\wedge d_A^*(\phi_A-\phi_n)\}^+.\]
Hence 
\[ |\eta_{n+1}-\eta_A|\leq 2c_0\norm{F_A^+}_{L^\infty}\norm{\phi_n-\phi_A}_{L^\infty}
    + 2c_0(|\nabla_A\phi_n| + |\nabla_A\phi_A|)|\nabla_A\phi_A-\nabla_A\phi_n| .\]
$\phi_n\to \phi_A$ in $L^\infty(X)$ and in $\mathcal{C}^\infty$ over every compact subset.
Moreover $|\nabla_A\phi_n|$ are uniformly bounded and 
uniformly go to zero at infinity 
(Lemma \ref{lemma: gradient estimate} and Lemma \ref{lemma: decay estimate for the gradient}).
Then the above inequality implies that $\norm{\eta_{n+1}-\eta_A}_{L^\infty}$ goes to $0$.
\end{proof}
\begin{lemma} \label{lemma: the bound on the second derivative}
$\norm{d_A d_A^*\phi_A}_{L^\infty} < \infty$.
\end{lemma}
\begin{proof}
It is enough to prove $|d_A d_A^*\phi_A(\theta, t)|\leq \const$ for $|t|> T+2$.
Take a trivialization $g$ of $E$ over $|t| > T+1$ such that $g(A)$ is the product connection, 
and set $\phi' := g(\phi_A)$. This satisfies
\[ (\nabla^*\nabla +S/3)\phi' = -2(d^*\phi'\wedge d^*\phi')^+ \quad (|t|> T+1) .\]
Since $|\phi'|$ and $|\nabla \phi'|$ go to zero at infinity 
(Proposition \ref{prop: decay estimate} and Lemma \ref{lemma: decay estimate for the gradient}),
this shows (by using the elliptic estimates) that $|d d^*\phi'|$ is bounded.
\end{proof}
\begin{lemma}\label{lemma: finiteness of energy}
\[ \frac{1}{8\pi^2}\int_X |F(A+d_A^*\phi_A)|^2 d\vol = \frac{1}{8\pi^2}\int_X tr(F_A^2) .\]
Recall that $A$ is flat over $|t| > T+1$. Hence the right hand side is finite.
(Indeed it is a non-negative integer by the Chern-Weil theory.)
\end{lemma}
\begin{proof}
Set $a:= d_A^*\phi_A$ and $cs_A(a) := \frac{1}{8\pi^2}tr(2a\wedge F_A + a\wedge d_A a +\frac{2}{3}a^3)$.
We have $\frac{1}{8\pi^2}tr(F(A+a)^2) -\frac{1}{8\pi^2}tr(F(A)^2) = d cs_A(a)$.
Since $A+a$ is ASD, we have $|F(A+a)|^2d\vol = tr(F(A+a)^2)$ and 
\[ \frac{1}{8\pi^2}\int_{|t|\leq R} tr (F(A+a)^2) - \frac{1}{8\pi^2}\int_{|t|\leq R}tr (F(A)^2)
= \int_{t=R} cs_A(a) - \int_{t=-R} cs_A(a) .\]
From Lemma \ref{lemma: decay estimate for the gradient}, $|a| = |d_A^*\phi_A|$ goes to zero
at infinity. 
From Lemma \ref{lemma: the bound on the second derivative}, 
$|d_A a| = |d_A d_A^*\phi_A|$ is bounded.
$F_A$ vanishes over $|t|>T+1$.
Hence $|cs_A(a)|$ goes to zero at infinity.
Thus the right-hand-side of the above equation goes to zero as $R \to \infty$.
\end{proof}

\subsection{Conclusion of the construction} \label{subsection: conclusion of the construction}
The following is the conclusion of Sections \ref{subsection: construction} and 
\ref{subsection: regularity and the behavior at the end}.
This will be used in Sections \ref{section: continuity of the perturbation} 
and \ref{Section: proof of the upper bound}.
(Notice that we have not so far used the condition (iii) 
in the beginning of Section \ref{subsection: construction}.)
\begin{proposition} \label{prop: conclusion of the perturbation}
Let $T>0$.
Let $E$ be a principal $SU(2)$-bundle over $X$, and $A$ be a connection on $E$ satisfying 
$F_A=0$ $(|t|>T+1)$, $\supp F_A^+\subset \{T< |t|<T+1\}$ and 
$\norm{F_A^+}_{\mathrm{T}}\leq \varepsilon_0 = 1/1000$.
Then we can construct $\phi_A\in \Omega^+(\ad E)_0$ satisfying the following conditions.

\noindent 
(a) $A+d_A^*\phi_A$ is an ASD connection: $F^+(A+d_A^*\phi_A)=0$.

\noindent
(b) 
\[  \frac{1}{8\pi^2}\int_X |F(A+d_A^*\phi_A)|^2 d\vol = \frac{1}{8\pi^2}\int_X tr(F_A^2).\]

\noindent 
(c) $|\phi_A(x)|\leq 3\delta(x)\norm{F_A^+}_{L^\infty}$, where $\delta(x) = \int_{T<|t|<T+1}g(x,y)d\vol(y)$.

\noindent
(d) $\eta_A := (\nabla_A^*\nabla_A+S/3)\phi_A$ is contained in $\Omega^+(\ad E)_0$ and 
$\norm{\eta_A}_{\mathrm{T}} \leq 3\varepsilon_0$.

Moreover this construction $(E, A)\mapsto \phi_A$ is gauge equivariant, i.e., if $F$ is another
principal $SU(2)$-bundle over $X$ admitting a bundle map $g:E\to F$, then $\phi_{g(A)} = g(\phi_A)$.
\end{proposition}
\begin{proof}
The conditions (a), (b), (c), (d) have been already proved.
The gauge equivariance is obvious by the construction of $\phi_A$ in Section \ref{subsection: construction}.
\end{proof}

\subsection{Interior estimate} \label{subsection: interior estimate}
In the proof of the upper bound on the mean dimension, we need to use an ``interior estimate" of 
$\phi_A$ (Lemma \ref{lemma: interior estimate} below), which we investigate in this subsection.
We use the argument of Donaldson \cite[pp. 189-190]{Donaldson}.
Recall that $|F_A|\leq d$ on $|t|\leq T$ and $\norm{F_A^+}_{L^\infty(X)}\leq d'$ by the condition 
(iii) in the beginning of Section \ref{subsection: construction}.
We fix $r_0>0$ so that $r_0$ is less than the injectivity radius of $S^3\times \mathbb{R}$
(cf. the proof of Lemma \ref{lemma: gradient estimate}).
\begin{lemma}\label{lemma: sublemma for the interior estimate}
For any $\varepsilon>0$, 
there exists a constant $\delta_0 =\delta_0(d,\varepsilon)>0$ depending only on $d$ and $\varepsilon$
such that the following statement holds.
For any $\phi\in \Omega^+(\ad E)$ and any closed $r_0$-ball $B$ contained in $S^3\times [-T+1,T-1]$, if 
$\phi$ satisfies
\begin{equation} \label{eq: condition on the r_0-ball B}
 (\nabla_A^*\nabla_A+S/3)\phi 
  = -2(d_A^*\phi \wedge d_A^*\phi)^+  \text{ over $B$ and } \norm{\phi}_{L^\infty(B)}\leq \delta_0,
\end{equation}
then we have 
\[ \sup_{x\in B} |\nabla_A\phi(x)|d(x,\partial B) \leq \varepsilon.\]
Here $d(x, \partial B)$ is the distance between $x$ and $\partial B$. 
(If $T<1$, then $S^3\times [-T+1,T-1]$ is empty, and the above statement has no meaning.)
\end{lemma}
\begin{proof}
Suppose $\phi$ satisfies
\[ \sup_{x\in B}|\nabla_A\phi(x)|d(x,\partial B) > \varepsilon,\]
and the supremum is attained at $x_0\in B$ ($x_0$ is an inner point of $B$).
Set $R:=|\nabla_A\phi(x_0)|$ and $r_0':=d(x_0,\partial B)/2$.
We have $2r_0'R>\varepsilon$.
Let $B'$ be the closed $r_0'$-ball centered at $x_0$.
We have $|\nabla_A\phi|\leq 2R$ on $B'$.
We consider the geodesic coordinate over $B'$ centered at $x_0$, and 
we trivialize the bundle $E$ over $B'$ by the exponential gauge centered at $x_0$.
Since $A$ is ASD and $|F_A|\leq d$ over $-T\leq t\leq T$,
the $\mathcal{C}^1$-norm of the connection matrix of $A$ in the exponential gauge over $B'$ is bounded by a constant depending 
only on $d$.
From the equation (\ref{eq: condition on the r_0-ball B}) and $|\nabla_A\phi|\leq 2R$ on $B'$, 
\[ \left|\sum g^{ij}\partial_i\partial_j\phi\right|\leq \const_{d,\varepsilon}\cdot R^2 \quad \text{over $B'$},\]
where $(g^{ij}) = (g_{ij})^{-1}$ and $g_{ij}$ is the Riemannian metric tensor in the geodesic coordinate over $B'$.
Here we consider $\phi$ as a vector valued function over $B'$.
Set $\tilde{\phi}(x) := \phi(x/R)$. Since $2r_0'R>\varepsilon$, 
$\tilde{\phi}$ is defined over the $\varepsilon/2$-ball $B(\varepsilon/2)$
centered at the origin in $\mathbb{R}^4$, and it satisfies
\[ \left|\sum\tilde{g}^{ij}\partial_i\partial_i\tilde{\phi}\right|\leq \const_{d,\varepsilon} 
   \quad \text{over $B(\varepsilon/2)$}.\]
Here $\tilde{g}^{ij}(x):=g^{ij}(x/R)$.
The eigenvalues of the matrix $(\tilde{g}^{ij})$ are bounded from below by a positive constant depending only on the 
geometry of $X$,
and the $\mathcal{C}^1$-norm of $\tilde{g}^{ij}$ is bounded from above by a constant depending only on 
$\varepsilon$ and the geometry of $X$. 
(Note that $R>\varepsilon/(2r_0') \geq \varepsilon/(2r_0)$.)
Then by using the elliptic estimate \cite[Theorem 9.11]{Gilbarg-Trudinger} 
and the Sobolev embedding $L^8_2(B(\varepsilon/4))\hookrightarrow \mathcal{C}^{1,1/2}(B(\varepsilon/4))$ 
(the H\"{o}lder space), we get 
\[ |\!|\tilde{\phi}|\!|_{\mathcal{C}^{1,1/2}(B(\varepsilon/4))}\leq 
\const_\varepsilon\cdot |\!|\tilde{\phi}|\!|_{L^8_2(B(\varepsilon/4))}\leq C = C(d,\varepsilon).\]
Hence $|\nabla\tilde{\phi}(x)-\nabla\tilde{\phi}(0)|\leq C|x|^{1/2}$ on $B(\varepsilon/4)$.
Set $u:=\nabla\tilde{\phi}(0)$.
From the definition, we have $|u| =1$.
\[ \tilde{\phi}(tu)-\tilde{\phi}(0) = t\int_0^1\nabla\tilde{\phi}(tsu)\cdot uds 
    = t + t\int_0^1(\nabla\tilde{\phi}(tsu)-u)\cdot uds.\]
Hence 
\[ |\tilde{\phi}(tu)-\tilde{\phi}(0)|\geq t-t\int_0^1C|tsu|^{1/2}ds = t- 2Ct^{3/2}/3.\]
We can suppose $C\geq 2/\sqrt{\varepsilon}$. Then $u/C^2\in B(\varepsilon/4)$ and
\[ |\tilde{\phi}(u/C^2)-\tilde{\phi}(0)|\geq 1/(3C^2).\]
If $|\phi|\leq \delta_0 < 1/(6C^2)$, then this inequality becomes a contradiction.
\end{proof}
The following will be used in Section \ref{Section: proof of the upper bound}.
\begin{lemma} \label{lemma: interior estimate}
For any $\varepsilon >0$ there exists a positive number $D=D(d, d', \varepsilon)$ such that 
\[ \norm{d^*_A\phi_A}_{L^\infty(S^3\times [-T+D, T-D])} \leq \varepsilon .\]
(If $D>T$, then $S^3\times [-T+D,T-D]$ is the empty set.)
Here the important point is that $D$ is independent of $T$.
\end{lemma}
\begin{proof}
Note that $|d_A^*\phi_A|\leq \sqrt{3/2}|\nabla_A\phi_A|$.
We have $|\phi_A(x)|\leq 3d'\delta(x)$ by Proposition \ref{prop: conclusion of the perturbation} (c) 
(or Proposition \ref{prop: decay estimate}) and 
\[ \delta(x) = \int_{T<|t|<T+1}g(x,y)d\vol(y).\]
Set $D' := D-r_0$. (We choose $D$ so that $D'\geq 1$.)
Since $g(x,y) \leq \const\cdot e^{-\sqrt{S/3}d(x,y)}$ for $d(x,y)\geq 1$, we have 
\[ \delta(x) \leq C \cdot e^{-\sqrt{S/3}D'} \quad \text{for $x\in S^3\times [-T+D', T-D']$}.\]
We choose $D = D(d,d',\varepsilon) \geq r_0+1$ so that  
\[ 3d'Ce^{-\sqrt{S/3}D'}\leq \delta_0(d, r_0\varepsilon\sqrt{2/3})  .\]
Here $\delta_0(d, r_0\varepsilon\sqrt{2/3})$ is the positive constant introduced 
in Lemma \ref{lemma: sublemma for the interior estimate}.
Note that this condition is independent of $T$.
Then $\phi_A$ satisfies, for $x\in S^3\times [-T+D',T-D']$,
\begin{equation*}
 |\phi_A(x)|\leq \delta_0(d, r_0\varepsilon\sqrt{2/3}).
\end{equation*}
$\phi_A$ satisfies $(\nabla_A^*\nabla_A+S/3)\phi_A=-2(d_A^*\phi_A\wedge d^*_A\phi_A)^+$ over $|t|\leq T$.
Then Lemma \ref{lemma: sublemma for the interior estimate} implies
\begin{equation*}
 |\nabla_A\phi_A(x)|\leq \varepsilon\sqrt{2/3} \quad \text{for $x\in S^3\times [-T+D, T-D]$}.
\end{equation*}
(Note that, for $x\in S^3\times [-T+D, T-D]$, 
we have $B(x,r_0)\subset S^3\times [-T+D',T-D']$ and hence $|\phi_A|\leq \delta_0(d,r_0\varepsilon\sqrt{2/3})$
over $B(x,r_0)$.) Then, for $x\in S^3\times [-T+D,T-D]$, 
\[ |d_A^*\phi_A(x)|\leq \sqrt{3/2}|\nabla_A\phi_A(x)|\leq  \varepsilon. \]
\end{proof}

\section{Continuity of the perturbation} \label{section: continuity of the perturbation}
The purpose of this section is to show the continuity of the perturbation construction in 
Section \ref{section: solving ASD equation}.
The conclusion of Section \ref{section: continuity of the perturbation} 
is Proposition \ref{prop: continuity of the perturbation in C^0}.
As in Section \ref{section: solving ASD equation}, $X = S^3\times \mathbb{R}$, 
$T>0$ is a positive number, and $E\to X$ is a principal $SU(2)$-bundle.
Let $\rho$ be a flat connection on $E|_{|t|>T+1}$.
($E|_{|t|>T+1}$ is the restriction of $E$ to 
$X_{|t|>T+1} = \{(\theta,t)\in S^3\times \mathbb{R}|\, |t|>T+1\}$.)
We define $\mathcal{A}'$ as the set of connections $A$ on $E$ satisfying the following
(i), (ii), (iii).

\noindent
(i) $A|_{|t|>T+1} = \rho$, i.e., $A$ coincides with $\rho$ over $|t|>T+1$.

\noindent 
(ii) $F_A^+$ is supported in $\{(\theta, t)\in S^3\times \mathbb{R}|\, T<|t|<T+1\}$.

\noindent 
(iii) $\norm{F_A^+}_{\mathrm{T}} \leq \varepsilon_0 = 1/1000$.

By Proposition \ref{prop: conclusion of the perturbation}, for each $A \in \mathcal{A}'$,
we have $\phi_A\in \Omega^+(\ad E)_0$ and $\eta_A = (\nabla_A^*\nabla_A+S/3)\phi_A\in \Omega^+(\ad E)_0$ 
satisfying
\begin{equation} \label{eq: definition of eta_A}
 \eta_A = -2F_A^+\cdot \phi_A -2(d_A^*\phi_A\wedge d_A^*\phi_A)^+ -2F_A^+, \quad 
 \norm{\eta_A}_{\mathrm{T}} \leq 3\varepsilon_0.
\end{equation}
The first equation in the above is equivalent to the ASD equation $F^+(A+d_A^*\phi_A)=0$. 
Since $\phi_A = (\nabla_A^*\nabla_A+S/3)^{-1}\eta_A$, 
we have ((\ref{eq: phi leq eta}) and Lemma \ref{lemma: magical lemma of the Taubes norm}) 
\[ \norm{\phi_A}_{L^\infty}\leq \norm{\eta_A}_{\mathrm{T}}\leq 3\varepsilon_0, \quad 
   \norm{|\nabla_A\phi_A|^2}_{\mathrm{T}}\leq \norm{\eta_A}_{\mathrm{T}}^2 \leq 9\varepsilon_0^2 .\]
Then (by the Cauchy-Schwartz inequality) 
\[ \norm{\nabla_A\phi_A}_{\mathrm{T}} := \sup_{x\in X} \int_X g(x,y)|\nabla_A\phi_A(y)|d\vol(y)
 \leq 3\varepsilon_0\sqrt{K},\]
where $K = \int_Xg(x, y)d\vol(y)$. (The value of $K$ is independent of $x\in X$.)

Let $A, B\in \mathcal{A}'$. We want to estimate $\norm{\phi_A-\phi_B}_{L^\infty}$.
Set $a := B-A$. Since both $A$ and $B$ coincide with $\rho$ (the fixed flat connection)
over $|t|>T+1$, $a$ is compact-supported.
We set 
\[ \norm{a}_{\mathcal{C}^1_A} := \norm{a}_{L^\infty} + \norm{\nabla_A a}_{L^\infty}.\]
We suppose 
\[ \norm{a}_{\mathcal{C}^1_A} \leq 1.\]
\begin{lemma}\label{lemma: phi_A-phi_B}
$\norm{\phi_A-\phi_B}_{L^\infty} \leq \norm{\eta_A-\eta_B}_{\mathrm{T}} + \const \norm{a}_{\mathcal{C}^1_A}$,
where $\const$ is an universal constant independent of $A, B$.
\end{lemma}
\begin{proof}
We have $\eta_A = (\nabla^*_A\nabla_A+S/3)\phi_A$ and 
\[ \eta_B = (\nabla_B^*\nabla_B +S/3)\phi_B = 
(\nabla_A^*\nabla_A +S/3)\phi_B + (\nabla^*_A a)*\phi_B + a*\nabla_B\phi_B  + a*a*\phi_B,\]
where $*$ are algebraic multiplications.
Then 
\begin{equation*}
 \begin{split}
 \norm{\phi_A-\phi_B}_{L^\infty} &\leq \norm{(\nabla_A^*\nabla_A+S/3)(\phi_A-\phi_B)}_{\mathrm{T}} ,\\
 &\leq \norm{\eta_A-\eta_B}_{\mathrm{T}} + 
 \const \left(\norm{\nabla_A a}_{L^\infty}\norm{\phi_B}_{\mathrm{T}} 
 + \norm{a}_{L^\infty}\norm{\nabla_B \phi_B}_{\mathrm{T}}
 + \norm{a}_{L^\infty}^2\norm{\phi_B}_{\mathrm{T}}\right) ,\\
 &\leq \norm{\eta_A-\eta_B}_{\mathrm{T}} + \const \norm{a}_{\mathcal{C}^1_A}.
 \end{split}
\end{equation*}
\end{proof}
\begin{lemma}  \label{lemma: estimate on the quadratic terms}
\[ \norm{(d_A^*\phi_A\wedge d_A^*\phi_A)^+ - (d_B^*\phi_B\wedge d_B^*\phi_B)^+}_{\mathrm{T}}
   \leq  \left(\frac{1}{4}+\const\norm{a}_{\mathcal{C}^1_A}\right)\norm{\eta_A-\eta_B}_{\mathrm{T}} 
   + \const\norm{a}_{\mathcal{C}^1_A}.\]
\end{lemma}
\begin{proof}
\begin{equation*}
 \begin{split}
 (d_A^*\phi_A\wedge& d_A^*\phi_A)^+ - (d_B^*\phi_B\wedge d_B^*\phi_B)^+ =\\
 &\underbrace{(d_A^*\phi_A\wedge d_A^*\phi_A)^+ - (d_A^*\phi_B\wedge d_A^*\phi_B)^+}_{(I)} 
 + \underbrace{(d_A^*\phi_B\wedge d_A^*\phi_B)^+ -(d_B^*\phi_B\wedge d_B^*\phi_B)^+ }_{(II)}.
 \end{split}
\end{equation*}
We first estimate the term $(II)$. Since $B=A+a$,
\begin{equation*}
 \begin{split}
  (d_B^*\phi_B &\wedge d_B^*\phi_B)^+ -(d_A^*\phi_B\wedge d_A^*\phi_B)^+ =\\ 
  &(d_A^*\phi_B\wedge (a*\phi_B))^+ + ((a*\phi_B)\wedge d_A^*\phi_B)^+ 
  + ((a*\phi_B)\wedge (a*\phi_B))^+.
 \end{split}
\end{equation*} 
\begin{equation*}
 \begin{split}
 \norm{(II)}_{\mathrm{T}} &
 \leq \const \norm{\nabla_A\phi_B}_{\mathrm{T}}\norm{a}_{L^\infty}\norm{\phi_B}_{L^\infty}
                     +\const \norm{a}_{L^\infty}^2\norm{\phi_B}_{L^\infty}^2 ,\\ 
               &\leq \const\cdot  \norm{\nabla_A\phi_B}_{\mathrm{T}}\norm{a}_{L^\infty}
                    + \const\cdot \norm{a}_{L^\infty}. 
 \end{split}
\end{equation*}
We have 
\[ \norm{\nabla_A\phi_B}_{\mathrm{T}} = \norm{\nabla_B\phi_B + a*\phi_B}_{\mathrm{T}} 
   \leq \norm{\nabla_B\phi_B}_{\mathrm{T}} + \const \norm{a}_{L^\infty}\norm{\phi_B}_{L^\infty} \leq \const.\]
Hence $\norm{(II)}_{\mathrm{T}} \leq \const \norm{a}_{L^\infty}$.

Next we estimate the term $(I)$.
For $\eta_1, \eta_2\in \Omega^+(\ad E)_0$, set 
$\phi_i := (\nabla_A^*\nabla_A+S/3)^{-1}\eta_i \in \Omega^+(\ad E)_0$ $(i=1,2)$, 
and define (see (\ref{eq: definition of beta}))
\begin{equation*} 
 \beta_A(\eta_1, \eta_2) := (d_A^*\phi_1\wedge d_A^*\phi_2)^+ + (d_A^*\phi_2\wedge d^*_A\phi_1)^+.
\end{equation*}
Set $\eta'_B := (\nabla_A^*\nabla_A+S/3)\phi_B = \eta_B + (\nabla^*_A a)*\phi_B + a*\nabla_B\phi_B + a*a*\phi_B$.
Then $(d_A^*\phi_B\wedge d_A^*\phi_B)^+ = \beta_A(\eta'_B, \eta'_B)/2$ and
$(I) = (\beta_A(\eta_A, \eta_A)-\beta_A(\eta_B', \eta_B'))/2 = \beta_A(\eta_A +\eta_B', \eta_A-\eta_B')/2$.
From Lemma \ref{lemma: estimate of beta},
\[ \norm{(I)}_{\mathrm{T}}\leq 2c_0\norm{\eta_A+\eta_B'}_{\mathrm{T}}\norm{\eta_A-\eta_B'}_{\mathrm{T}}.\]
$\norm{\eta_A+\eta'_B}_{\mathrm{T}} \leq \norm{\eta_A+\eta_B}_{\mathrm{T}}+\norm{\eta_B'-\eta_B}_{\mathrm{T}} 
 \leq 6\varepsilon_0 + \const \norm{a}_{\mathcal{C}^1_A}$, and
$\norm{\eta_A-\eta_B'}_{\mathrm{T}}\leq \norm{\eta_A-\eta_B}_{\mathrm{T}}+\const \norm{a}_{\mathcal{C}^1_A}$.
From (\ref{eq: fixing condition of varepsilon_0}), we have $12c_0\varepsilon_0\leq 1/4$.
Then 
\[ \norm{(I)}_{\mathrm{T}} \leq 
 \left(\frac{1}{4}+\const\norm{a}_{\mathcal{C}^1_A}\right)\norm{\eta_A-\eta_B}_{\mathrm{T}} 
   + \const\norm{a}_{\mathcal{C}^1_A} .\]
\end{proof}
We have $F_B^+ = F_A^+ + d_A^+a + (a\wedge a)^+$. 
Recall that we have supposed $\norm{a}_{\mathcal{C}^1_A}\leq 1$.
Hence 
\[ |F_B^+-F_A^+|\leq \const \norm{a}_{\mathcal{C}^1_A}.\]
\begin{proposition}
There exists $\delta>0$ such that if $\norm{a}_{\mathcal{C}^1_A} \leq \delta$ then 
\[ \norm{\eta_A -\eta_B}_{\mathrm{T}} \leq \const \norm{a}_{\mathcal{C}^1_A}.\]
\end{proposition}
\begin{proof}
From (\ref{eq: definition of eta_A}),
\begin{equation*}
 \begin{split}
 \eta_A-\eta_B = \, & 2(F_B^+-F^+_A)\cdot \phi_B + 2F_A^+\cdot (\phi_B-\phi_A) \\
                 &+ 2((d_B^*\phi_B\wedge d_B^*\phi_B)^+ -(d_A^*\phi_A\wedge d_A^*\phi_A)^+)
                 +2(F_B^+ -F_A^+)
 \end{split}
\end{equation*}
Using $\norm{\phi_B}_{L^\infty}\leq 3\varepsilon_0$, $\norm{F_A^+}_{\mathrm{T}}\leq \varepsilon_0$
and Lemma \ref{lemma: estimate on the quadratic terms},
\[ \norm{\eta_A-\eta_B}_{\mathrm{T}}
  \leq \const \norm{a}_{\mathcal{C}^1_A} + 2c_0\varepsilon_0\norm{\phi_A-\phi_B}_{L^\infty} 
  + \left(\frac{1}{2}+\const \norm{a}_{\mathcal{C}^1_A}\right)\norm{\eta_A-\eta_B}_{\mathrm{T}}.\]
Using Lemma \ref{lemma: phi_A-phi_B},
\[ \norm{\eta_A-\eta_B}_{\mathrm{T}} \leq \const \norm{a}_{\mathcal{C}^1_A}
  + \left(\frac{1}{2}+\const \norm{a}_{\mathcal{C}^1_A} 
  + 2c_0\varepsilon_0\right)\norm{\eta_A-\eta_B}_{\mathrm{T}} .\]
From (\ref{eq: fixing condition of varepsilon_0}), we can choose $\delta>0$ so that if 
$\norm{a}_{\mathcal{C}^1_A} \leq \delta$ then
\[   \left(\frac{1}{2}+\const \norm{a}_{\mathcal{C}^1_A} + 2c_0\varepsilon_0\right) \leq 3/4 .\]                   
Then we get 
\[  \norm{\eta_A-\eta_B}_{\mathrm{T}} 
 \leq \const \norm{a}_{\mathcal{C}^1_A} + (3/4)\norm{\eta_A-\eta_B}_{\mathrm{T}}.\]
Then $\norm{\eta_A-\eta_B}_{\mathrm{T}} \leq \const \norm{a}_{\mathcal{C}^1_A}$.
\end{proof}
From Lemma \ref{lemma: phi_A-phi_B}, we get (under the condition $\norm{a}_{\mathcal{C}^1_A}\leq \delta$)
\[ \norm{\phi_A-\phi_B}_{L^\infty}\leq \norm{\eta_A-\eta_B}_{\mathrm{T}} + \const \norm{a}_{\mathcal{C}^1_A}
                                  \leq \const \norm{a}_{\mathcal{C}^1_A} .\]
Therefore we get the following.
\begin{corollary} \label{cor: L^infty convergence}
The map 
\[ (\mathcal{A}',\, \text{$\mathcal{C}^1$-topology}) \to (\Omega^+(\ad E)_0,\, \norm{\cdot}_{L^\infty}), 
\quad A\mapsto \phi_A ,\]
is continuous.
\end{corollary}
Let $A_n$ $(n\geq 1)$ be a sequence in $\mathcal{A}'$ which converges to $A\in \mathcal{A}'$ 
in the $\mathcal{C}^1$-topology:
$\norm{A_n-A}_{\mathcal{C}^1_A} \to 0$ $(n\to \infty)$.
By Corollary \ref{cor: L^infty convergence}, we get $\norm{\phi_{A_n}-\phi_A}_{L^\infty}\to 0$.
Set $a_n := A_n-A$. 
\begin{lemma} \label{lemma: bound on the derivative in the argument of continuity}
$\sup_{n\geq 1}\norm{\nabla_{A_n}\phi_{A_n}}_{L^\infty} < \infty$. 
(Equivalently, $\sup_{n\geq 1}\norm{\nabla_A\phi_{A_n}}_{L^\infty} < \infty$.)
\end{lemma}
\begin{proof}
Note that $|\nabla_{A_n}\phi_{A_n}|$ vanishes at infinity 
(see Lemma \ref{lemma: decay estimate for the gradient}).
Hence we can take a point $x_n\in S^3\times \mathbb{R}$ satisfying 
$|\nabla_{A_n}\phi_{A_n}(x_n)|=\norm{\nabla_{A_n}\phi_{A_n}}_{L^\infty}$.
$\norm{\phi_{A_n}-\phi_A}_{L^\infty}\to 0$ $(n\to \infty)$, and $\phi_{A_n}$ uniformly go to zero at infinity 
(see Proposition \ref{prop: conclusion of the perturbation} (c) or Proposition \ref{prop: decay estimate}).
Then the rescaling argument as in the proof of Lemma \ref{lemma: gradient estimate}
shows the above statement.
\end{proof}
Since $(\nabla_{A_n}^*\nabla_{A_n}+S/3)\phi_{A_n} 
= -2F_{A_n}^+\cdot \phi_{A_n} -2(d_{A_n}^*\phi_{A_n}\wedge d_{A_n}^*\phi_{A_n})^+ -2F_{A_n}^+$,
\[ \sup_{n\geq 1}\norm{\nabla_{A_n}^*\nabla_{A_n}\phi_{A_n}}_{L^\infty} < \infty.\]
We have 
$\nabla_{A_n}^*\nabla_{A_n}\phi_{A_n} = \nabla_A^*\nabla_A\phi_{A_n}
+(\nabla^*_A a_n)*\phi_{A_n} + a_n*\nabla_{A_n}\phi_{A_n} + a_n*a_n*\phi_{A_n}$.
Hence 
\[ \sup_{n\geq 1} \norm{\nabla_{A}^*\nabla_{A}\phi_{A_n}}_{L^\infty} < \infty.\]
By the elliptic estimate, we conclude that $\phi_{A_n}$ converges to $\phi_A$ in $\mathcal{C}^1$ over 
every compact subset. 
Then we get the following conclusion. This will be used in Section \ref{Section: proof of the upper bound}.
\begin{proposition}  \label{prop: continuity of the perturbation in C^0}
Let $\{A_n\}_{n\geq 1}$ be a sequence in $\mathcal{A}'$ which converges to $A\in \mathcal{A}'$ in the $\mathcal{C}^1$-topology.
Then $\phi_{A_n}$ converges to $\phi_A$ in the $\mathcal{C}^1$-topology over every compact subset in $X$.
Therefore $d_{A_n}^*\phi_{A_n}$ converges to $d_A^*\phi_A$ in the $\mathcal{C}^0$-topology 
over every compact subset in $X$.
Moreover, for any $n\geq 1$,
\[ \int_X |F(A_n + d_{A_n}^*\phi_{A_n})|^2 d\vol = \int_X |F(A+d_A^*\phi_A)|^2d\vol.\]
(This means that no energy is lost at the end.)
\end{proposition}
\begin{proof}
The last statement follows from Proposition \ref{prop: conclusion of the perturbation} (b) 
(or Lemma \ref{lemma: finiteness of energy}) and the fact that for any $A$ and $B$ in $\mathcal{A}'$ we have 
\[ \int_X tr(F_A^2) = \int_X tr(F_B^2) .\]
This is because $tr F_B^2 -tr F_A^2 = d (tr(2a\wedge F_A + a\wedge d_A a + \frac{2}{3}a^3))$
$(a = B-A)$, and
both $A$ and $B$ coincide with the fixed flat connection $\rho$ over $|t|>T+1$.
\end{proof}

\section{``Non-flat'' implies ``irreducible''}  \label{section: non-flat implies irreducible}
This section is short. But the results in this section are crucial for both proofs of the upper and lower bounds
on the mean dimension.
Note that the following trivial fact: 
if a smooth function $u$ on $\mathbb{R}$ is bounded and convex ($u''\geq 0$)
then $u$ is a constant function.
\begin{lemma} \label{lemma: bounded, nonnegative subharminic function}
If a smooth function $f$ on $S^3\times \mathbb{R}$ is bounded, non-negative and sub-harmonic 
($\Delta f\leq 0$)\footnote{Our convention of the sign of the Laplacian is geometric; we have 
$\Delta = -\partial^2/\partial x_1^2 -\partial^2/\partial x_2^2 -\partial^2/\partial x_3^2 -\partial^2/\partial x_4^2$
on $\mathbb{R}^4$},
then $f$ is a constant function.
\end{lemma}
\begin{proof}
We have $\Delta = -\partial^2/\partial t^2 + \Delta_{S^3}$ where $t$ is the coordinate of the $\mathbb{R}$-factor of 
$S^3\times \mathbb{R}$ and $\Delta_{S^3}$ is the Laplacian of $S^3$.
We have 
\[ \frac{\partial^2}{\partial t^2}f^2 = 2\left(\frac{\partial f}{\partial t}\right)^2 + 2f\Delta_{S^3}f -2f\Delta f.\]
Then we have 
\[ \frac{1}{2}\frac{\partial^2}{\partial t^2}\int_{S^3\times \{t\}} f^2 d\vol
= \int_{S^3\times\{t\}}\left( \left|\frac{\partial f}{\partial t}\right|^2 + |\nabla_{S^3} f|^2 + f(-\Delta f)\right) d\vol \geq 0.\]
Here we have used $f\geq 0$ and $\Delta f\leq 0$. 
This shows that $u(t) = \int_{S^3\times\{t\}}f^2$ is a bounded convex function. Hence 
it is a constant function. In particular $u'' \equiv 0$. Then the above formula implies 
$\partial f/\partial t \equiv \nabla_{S^3} f \equiv 0$. This means that $f$ is a constant function. 
\end{proof}
\begin{lemma}
If $A$ is a $U(1)$-ASD connection on $S^3\times \mathbb{R}$ satisfying $\norm{F_A}_{L^\infty}<\infty$, then $A$ is flat.
\end{lemma}
\begin{proof}
We have $F_A \in \sqrt{-1}\Omega^-$.
The Weitzenb\"{o}ck formula (cf. (\ref{eq: Weitzenbock formula})) gives 
$(\nabla^*\nabla+S/3)F_A= 2d^{-}d^*F_A = 0$.
We have 
\[ \Delta |F_A|^2 = -2|\nabla F_A|^2 + 2(F_A, \nabla^*\nabla F_A) 
= -2|\nabla F_A|^2 -(2S/3)|F_A|^2\leq 0.\]
This shows that $|F_A|^2$ is a non-negative, bounded, subharmonic function. Hence it is a constant function.
In particular $\Delta |F_A|^2 \equiv 0$. Then the above formula implies $F_A\equiv 0$.
\end{proof}
\begin{corollary}\label{cor: non-flat implies irreducible}
If $A$ is a non-flat $SU(2)$-ASD connection on $S^3\times \mathbb{R}$ satisfying $\norm{F_A}_{L^\infty}<\infty$, 
then $A$ is irreducible.
\end{corollary}
This corollary will be used in the proof of the lower bound on the mean dimension.
The following proposition will be used in the cut-off construction in Section \ref{section: cut-off constructions}.
\begin{proposition} \label{prop: irreducibility on 3-sphere}
Let $A$ be a non-flat $SU(2)$-ASD connection on $S^3\times \mathbb{R}$ satisfying $\norm{F_A}_{L^\infty}<+\infty$.
The restriction of $A$ to $S^3\times \{0\}$ is irreducible.
\end{proposition}
\begin{proof}
This follows from the above Corollary \ref{cor: non-flat implies irreducible} and 
the result of Taubes \cite[Theorem 5]{Taubes-3}.
Here we give a brief proof for readers' convenience.
Note that the Riemannian metric on $S^3\times \mathbb{R}$ is real analytic.
(Later we will use Cauchy-Kovalevskaya's theorem. Hence the real analyticity of all data is essential.)
Suppose $A|_{S^3\times \{0\}}$ is reducible.
Fix $p\in S^3$ and take a small open neighborhood $\Omega\subset S^3$ of $p$.
Let $\varepsilon>0$ be a small positive number.
By using the Uhlenbeck gauge \cite[Corollary 1.4]{Uhlenbeck}, 
we can suppose that $A$ is represented by a real analytic connection matrix over 
$\Omega\times (-\varepsilon, \varepsilon)$.
Moreover, by using the (real analytic) temporal gauge (see Donaldson \cite[Chapter 2]{Donaldson-Furuta-Kotschick}),
we can assume that the (real analytic) connection matrix of $A$ over $\Omega\times (-\varepsilon,\varepsilon)$ 
is $dt$-part free and satisfies
\begin{equation*} 
\frac{\partial}{\partial t}A(t) = *_3 F(A(t))_3,
\end{equation*}
where $A(t) := A|_{\Omega\times \{t\}}$ and $F(A(t))_3$ is the curvature of $A(t)$ as a connection over 
the 3-manifold $\Omega\times \{t\}$. $*_3$ is the Hodge star on $\Omega\times \{t\}$.

Since $A(0)$ is reducible, 
there exists a real analytic gauge transformation $u$ $(\neq \pm 1)$  over $\Omega$ satisfying 
$u(A(0)) = A(0)$.
Set $B := u(A)$ over $\Omega\times (-\varepsilon,\varepsilon)$.
$B$ is real analytic and satisfies 
\begin{equation*}
\frac{\partial}{\partial t}B(t) = *_3F(B(t))_3.
\end{equation*}
$A$ and $B$ are both real analytic and satisfy the same real analytic equation
of the normal form with the same real analytic initial value
$A(0) = B(0)$.
Therefore Cauchy-Kovalevskaya's theorem implies $A=B = u(A)$.
This means that $A$ is reducible over an open set $\Omega\times (-\varepsilon,\varepsilon)\subset S^3\times \mathbb{R}$.
Then the unique continuation principle (see Donaldson-Kronheimer \cite[Lemma 4.3.21]{Donaldson-Kronheimer})
implies that $A$ is reducible all over $S^3\times \mathbb{R}$.
But this contradicts Corollary \ref{cor: non-flat implies irreducible}.
\end{proof}

\section{Cut-off constructions} \label{section: cut-off constructions}
As we explained in Section \ref{section: outline of the proofs of the main theorems}, 
we need to define a `cut-off' of $[\bm{A}]\in \moduli_d$.
Section \ref{subsection: gauge fixing on S^3} is a preparation to define a cut-off construction, 
and we define it in Section \ref{subsection: cut-off construction}.

Let $\delta_1>0$.
We define $\delta'_1 =\delta'_1(\delta_1)$ by 
\[ \delta'_1 := \sup_{x\in S^3\times \mathbb{R}}
 \left( \int_{S^3\times (-\delta_1,\delta_1)} g(x,y)d\vol(y)\right) .\]
Since we have $g(x,y)\leq \const/d(x,y)^2$ 
(see (\ref{eq: singulairy of g(x,y) along the diagonal}) and (\ref{eq: exponential decay of g(x,y)})),
\[ \int_{d(x,y)\leq (\delta_1)^{1/4}}g(x,y)d\vol(y)\leq \const\int_0^{(\delta_1)^{1/4}}rdr=\const'\sqrt{\delta_1},\]
\[ \int_{\{d(x,y)\geq (\delta_1)^{1/4}\} \cap S^3\times(-\delta_1,\delta_1)} g(x,y)d\vol(y) 
 \leq \const\cdot \delta_1\frac{1}{\sqrt{\delta_1}} = \const\sqrt{\delta_1}.\]
Hence $ \delta_1'\leq \const \sqrt{\delta_1}$ (this calculation is due to \cite[pp. 190-191]{Donaldson}).
In particular, we have $\delta'_1\to 0$ as $\delta_1 \to 0$.
For $d\geq 0$, we choose $\delta_1 =\delta_1(d)$ so that
$0<\delta_1<1$ and $\delta'_1 = \delta'_1(\delta_1(d))$ satisfies
\begin{equation} \label{eq: choice of delta_1}
 (5+7d+d^2)\delta'_1 \leq \varepsilon_0/4 = 1/(4000).
\end{equation}
The reason of this choice will be revealed in Proposition \ref{prop: conclusion of cut-off}.

\subsection{Gauge fixing on $S^3$ and gluing instantons}  \label{subsection: gauge fixing on S^3}
Let $F := S^3 \times SU(2)$ be the product principal $SU(2)$-bundle over $S^3$.
Let $\mathcal{A}_{S^3}$ be the space of connections on $F$, and
$\mathcal{G}$ be the gauge transformation group of $F$. 
$\mathcal{A}_{S^3}$ and $\mathcal{G}$ are equipped with the $\mathcal{C}^\infty$-topology.
Set $\mathcal{B}_{S^3} := \mathcal{A}_{S^3}/\mathcal{G}$ (with the quotient topology), and let 
$\pi :\mathcal{A}_{S^3} \to \mathcal{B}_{S^3}$ be the natural projection.
Note that the gauge transformations $\pm 1$ trivially act on $\mathcal{A}_{S^3}$.
\begin{proposition}\label{prop: gauge fixing on S^3}
Let $d\geq 0$, and $A\in \mathcal{A}_{S^3}$ be an irreducible connection.
There exist a closed neighborhood $U_A$ of $[A]$ in 
$\mathcal{B}_{S^3}$ and a continuous map $\Phi_{A}: \pi^{-1}(U_A) \to \mathcal{G}/\{\pm 1\}$ 
such that, for any $B\in \pi^{-1}(U_A)$, $[g] := \Phi_A(B)$ satisfies 
the following.

\noindent
(i) $g(B) = A + a$ with $\norm{a}_{L^\infty} \leq \delta_1 = \delta_1(d)$.
($\delta_1$ is the positive constant chosen in the above (\ref{eq: choice of delta_1}).)

\noindent
(ii) For any gauge transformation $h$ of $F$, we have $\Phi_A(h(B)) = [gh^{-1}]$.
\end{proposition}
\begin{proof}
Let $\varepsilon>0$ be sufficiently small, and we take a closed neighborhood $U_A$ of $[A]$ in 
$\mathcal{B}_{S^3}$ such that 
\[ U_A \subset \{[B]\in \mathcal{B}_{S^3}|\, \exists \text{$g$: gauge transformation of $F$ s.t. } 
   \norm{g(B)-A}_{L^4_1}< \varepsilon\}.\]
The usual Coulomb gauge construction shows that, for each $B\in \pi^{-1}(U_A)$, there uniquely exists
$[g]\in \mathcal{G}/\{\pm 1\}$ such that $g(B) = A+ a$ with $d_A^*a=0$ and 
$\norm{a}_{L^4_1} \leq \const\cdot\varepsilon$.
Since $L^4_1(S^3)\hookrightarrow \mathcal{C}^0(S^3)$, 
we have $\norm{a}_{L^\infty}\leq \const\cdot \varepsilon \leq \delta_1$ for 
sufficiently small $\varepsilon$.
We define $\Phi_A(B):=[g]$.
Then the condition (i) is obviously satisfied, 
and the condition (ii) follows from the uniqueness of $[g]$.
\end{proof}
\begin{proposition}\label{prop: gauge fixing on S^3: flat connection}
Let $d\geq 0$, and $\Theta$ be the product connection on $F =S^3\times SU(2)$.
There exist a closed neighborhood $U_{\Theta}$ of $[\Theta]$ in $\mathcal{B}_{S^3}$
and a continuous map $\Phi_{\Theta}: \pi^{-1}(U_\Theta)\to \mathcal{G}$ such that,
for any $A\in \pi^{-1}(U_\Theta)$, $g:=\Phi_\Theta(A)$ satisfies the following.

\noindent
(i) $g(A) = \Theta + a$ with $\norm{a}_{L^\infty}\leq \delta_1 = \delta_1(d)$.

\noindent
(ii) For any gauge transformation $h$ of $F$, there exists a constant gauge transformation $h'$ of $F$
(i.e. $h'(\Theta) =\Theta$) such that $\Phi_{\Theta}(h(A)) = h'gh^{-1}$.
\end{proposition}
\begin{proof}
Fix a point $\theta_0\in S^3$.
Let $\varepsilon>0$ be sufficiently small, and we take a closed neighborhood $U_\Theta$ of $[\Theta]$ in 
$\mathcal{B}_{S^3}$ such that 
\[ U_\Theta \subset \{[A]\in \mathcal{B}_{S^3}|\, \exists \text{$g$: gauge transformation of $F$ s.t. } 
   \norm{g(A)-\Theta}_{L^4_1}< \varepsilon\}.\]
For any $A\in \pi^{-1}(U_\Theta)$, there uniquely exists a gauge transformation $g$ with $g(\theta_0) =1$ 
such that $g(A) = \Theta+a$ with $d_\Theta^*a=0$ and 
$\norm{a}_{L^4_1}\leq \const\cdot \varepsilon \> (\leq \delta_1)$.
We set $\Phi_\Theta(A) := g$.
The conditions (i) is obvious, and condition (ii) follows from 
$\Phi_\Theta(h(A)) = h(\theta_0)gh^{-1}$
(Here $h(\theta_0)$ is a constant gauge transformation.
Note that $(h(\theta_0)gh^{-1})(\theta_0) =1$.)
\end{proof}
Recall the settings in Section \ref{section: introduction}.
Let $d\geq 0$. 
The moduli space $\moduli_d$ is the space of all gauge equivalence classes $[\bm{A}]$
where $\bm{A}$ is an ASD connection on $\bm{E} := X\times SU(2)$ satisfying $|F(\bm{A})|\leq d$.

We define $K_d \subset \mathcal{B}_{S^3}$ by 
\[ K_d := \{[\bm{A}|_{S^3\times \{0\}}] \in \mathcal{B}_{S^3}|\, [\bm{A}] \in \moduli_d\}, \]
where we identify $\bm{E}|_{S^3\times \{0\}}$ with $F$.
From the Uhlenbeck compactness \cite{Uhlenbeck, Wehrheim},
$\moduli_d$ is compact, and hence $K_d$ is also compact.
Proposition \ref{prop: irreducibility on 3-sphere} implies that, for any $[\bm{A}]\in \moduli_d$,
$\bm{A}|_{S^3\times\{0\}}$ is irreducible or a flat connection.
(The important point is that $\bm{A}|_{S^3\times\{0\}}$ never be a non-flat reducible connection.)

Set $A_0 := \Theta$ (the product connection on $F$). There exist irreducible connections 
$A_1, A_2, \cdots, A_N$ $(N=N(d))$ on $F$ such that 
$K_d \subset \mathrm{Int}(U_{A_0}) \cup \mathrm{Int}(U_{A_1})\cup \cdots \cup \mathrm{Int}(U_{A_N})$ and 
$[A_i]\in K_d$ $(0\leq i\leq N)$.
Here $\mathrm{Int}(U_{A_i})$ is the interior of the closed set $U_{A_i}$
introduced in Propositions \ref{prop: gauge fixing on S^3} and \ref{prop: gauge fixing on S^3: flat connection}.
Note that we can naturally identify $K_d$ with the space 
$\{[\bm{A}|_{S^3\times\{T\}}]\in \mathcal{B}_{S^3}|\, [\bm{A}]\in \moduli_d\}$
for any real number $T$ because $\moduli_d$ admits the natural $\mathbb{R}$-action.

For the statement of the next proposition, we introduce a new notation.
We denote $F\times \mathbb{R}$ 
as the pull-back of $F$ by the natural projection $X=S^3\times \mathbb{R}\to S^3$.
So $F\times \mathbb{R}$ is a principal $SU(2)$-bundle over $X$.
Of course, we can naturally identify $F\times \mathbb{R}$ with $\bm{E}$, but here
we use this notation for the later convenience.
We define $\hat{A}_0$ as the pull-back of $\Theta$ by the projection $X=S^3\times \mathbb{R}\to S^3$.
(Hence $\hat{A}_0$ is the product connection on $F\times \mathbb{R}$ under the natural identification 
$F\times \mathbb{R} = \bm{E}$.)
\begin{proposition} \label{prop: gluing instantons}
For each $i = 1, 2, \cdots, N$ there exists a connection $\hat{A}_i$ on $F\times \mathbb{R}$ satisfying the following. (Recall $0<\delta_1<1$.)

\noindent 
(i) $\hat{A}_i = A_i$ over $S^3\times [-\delta_1, \delta_1]$. Here $A_i$ 
(a connection on $F\times \mathbb{R}$)
means the pull-back of $A_i$ (a connection on $F$) by the natural projection $X\to S^3$.

\noindent
(ii) $F(\hat{A}_i)$ is supported in $S^3\times (-1, 1)$.

\noindent
(iii) $\norm{F^+(\hat{A}_i)|_{\delta_1<|t|<1}}_{\mathrm{T}} \leq  \varepsilon_0/4 =1/(4000)$, where 
$F^+(\hat{A}_i)|_{\delta_1<|t|<1} = F^+(\hat{A}_i)\times 1_{\delta_1<|t|<1}$ 
and $1_{\delta_1<|t|<1}$ is the characteristic function of the set 
$\{(\theta, t)\in S^3\times \mathbb{R}|\, \delta_1<|t|<1\}$.
\end{proposition}
\begin{proof}
By using a cut-off function, we can construct a connection $A_i'$ on $F\times \mathbb{R}$ 
such that $A_i'=A_i$
over $S^3\times [-\delta_1,\delta_1]$ and $\supp F(A_i') \subset S^3\times (-1,1)$.
We can reduce the self-dual part of $F(A_i')$ by ``gluing instantons'' to $A_i'$ over $\delta_1< |t| < 1$.
This technique is essentially well-known for the specialists in the gauge theory.
For the detail, see Donaldson \cite[pp. 190-199]{Donaldson}.

By the argument of \cite[pp. 196-198]{Donaldson}, we get the following situation.
For any $\varepsilon>0$, there exists a connection $\hat{A}_i$ satisfying the following. 
$\hat{A}_i=A'_i =A_i$ over $|t|\leq \delta_1$, and $\supp F(\hat{A}_i)\subset S^3\times (-1,1)$.
Moreover $F^+(\hat{A}_i)=F^+_1+F^+_2$ over $\delta_1 < |t| < 1$ such that 
$|F^+_1|\leq \varepsilon$ and 
\[ |F^+_2|\leq \const,  \quad  \vol(\supp (F^+_2))\leq \varepsilon,\]
where $\const$ is a positive constant depending only on $A_i'$ and independent of $\varepsilon$.
If we take $\varepsilon$ sufficiently small, then 
\[ \norm{F^+(\hat{A}_i)|_{\delta_1<|t|<1}}_{\mathrm{T}}\leq \varepsilon_0/4.\]
\end{proof}

\subsection{Cut-off construction} \label{subsection: cut-off construction}
Let $T$ be a positive real number.
We define a closed subset $\moduli_{d,T}(i, j) \subset \moduli_d$ 
$(0\leq i, j\leq N = N(d))$ as the set of $[\bm{A}]\in \moduli_d$ satisfying
$[\bm{A}|_{S^3\times \{T\}}]\in U_{A_i}$ and 
$[\bm{A}|_{S^3\times \{-T\}}]\in U_{A_j}$.
Here we naturally identify $E_{T} := \bm{E}|_{S^3\times \{T\}}$ 
and $E_{-T}:= \bm{E}|_{S^3\times \{-T\}}$ with $F$, and 
$A_i$ $(0\leq i \leq N)$ are the connections on $F$ introduced in the previous subsection.
We have 
\begin{equation}\label{eq: decomposition of the moduli space}
 \moduli_d = \bigcup_{0 \leq i, j\leq N} \moduli_{d,T}(i, j) .
\end{equation}
Of course, this decomposition depends on the parameter $T>0$.
The important point is that $N$ is independent of $T$. 
We will define a cut-off construction for each piece $\moduli_{d,T}(i,j)$.

Let $\bm{A}$ be an ASD connection on $\bm{E}$ satisfying 
$[\bm{A}]\in \moduli_{d,T}(i, j)$.
Let $u_+:\bm{E}|_{t\geq T} \to E_T\times [T,+\infty)$ be the temporal gauge of $\bm{A}$ with 
$u_+ = \mathrm{id}$ on $\bm{E}|_{S^3\times\{T\}}=E_T$. (See Donaldson \cite[Chapter 2]{Donaldson-Furuta-Kotschick}.)
Here $\bm{E}|_{t\geq T}$ is the restriction of $\bm{E}$ to $S^3\times [T, +\infty)$, 
and $E_T\times [T, +\infty)$ is the pull-back of $E_T$ by the projection
$S^3\times [T, \infty)\to S^3\times \{T\}$.
We will repeatedly use these kinds of notations.
In the same way, let $u_-:\bm{E}|_{t\leq -T} \to E_{-T}\times (-\infty, -T]$ be the temporal gauge of $\bm{A}$
with $u_- = \mathrm{id}$ on $\bm{E}|_{S^3\times\{-T\}}=E_{-T}$.
We define $A(t)$ $(|t|\geq T)$ by setting $A(t) := u_+(\bm{A})$ for $t\geq T$ and 
$A(t) := u_-(\bm{A})$ for $t\leq -T$.
$A(t)$ becomes $dt$-part free.
Since $\bm{A}$ is ASD, we have 
\begin{equation}\label{eq: parabolic ASD equation}
 \frac{\partial A(t)}{\partial t} = *_3 F(A(t))_3 ,
\end{equation}
where $*_3$ is the Hodge star on $S^3 \times \{t\}$ and $F(A(t))_3$ is the curvature of $A(t)$ as a connection on the 
3-manifold $S^3\times \{t\}$.

We have $[A(T)]\in U_{A_i}$ and $[A(-T)]\in U_{A_j}$.
By using Propositions \ref{prop: gauge fixing on S^3} and \ref{prop: gauge fixing on S^3: flat connection},
we set $[g_+] := \Phi_{A_i}(A(T))$ if $i>0$ and $g_+:=\Phi_\Theta(A(T))$ if $i=0$.
(If $i>0$, the gauge transformation $g_+$ is not uniquely determined because there exists the ambiguity coming from $\pm 1$.
For this point, see Lemma \ref{lemma: gauge equivariance of cut-off} and its proof.)
In the same way we set $[g_-] := \Phi_{A_j}(A(-T))$ if $j>0$ and $g_- := \Phi_\Theta(A(-T))$ if $j=0$.
We consider $g_+$ (resp. $g_-$) as the gauge transformation of $E_{T}$ (resp. $E_{-T}$).
They satisfy
\begin{equation} \label{eq: estimate of deviation in cut-off}
 \norm{g_+(A(T))-A_i}_{L^\infty} \leq \delta_1, \quad 
 \norm{g_-(A(-T))-A_j}_{L^\infty}\leq \delta_1.
\end{equation}
We define a principal $SU(2)$-bundle $\bm{E}'$ over $X$ by 
\[ \bm{E}' := \bm{E}|_{|t|<T+\delta_1/4} \sqcup E_T\times (T, +\infty) \sqcup E_{-T}\times (-\infty, -T)/\sim,\]
where the identification $\sim$ is given as follows.
$\bm{E}|_{|t|<T+\delta_1/4}$ is identified with $E_T\times (T, +\infty)$ over the region $T<t<T+\delta_1/4$ by 
the map $g_+\circ u_+ :\bm{E}|_{T<t<T+\delta_1/4} \to E_T\times (T, T+\delta_1/4)$.
Here we consider $g_+$ as a gauge transformation of $E_T\times (T,T+\delta_1/4)$
by $g_+:E_T\times (T,T+\delta_1/4)\to E_T\times (T,T+\delta_1/4)$, $(p,t)\mapsto (g_+(p),t)$.
Similarly, we identify $\bm{E}|_{|t|<T+\delta_1/4}$ with $E_{-T}\times (-\infty, -T)$ over the region 
$-T-\delta/4 <t<-T$ by the map $g_-\circ u_-:\bm{E}|_{-T-\delta_1/4<t<-T}\to E_{-T}\times (-T-\delta_1/4, -T)$.

Let $\rho(t)$ be a smooth function on $\mathbb{R}$ such that 
$0\leq \rho\leq 1$, $\rho = 0$ $(|t|\leq \delta_1/4)$, $\rho =1$ $(|t|\geq 3\delta_1/4)$ and
\begin{equation*}
|\rho'|\leq 4/\delta_1.
\end{equation*} 
We define a (not necessarily ASD) connection $\bm{A}'$ on $\bm{E}'$ as follows.
Over the region $|t|<T+\delta_1/4$ where
$\bm{E}'$ is equal to $\bm{E}$, we set 
\begin{equation}\label{eq: definition of A', first}
\bm{A}' := \bm{A} \quad \text{on $\bm{E}|_{|t|<T+\delta_1/4}$}.
\end{equation}
Over the region $t>T$, we set
\begin{equation}\label{eq: definition of A', second}
 \bm{A}' := (1-\rho(t-T))g_+(A(t)) + \rho(t-T)\hat{A}_{i, T}\quad \text{on $E_T\times (T, +\infty)$}, 
\end{equation}
where $\hat{A}_{i, T}$ is the pull-back of the connection $\hat{A}_i$ introduced in the previous subsection 
(see Proposition \ref{prop: gluing instantons}) by the map $t\mapsto t-T$. 
So, in particular, $\hat{A}_{i, T} = A_i$ over $T-\delta_1\leq t\leq T+\delta_1$ and 
$F(\hat{A}_{i, T}) = 0$ over $t\geq T+1$.
(\ref{eq: definition of A', second}) is compatible with (\ref{eq: definition of A', first}) 
over $T<t<T+\delta_1/4$ where $\rho(t-T) =0$.
In the same way, over the region $t<-T$, we set 
\[ \bm{A}' := (1-\rho(t+T))g_-(A(t)) + \rho(t+T)\hat{A}_{j, -T} \quad \text{on $E_{-T}\times (-\infty, -T)$}.\]
We have $F(\bm{A}') =0$ $(|t|\geq T+1)$.
Then we have constructed $(\bm{E}', \bm{A}')$ from $\bm{A}$
with $[\bm{A}] \in \moduli_{d,T}(i, j)$.
\begin{lemma} \label{lemma: gauge equivariance of cut-off}
The gauge equivalence class of $(\bm{E}',\bm{A}')$ depends only on the gauge equivalence class of $\bm{A}$.
\end{lemma}
\begin{proof}
Suppose $[\bm{A}]\in \moduli_{d,T}(0,1)$. Other cases can be proved in the same way.
Let $h:\bm{E}\to \bm{E}$ be a gauge transformation and set $\bm{B}:=h(\bm{A})$.
Let $(\bm{E}'_{\bm{A}}, \bm{A}')$ and $(\bm{E}'_{\bm{B}},\bm{B}')$ be the bundles and connections constructed by 
the above cut-off procedure form $\bm{A}$ and $\bm{B}$, respectively.
Let $u_{\pm, \bm{A}}$ and $u_{\pm,\bm{B}}$ be the temporal gauges of $\bm{A}$ and $\bm{B}$ over $t\geq T$ or $t\leq -T$.
We have $u_{\pm,\bm{B}}=h_{\pm T}\circ u_{\pm,\bm{A}}\circ h^{-1}$ where $h_{\pm T}:=h|_{t=\pm T}$ on $E_{\pm T}$.
Set $g_{+,\bm{A}}:=\Phi_{\Theta}(A(T))$ and $g_{+,\bm{B}}:=\Phi_{\Theta}(B(T)) =\Phi_\Theta(h_T(A(T)))$.
From Proposition \ref{prop: gauge fixing on S^3: flat connection} (ii), we have 
$g_{+,\bm{B}}=h' g_{+,\bm{A}}h_T^{-1}$, where $h'$ is a constant gauge transformation of $E_T$ ($h'(\Theta) =\Theta$).
Set $[g_{-,\bm{A}}]:=\Phi_{A_1}(A(-T))$ and $[g_{-,\bm{B}}]:=\Phi_{A_1}(B(-T)) = \Phi_{A_1}(h_{-T}(A(-T)))$.
We have $g_{-,\bm{B}} = \pm g_{-,\bm{A}} h_{-T}^{-1}$.
We define a gauge transformation $g:\bm{E}'_{\bm{A}}\to \bm{E}'_{\bm{B}}$ by the following way:
Over the region $|t|<T+\delta_1/4$, we set 
$g:=h$ on $\bm{E}|_{|t|<T+\delta_1/4}$.
Over the region $t>T$, we set $g:=h'$ on $E_T\times (T,+\infty)$.
Over the region $t<-T$, we set $g:=\pm 1$ on $E_{-T}\times (-\infty,-T)$.
We have $g(\bm{A}') = \bm{B}'$. Indeed, over the region $t>T$,
\[ h'((1-\rho)g_{+,\bm{A}}(A(t))+\rho \Theta) = (1-\rho)g_{+,\bm{B}}(B(t))+\rho \Theta,   \quad 
(\rho = \rho(t-T)),\]
because $h'g_{+,\bm{A}}u_{+,\bm{A}}= g_{+,\bm{B}}u_{+,\bm{B}}h$ and $h'(\Theta) = \Theta$.
\end{proof}
\begin{lemma} \label{lemma: result of cut-off}
\[ |F^+(\bm{A}')|\leq 5+7d+d^2  \quad \text{on $T\leq |t|\leq T+\delta_1$}.\]
\end{lemma}
\begin{proof}
We consider the case $T<t\leq T+\delta_1$ where $\hat{A}_{i, T} = A_i$.
We have $\bm{A}' = (1-\rho)g_+(A(t)) + \rho A_i$, $\rho = \rho(t-T)$.
Set $a := A_i - g_+(A(t))$. Then $\bm{A}' = g_+(A(t)) + \rho a$.
We have 
\[
F^+(\bm{A}') = (\rho'dt\wedge a)^+ + \frac{\rho}{2}(F(A_i) + *_3 F(A_i)\wedge dt) + 
(\rho^2-\rho)(a\wedge a)^+.
\]
We have $|F(A_i)|\leq d$ and $|\rho'|\leq 4/\delta_1$.
From (\ref{eq: estimate of deviation in cut-off}), 
$|A_i -g_+(A(T))|\leq \delta_1$.
From the ASD equation (\ref{eq: parabolic ASD equation}) and $|F(\bm{A})|\leq d$, 
$|A(t)-A(T)|\leq d |t-T|\leq d\delta_1$.
Hence 
\begin{equation} \label{eq: bound on a by ASD equation}
 |a|\leq |A_i -g_+(A(T))| + |g_+(A(T))-g_+(A(t))| \leq (1+d)\delta_1 \quad (T\leq  t \leq T+\delta_1).
\end{equation}
Therefore, for $T\leq t\leq T+\delta_1$,
\[ |F^+(\bm{A}')|\leq 4(1+d) + d + (1+d)^2 = 5+7d +d^2.\]
\end{proof}
\begin{proposition} \label{prop: conclusion of cut-off}
$F(\bm{A}') =0$ over $|t|\geq T+1$, and $F^+(\bm{A}')$ is supported in $\{T< |t| <T+1\}$.
We have $|F(\bm{A}')|\leq d$ over $|t|\leq T$, and
\begin{equation}\label{eq: norms of error terms}
 \norm{F^+(\bm{A}')}_{L^\infty}\leq d',  \quad
 \norm{F^+(\bm{A}')}_{\mathrm{T}}\leq \varepsilon_0 = 1/(1000), 
\end{equation}
where $d' = d'(d)$ is a positive constant depending only on $d$.
Moreover 
\[ \frac{1}{8\pi^2} \int_X tr(F(\bm{A}')^2) 
\leq \frac{1}{8\pi^2} \int_{|t|\leq T} |F(\bm{A})|^2 d\vol + C_1(d)
\leq \frac{2Td^2\vol(S^3)}{8\pi^2} + C_1(d) .\]
Here $C_1(d)$ depends only on $d$.
\end{proposition}
\begin{proof}
The statements about the supports of $F(\bm{A}')$ and $F^+(\bm{A}')$ are obvious by the construction.
Since $\bm{A}'=\bm{A}$ over $|t|\leq T$, $|F(\bm{A}')|\leq d$ over $|t|\leq T$.
We have $\bm{A}' = \hat{A}_{i,T}$ for $t\geq T+\delta_1$ and $\bm{A}' = \hat{A}_{j,-T}$ for $t\leq -T-\delta_1$. 
Hence (from Lemma \ref{lemma: result of cut-off})
\[ \norm{F^+(\bm{A}')}_{L^\infty}\leq d' := \max\left(5+7d+d^2, 
\norm{F^+(\hat{A}_1)}_{L^\infty},\norm{F^+(\hat{A}_2)}_{L^\infty}
 \cdots, \norm{F^+(\hat{A}_N)}_{L^\infty}\right).\]
By using Lemma \ref{lemma: result of cut-off}, (\ref{eq: choice of delta_1}) and Proposition \ref{prop: gluing instantons} (iii)
(note that $g(x,y)$ is invariant under the translations $t \mapsto t-T$ and $t\mapsto t+T$),
\[ \norm{F^+(\bm{A}')}_{\mathrm{T}} \leq 2(5+7d+d^2)\delta_1' + \varepsilon_0/2 \leq \varepsilon_0.\]

We have $\bm{A}'= \bm{A}$ over $|t|\leq T$ and 
\[ F(\bm{A}') = (1-\rho)g_+\circ u_+(F(\bm{A})) + \rho F(A_i) +\rho'dt\wedge a + (\rho^2-\rho)a^2,\]
over $T<t<T+\delta_1$. 
Hence $|F(\bm{A}')|\leq \const_d$ over $T<|t|<T+\delta_1$ by using (\ref{eq: bound on a by ASD equation}).
Then the last statement can be easily proved.
\end{proof}

\subsection{Continuity of the cut-off} \label{subsection: continuity of the cut-off}
Fix $0\leq i,j\leq N$.
Let $[\bm{A}_n]$ $(n\geq 1)$ be a sequence in $\moduli_{d,T}(i,j)$ converging 
to $[\bm{A}] \in \moduli_{d,T}(i,j)$ in the $\mathcal{C}^\infty$-topology over every compact subset in $X$.
Let $[\bm{E}'_n,\bm{A}'_n]$ (respectively $[\bm{E}',\bm{A}']$) be the gauge equivalence classes of the connections
constructed by cutting off $[\bm{A}_n]$ (respectively $[\bm{A}]$) 
as in Section \ref{subsection: cut-off construction}.
\begin{lemma} \label{lemma: continuity in the cut-off}
There are gauge transformations $h_n: \bm{E}'_n\to \bm{E}'$ $(n\gg 1)$ such that $h_n(\bm{A}'_n) = \bm{A}'$ for $|t|\geq T+1$ and
$h_n(\bm{A}'_n)$ converges to $\bm{A}'$ in the $\mathcal{C}^\infty$-topology over $X$.
(Indeed, we will need only $\mathcal{C}^1$-convergence in the later argument)
\end{lemma}
\begin{proof}
We can suppose that $\bm{A}_n$ converges to $\bm{A}$ in the $\mathcal{C}^\infty$-topology over $|t|\leq T+2$.
Let $u_{+,n}:\bm{E}|_{t\geq T}\to E_T\times [T, +\infty)$ (resp. $u_+$) 
be the temporal gauge of $\bm{A}_n$ (resp. $\bm{A}$), and 
set $A_n(t) := u_{+,n}(\bm{A}_n)$ and $A(t) := u_+(\bm{A})$ for $t\geq T$.
We set $[g_{+,n}]:=\Phi_{A_i}(A_n(T))$ and $[g_+] := \Phi_{A_i}(A(T))$ if $i>0$, and
we set $g_{+,n}:=\Phi_{\Theta}(A_n(T))$ and $g_+:=\Phi_{\Theta}(A(T))$ if $i=0$. 

$u_{+,n}$ converges to $u_+$ in the $\mathcal{C}^\infty$-topology over $T\leq t\leq T+1$, and 
we can suppose that $g_{+,n}$ converges to $g_+$ in the $\mathcal{C}^\infty$-topology.
Hence there are $\chi_n \in \Gamma(S^3\times [T,T+1], \ad E_T\times [T,T+1])$ $(n\gg 1)$ satisfying 
$g_{+}\circ u_{+} = e^{\chi_n}g_{+,n}\circ u_{+,n}$.
$\chi_n \to 0$ in the $\mathcal{C}^\infty$-topology over $T\leq t\leq T+1$.
Let $\varphi$ be a smooth function on $X$ 
such that $0\leq \varphi \leq 1$, $\varphi = 1$ over $t\leq T+\delta_1$
and $\varphi=0$ over $t\geq T+1$.
We define $h_n:\bm{E}'_n\to \bm{E}'$ $(n\gg 1)$ as follows.

\noindent
(i) In the case of $|t|<T+\delta_1/4$, we set $h_n :=\mathrm{id}:\bm{E}\to \bm{E}$.

\noindent 
(ii) In the case of $t>T$, we set 
$h_n := e^{\varphi\chi_n}:E_T\times (T,+\infty)\to E_T\times (T,+\infty)$.
This is compatible with the case (i).

\noindent 
(iii) In the case of $t<-T$, we define $h_n: E_{-T}\times (-\infty,-T)\to E_{-T}\times (-\infty,-T)$ in the same way
as in the above (ii).

Then we can easily check that these $h_n$ satisfy the required properties.
\end{proof}

\section{Proofs of the upper bounds} \label{Section: proof of the upper bound}
\subsection{Proof of $\dim(\moduli_d:\mathbb{R}) <\infty$} \label{subsection: proof of mean dim. < infty}
As in Section \ref{section: introduction}, $\bm{E} = X\times SU(2)$ and 
$\moduli_d$ $(d\geq 0)$ is the space of all gauge equivalence classes $[\bm{A}]$
where $\bm{A}$ is an ASD connection on $\bm{E}$ satisfying $\norm{F(\bm{A})}_{L^\infty}\leq d$.
We define a distance on $\moduli_d$ as follows.
For $[\bm{A}], [\bm{B}]\in \moduli_d$, we set 
\begin{equation*}
 \dist([\bm{A}], [\bm{B}]) 
 := \inf_{g:\bm{E}\to \bm{E}} 
 \left\{\sum_{n\geq 1} 2^{-n}\frac{\norm{g(\bm{A})-\bm{B}}_{L^\infty(|t|\leq n)}}
 {1+ \norm{g(\bm{A})-\bm{B}}_{L^\infty(|t|\leq n)}} \right\},
\end{equation*}
where $g$ runs over all gauge transformations of $\bm{E}$, and $|t|\leq n$ means the region 
$\{(\theta, t)\in S^3\times \mathbb{R}|\, |t|\leq n\}$.
This distance is compatible with the topology of $\moduli_d$ introduced in Section \ref{section: introduction}.
For $R=1,2,3,\cdots$, we define an amenable sequence $\Omega_R \subset \mathbb{R}$ by 
$\Omega_R =\{s\in \mathbb{R}|\, -R\leq s\leq R\}$.
We define $\dist_{\Omega_R}([\bm{A}], [\bm{B}])$ 
as in Section \ref{subsection: Review of mean dimension}, i.e., 
\[ \dist_{\Omega_R}([\bm{A}],[\bm{B}]) 
 := \sup_{s\in \Omega_R}\dist([s^*\bm{A}], [s^*\bm{B}]),\]
where $s^*\bm{A}$ is the pull-back of $\bm{A}$ by $s:\bm{E}\to \bm{E}$.

Let $\varepsilon>0$. 
We take a positive integer $L = L(\varepsilon)$ so that 
\begin{equation}\label{eq: definition of L}
 \sum_{n>L}2^{-n} < \varepsilon/2.
\end{equation}
We define $D=D(d, d', \varepsilon/4)$ as the positive number introduced in Lemma \ref{lemma: interior estimate},
where $d'=d'(d)$ is the positive constant introduced in Proposition \ref{prop: conclusion of cut-off}.
We set $T=T(R,d, \varepsilon) =R+L+D >0$. 

We have the decomposition $\moduli_d = \bigcup_{0 \leq i, j\leq N}\moduli_{d,T}(i, j)$
$(N = N(d))$ as in Section \ref{subsection: cut-off construction}. 
$\moduli_{d,T}(i,j)$ is the space of $[\bm{A}]\in \moduli_d$ satisfying 
$[\bm{A}|_{S^3\times\{T\}}]\in U_{A_i}$ and 
$[\bm{A}|_{S^3\times\{-T\}}]\in U_{A_j}$.
Fix $0 \leq i, j\leq N$.
Let $\bm{A}$ be an ASD connection on $\bm{E}$ satisfying $[\bm{A}]\in \moduli_{d,T}(i,j)$.
By the cut-off construction in Section \ref{subsection: cut-off construction},
we have constructed $(\bm{E}', \bm{A}')$
satisfying the following conditions (see Proposition \ref{prop: conclusion of cut-off}).
$\bm{E}'$ is a principal $SU(2)$-bundle over $X$, and $\bm{A}'$ is a connection on $\bm{E}'$ such that 
$F(\bm{A}') =0$ for $|t|\geq T+1$, $F^+(\bm{A}')$ is supported in $\{T<|t|<T+1\}$, and that
\[ \norm{F^+(\bm{A}')}_{\mathrm{T}}\leq \varepsilon_0, \quad \norm{F^+(\bm{A}')}_{L^\infty}\leq d' ,\quad
    \norm{F(\bm{A}')}_{L^\infty(|t|\leq T)} \leq d.\]
We can identify $\bm{E}'$ with $\bm{E}$ over $|t|<T+\delta_1/4$ by the definition, and 
\begin{equation} \label{eq: equation of (E,A,p) and (E',A',p')}
  \bm{A}'|_{|t|<T+\delta_1/4} = \bm{A}|_{|t|<T+\delta_1/4}.
\end{equation}
$(\bm{E}', \bm{A}')$ satisfies the conditions (i), (ii), (iii) in the beginning of 
Section \ref{subsection: construction}.
Therefore, by using the perturbation argument in Section \ref{section: solving ASD equation} 
(see Proposition \ref{prop: conclusion of the perturbation}),
we can construct the ASD connection $\bm{A}'' := \bm{A}' + d^*_{\bm{A}'}\phi_{\bm{A}'}$ on $\bm{E}'$.
By Lemma \ref{lemma: interior estimate},
\begin{equation} \label{eq: error over the interior region}
 |\bm{A}-\bm{A}''| = |\bm{A}'-\bm{A}''| \leq \varepsilon/4 \quad (|t|\leq T-D = R+L).
\end{equation}
From Proposition \ref{prop: conclusion of cut-off}
and Proposition \ref{prop: conclusion of the perturbation} (b),
\begin{equation}  \label{eq: energy bound} 
 \begin{split}
 \frac{1}{8\pi^2}\int_X |F(\bm{A}'')|^2d\vol &= \frac{1}{8\pi^2}\int_X tr(F(\bm{A}')^2) \\
    &\leq \frac{1}{8\pi^2}\int_{|t|\leq T} |F(\bm{A})|^2 d\vol + C_1(d) 
    \leq \frac{2Td^2\vol(S^3)}{8\pi^2} + C_1(d) ,
 \end{split}
\end{equation}
where $C_1(d)$ is a positive constant depending only on $d$.
Since the cut-off and perturbation constructions respect the gauge symmetry
(see Proposition \ref{prop: conclusion of the perturbation} and Lemma \ref{lemma: gauge equivariance of cut-off}), 
the gauge equivalence class $[\bm{E}',\bm{A}'']$ depends only on the gauge equivalence class $[\bm{A}]$.
We set $F_{i,j}([\bm{A}]) := [\bm{E}',\bm{A}'']$.

For $c\geq 0$, we define $M(c)$ as the space of all gauge equivalence classes $[E, A]$ 
satisfying the following.
$E$ is a principal $SU(2)$-bundle over $X$, and $A$ is an ASD connection on $E$ satisfying 
\[ \frac{1}{8\pi^2}\int_X |F_A|^2d\vol \leq c.  \]
The topology of $M(c)$ is defined as follows. 
A sequence $[E_n,A_n]\in M(c)$ $(n\geq 1)$ converges to $[E,A]\in M(c)$ if 
the following two conditions are satisfied:

\noindent 
(i) $\int_X |F(A_n)|^2d\vol = \int_X |F(A)|^2d\vol$ for $n\gg 1$.

\noindent 
(ii) There are gauge transformations $g_n:E_n\to E$ $(n\gg 1)$ such that for any compact set $K\subset X$ 
we have $\norm{g_n(A_n)-A}_{\mathcal{C}^0(K)} \to 0$.

Using the index theorem, we have 
\begin{equation}\label{eq: index theorem}
 \dim M(c) \leq 8c.
\end{equation}
Here $\dim M(c)$ denotes the topological covering dimension of $M(c)$.
By (\ref{eq: energy bound}), we get the map
\[ F_{i,j}: \moduli_{d,T}(i,j) \to M\left(\frac{2Td^2\vol(S^3)}{8\pi^2} + C_1(d)\right), \quad
  [\bm{A}]\mapsto [\bm{E}', \bm{A}'']  .\]
\begin{lemma} \label{lemma: varepsilon-embedding}
For $[\bm{A}_1]$ and $[\bm{A}_2]$ in $\moduli_{d,T}(i, j)$,
if $F_{i,j}([\bm{A}_1]) = F_{i,j}([\bm{A}_2])$, then 
\[ \dist_{\Omega_R}([\bm{A}_1],[\bm{A}_2]) < \varepsilon.\]
\end{lemma}
\begin{proof}
From (\ref{eq: equation of (E,A,p) and (E',A',p')}) and (\ref{eq: error over the interior region}),
there exists a gauge transformation $g$ of $\bm{E}$ defined over $|t| <T+\delta_1/4$ such that 
$|g(\bm{A}_1)-\bm{A}_2|\leq \varepsilon/2$ over $|t|\leq R+L$.
There exists a gauge transformation $\tilde{g}$ of $\bm{E}$ defined all over $X$ satisfying 
$\tilde{g}=g$ on $|t|\leq T = R+L+D$.
Then we have $|\tilde{g}(\bm{A}_1) -\bm{A}_2|\leq \varepsilon/2$ on $|t|\leq R+L$.
For $s \in \Omega_R$ (i.e. $|s|\leq R$), by using (\ref{eq: definition of L}), 
\begin{equation*}
 \begin{split}
 \dist([s^*\bm{A}_1], [s^*\bm{A}_2]) 
 &\leq \sum_{n\geq 1} 2^{-n}\frac{\norm{\tilde{g}(\bm{A}_1)-\bm{A}_2}_{L^\infty(|t-s|\leq n)}}{1
 +\norm{\tilde{g}(\bm{A}_1)-\bm{A}_2}_{L^\infty(|t-s|\leq n)}} \\
 &\leq\sum_{n=1}^L 2^{-n}(\varepsilon/2) + \sum_{n>L} 2^{-n} < \varepsilon/2+\varepsilon/2  = \varepsilon.
 \end{split}
\end{equation*}
\end{proof}
\begin{lemma} \label{lemma: continuity of F_{i,j}}
The map $F_{i,j}:\moduli_{d,T}(i,j)\to M\left(\frac{2Td^2\vol(S^3)}{8\pi^2} + C_1(d)\right)$ is continuous.
\end{lemma}
\begin{proof}
Let $[\bm{A}_n]\in \moduli_{d,T}(i,j)$ 
be a sequence converging to $[\bm{A}]\in \moduli_{d,T}(i,j)$.
From Lemma \ref{lemma: continuity in the cut-off}, 
there are gauge transformations $h_n:\bm{E}'_n \to \bm{E}'$ $(n\gg 1)$ such that
$h_n(\bm{A}'_n) = \bm{A}'$ over $|t|\geq T+1$ and 
that $h_n(\bm{A}'_n)$ converges to $\bm{A}'$ in the $\mathcal{C}^\infty$-topology over $X$.
Since the perturbation construction in Section \ref{section: solving ASD equation} is gauge equivariant 
(Proposition \ref{prop: conclusion of the perturbation}), we have 
\[ (\bm{E}', h_n(\bm{A}'_n) + d_{h_n(\bm{A}'_n)}^*\phi_{h_n(\bm{A}'_n)})  
    = (h_n(\bm{E}'_n), h_n(\bm{A}''_n)).\]
From Proposition \ref{prop: continuity of the perturbation in C^0},
$d_{h_n(\bm{A}'_n)}^*\phi_{h_n(\bm{A}'_n)}$ converges to $d_{\bm{A}'}^*\phi_{\bm{A}'}$ in the 
$\mathcal{C}^0$-topology over every compact subset in $X$ and 
\[ \int_X |F(h_n(\bm{A}'_n) + d_{h_n(\bm{A}'_n)}^*\phi_{h_n(\bm{A}'_n)})|^2d\vol 
   = \int_X |F(\bm{A}' + d_{\bm{A}'}^*\phi_{\bm{A}'})|^2d\vol  \quad \text{for $n\gg 1$}.\]
This shows that the sequence $[\bm{E}'_n, \bm{A}''_n] = 
[\bm{E}', h_n(\bm{A}'_n) + d_{h_n(\bm{A}'_n)}^*\phi_{h_n(\bm{A}'_n)}]$ $(n\geq 1)$ converges 
to $[\bm{E}',\bm{A}''] = [\bm{E}',\bm{A}'+d^*_{\bm{A}'}\phi_{\bm{A}'}]$ 
in $M_T\left(\frac{2Td^2\vol(S^3)}{8\pi^2} + C_1(d)\right)$.
\end{proof}
From Lemmas \ref{lemma: varepsilon-embedding} and \ref{lemma: continuity of F_{i,j}}, 
$F_{i,j}$ becomes an $\varepsilon$-embedding with respect to the distance $\dist_{\Omega_R}$. Hence 
\[ \widim_\varepsilon(\moduli_{d,T}(i,j),\dist_{\Omega_R})
 \leq \dim M\left(\frac{2Td^2\vol(S^3)}{8\pi^2} + C_1(d)\right).\]
Since $\moduli_d=\bigcup_{0 \leq i,j\leq N}\moduli_{d,T}(i,j)$ (each $\moduli_{d,T}(i,j)$ is a closed set),
by using Lemma \ref{lemma: widim and space sum}, we get
\[ \widim_\varepsilon(\moduli_d, \dist_{\Omega_R})
\leq (N+1)^2\dim M\left(\frac{2Td^2\vol(S^3)}{8\pi^2} + C_1(d)\right) +(N+1)^2-1.\]
From (\ref{eq: index theorem}) and $T=R+L+D$,
\[ \dim M \left(\frac{2Td^2\vol(S^3)}{8\pi^2} + C_1(d)\right) 
\leq \frac{2(R+L+D)d^2\vol(S^3)}{\pi^2} + 8C_1(d).\]
Since $N=N(d)$, $L=L(\varepsilon)$, $D=D(d,d'(d),\varepsilon/4)$ are independent of $R$, we get 
\[ \widim_\varepsilon(\moduli_d:\mathbb{R}) = 
\lim_{R\to \infty}\frac{\widim_\varepsilon (\moduli_d,\dist_{\Omega_R})}{|\Omega_R|}
\leq \frac{(N+1)^2d^2\vol(S^3)}{\pi^2}.\]
This holds for any $\varepsilon>0$. Thus 
\[ \dim(\moduli_d:\mathbb{R}) = \lim_{\varepsilon\to 0}\widim_\varepsilon(\moduli_d:\mathbb{R})
  \leq  \frac{(N+1)^2d^2\vol(S^3)}{\pi^2} < \infty.\]

\subsection{Upper bound on the local mean dimension} \label{subsection: upper bound on the local mean dimension}
\begin{lemma}\label{lemma: Lebesgue number}
There exists $r_1=r_1(d)>0$ satisfying the following. 
For any $[\bm{A}]\in \moduli_d$ and $s\in \mathbb{R}$, 
there exists an integer $i$ $(0\leq i\leq N)$ such that if $[\bm{B}]\in \moduli_d$ satisfies
$\dist_{\mathbb{R}}([\bm{A}],[\bm{B}])\leq r_1$ then 
\[ [\bm{B}|_{S^3\times \{s\}}]\in U_{A_i}.\]
Here we identifies $\bm{E}|_{S^3\times\{s\}}$ with $F$, and $U_{A_i}$ is the closed set introduced in 
Section \ref{subsection: gauge fixing on S^3}.
Recall $\dist_{\mathbb{R}}([\bm{A}],[\bm{B}]) 
= \sup_{s\in \mathbb{R}}\dist([s^*\bm{A}],[s^*\bm{B}])$.
\end{lemma}
\begin{proof}
There exists $r_1>0$ (the Lebesgue number) satisfying the following.
For any $[\bm{A}]\in \moduli_d$, there exists $i = i([\bm{A}])$ such that if 
$[\bm{B}]\in \moduli_d$ satisfies $\dist([\bm{A}],[\bm{B}])\leq r_1$ then 
$[\bm{B}|_{S^3\times\{0\}}]\in U_{A_i}$.
If $\dist_{\mathbb{R}}([\bm{A}],[\bm{B}])\leq r_1$, then for each $s\in \mathbb{R}$ we have 
$\dist([s^*\bm{A}],[s^*\bm{B}])\leq r_1$ and hence 
\[ [\bm{B}|_{S^3\times \{s\}}] = [(s^*\bm{B})|_{S^3\times \{0\}}]
\in U_{A_i},\]
for $i= i([s^*\bm{A}])$.
\end{proof}
\begin{lemma} \label{lemma: the difference of the curvatures is small}
For any $\varepsilon'>0$, there exists $r_2=r_2(\varepsilon')>0$ such that if $[\bm{A}]$ and $[\bm{B}]$
in $\moduli_d$ satisfy $\dist_{\mathbb{R}}([\bm{A}], [\bm{B}])\leq r_2$ then
\[ \norm{|F(\bm{A})|^2-|F(\bm{B})|^2}_{L^\infty(X)}\leq \varepsilon'.\]
\end{lemma}
\begin{proof}
The map $\moduli_d\ni [\bm{A}]\mapsto |F(\bm{A})|^2 \in \mathcal{C}^0(S^3\times [0,1])$ is continuous.
Hence there exists $r_2>0$ such that if $\dist([\bm{A}],[\bm{B}])\leq r_2$ then 
\[ \norm{|F(\bm{A})|^2-|F(\bm{B})|^2}_{L^\infty(S^3\times [0,1])} \leq \varepsilon'.\]
Then for each $s\in \mathbb{R}$, if $\dist([s^*\bm{A}],[s^*\bm{B}])\leq r_2$, 
\[ \norm{|F(\bm{A})|^2-|F(\bm{B})|^2}_{L^\infty(S^3\times [s,s+1])} \leq \varepsilon'.\]
Therefore if $\dist_{\mathbb{R}}([\bm{A}],[\bm{B}])\leq r_2$, then 
$\norm{|F(\bm{A})|^2-|F(\bm{B})|^2}_{L^\infty(X)}\leq \varepsilon'$.
\end{proof}
Let $[\bm{A}]\in \moduli_d$, and $\varepsilon, \varepsilon'>0$ be arbitrary two positive numbers.
There exists $T_0=T_0([\bm{A}],\varepsilon')>0$ such that for any $T_1\geq T_0$
\[ \frac{1}{8\pi^2 T_1}\sup_{t\in \mathbb{R}}\int_{S^3\times [t,t+T_1]}|F(\bm{A})|^2 d\vol \leq \rho(\bm{A}) +\varepsilon'/2.\]
The important point for the later argument is the following:
We can arrange $T_0$ so that $T_0([s^*\bm{A}],\varepsilon')=T_0([\bm{A}],\varepsilon')$
for all $s\in \mathbb{R}$.
We set
\[ r=r(d,\varepsilon') = \min\left(r_1(d), r_2\left(\frac{4\pi^2\varepsilon'}{\vol(S^3)}\right)\right),\]
where $r_1(\cdot)$ and $r_2(\cdot)$ are the positive constants introduced in Lemmas \ref{lemma: Lebesgue number} and
\ref{lemma: the difference of the curvatures is small}.
By Lemma \ref{lemma: the difference of the curvatures is small}, 
if $[\bm{B}]\in B_r([\bm{A}])_{\mathbb{R}}$ 
(the closed ball of radius $r$ centered at $[\bm{A}]$ in $\moduli_d$ with respect to the distance $\dist_{\mathbb{R}}$), 
then for any $T_1\geq T_0$
\begin{equation}\label{eq: bound of the energy of [B,q] around [A,p]}
 \frac{1}{8\pi^2T_1}\sup_{t\in \mathbb{R}}\int_{S^3\times [t,t+T_1]}|F(\bm{B})|^2d\vol 
 \leq \rho(\bm{A})+\varepsilon'/2 +\varepsilon'/2 = \rho(\bm{A})+\varepsilon'.
\end{equation}

We define positive numbers $L=L(\varepsilon)$ and $D=D(d,d'(d),\varepsilon/4)$ as in the previous subsection.
($L=L(\varepsilon)$ is a positive integer satisfying (\ref{eq: definition of L}), and
$D=D(d,d'(d),\varepsilon/4)$ is the positive number introduced in Lemma \ref{lemma: interior estimate}.)
Let $R$ be an integer with $R\geq T_0$, and set $T:=R+L+D$.
By Lemma \ref{lemma: Lebesgue number}, there exist $i,j$ $(0\leq i, j\leq N)$ depending on 
$[\bm{A}]$ and $T$ such that all $[\bm{B}]\in B_r([\bm{A}])_{\mathbb{R}}$
satisfy $[\bm{B}|_{S^3\times \{T\}}]\in U_{A_i}$ and 
$[\bm{B}|_{S^3\times\{-T\}}]\in U_{A_j}$.
(That is, $B_r([\bm{A}])_{\mathbb{R}}\subset \moduli_{d,T}(i,j)$.)

As in the previous subsection, by using the cut-off construction and perturbation, for each 
$[\bm{B}]\in B_r([\bm{A}])_{\mathbb{R}}$ we can construct the ASD connection
$[\bm{E}',\bm{B}'']$.
By (\ref{eq: energy bound}), (\ref{eq: bound of the energy of [B,q] around [A,p]}) and $T\geq T_0$, 
\[ \frac{1}{8\pi^2}\int_X |F(\bm{B}'')|^2d\vol 
\leq \frac{1}{8\pi^2}\int_{|t|\leq T}|F(\bm{B})|^2d\vol + C_1(d)
\leq 2T(\rho(\bm{A})+\varepsilon')+C_1(d), \]
where $C_1(d)$ depends only on $d$.
Therefore we get the map
\[ B_r([\bm{A}])_{\mathbb{R}}\to M(2T(\rho(\bm{A})+\varepsilon')+C_1(d)),
\quad [\bm{B}]\mapsto [\bm{E}',\bm{B}''].\]
This is an $\varepsilon$-embedding with respect to the distance $\dist_{\Omega_R}$ by Lemmas
\ref{lemma: varepsilon-embedding} and \ref{lemma: continuity of F_{i,j}}.
Therefore we get (by (\ref{eq: index theorem}))
\[ \widim_\varepsilon(B_r([\bm{A}])_{\mathbb{R}},\dist_{\Omega_R})
\leq 16T(\rho(\bm{A})+\varepsilon')+8C_1(d), \]
for $R\geq T_0([\bm{A}],\varepsilon')$ and $r=r(d,\varepsilon')$.
As we pointed out before, we have $T_0([s^*\bm{A}],\varepsilon')=T_0([\bm{A}],\varepsilon')$ 
for $s\in \mathbb{R}$.
Hence for all $s \in \mathbb{R}$ and $R\geq T_0 = T_0([\bm{A}],\varepsilon')$,
we have the same upper bound on 
$\widim_\varepsilon(B_r([s^*\bm{A}])_{\mathbb{R}},\dist_{\Omega_R})$.
Then for $R\geq T_0$,
\[ \frac{1}{|\Omega_R|}\sup_{s\in \mathbb{R}}
\widim_\varepsilon(B_r([s^*\bm{A}])_{\mathbb{R}},\dist_{\Omega_R}) \leq 
\frac{16T(\rho(\bm{A})+\varepsilon')+8C_1(d)}{2R}.\]
$T=R+L+D$. $L=L(\varepsilon)$ and $D=D(d,d'(d),\varepsilon/4)$ are independent of $R$. 
Hence
\begin{equation*}
 \begin{split}
  \widim_\varepsilon(B_r([\bm{A}])_{\mathbb{R}}\subset \moduli_d:\mathbb{R}) 
  &=\lim_{R\to \infty}\left(\frac{1}{|\Omega_R|}\sup_{s\in \mathbb{R}} 
  \widim_\varepsilon(B_r([s^*\bm{A}])_{\mathbb{R}},\dist_{\Omega_R})\right) \\
  &\leq 8(\rho(\bm{A})+\varepsilon').
 \end{split}
\end{equation*}
Here we have used (\ref{eq: one description of local mean widim}).
This holds for any $\varepsilon>0$. (Note that $r=r(d,\varepsilon')$ is independent of $\varepsilon$.)
Hence 
\begin{equation*}
 \dim(B_r([\bm{A}])_{\mathbb{R}}\subset \moduli_d:\mathbb{R}) 
 = \lim_{\varepsilon\to 0}\widim_\varepsilon(B_r([\bm{A}])_{\mathbb{R}}\subset \moduli_d:\mathbb{R}) 
 \leq 8(\rho(\bm{A})+\varepsilon').
\end{equation*}
Since $\dim_{[\bm{A}]}(\moduli_d:\mathbb{R})\leq \dim(B_r([\bm{A}])_{\mathbb{R}}\subset \moduli_d:\mathbb{R})$,
\[ \dim_{[\bm{A}]}(\moduli_d:\mathbb{R})\leq  8(\rho(\bm{A})+\varepsilon').\]
This holds for any $\varepsilon'>0$. Thus 
\[ \dim_{[\bm{A}]}(\moduli_d:\mathbb{R})\leq 8\rho(\bm{A}).\]
Therefore we get the conclusion:
\begin{theorem}
For any $[\bm{A}]\in \moduli_d$, 
\[\dim_{[\bm{A}]}(\moduli_d:\mathbb{R})\leq 8\rho(\bm{A}). \]
\end{theorem}

\section{Analytic preliminaries for the lower bound}\label{section: Analytic preliminaries for the lower bound}
Let $T>0$ be a positive real number, $\underbar{E}$ be a principal $SU(2)$-bundle over 
$S^3\times (\mathbb{R}/T\mathbb{Z})$,
and $\underbar{A}$ be an ASD connection on $\underbar{E}$. Suppose $\underbar{A}$ is not flat.
Let $\pi:S^3\times \mathbb{R}\to S^3\times (\mathbb{R}/T\mathbb{Z})$ be the natural projection, and 
$E :=\pi^*\underbar{E}$ and $A:=\pi^*\underbar{A}$ be the pull-backs.
Obviously $A$ is a non-flat ASD connection satisfying $\norm{F_A}_{L^\infty} <\infty$. 
Hence it is irreducible (Corollary \ref{cor: non-flat implies irreducible}).
Some constants introduced below (e.g. $C_2$, $C_3$, $\varepsilon_1$, $\varepsilon_2$) will depend on 
$(\underbar{E}, \underbar{A})$. But we consider that $(\underbar{E},\underbar{A})$ is fixed, and hence 
the dependence on it will not be explicitly written.
\begin{lemma} \label{lemma: uniform irreducibility}
There exists $C_2>0$ such that for any $u\in \Omega^0(\ad E)$
\[ \int_{S^3\times [0,T]}|u|^2 \leq C_2\int_{S^3\times [0,T]}|d_Au|^2 .\]
Then, from the natural $T$-periodicity of $A$, for every $n\in \mathbb{Z}$
\[ \int_{S^3\times [nT, (n+1)T]}|u|^2 \leq C_2\int_{S^3\times [nT, (n+1)T]}|d_Au|^2.\]
\end{lemma}
\begin{proof}
Since $A$ is ASD and irreducible, the restriction of $A$ to $S^3\times(0, T)$ is also irreducible
(by the unique continuation \cite[Section 4.3.4]{Donaldson-Kronheimer}).
Suppose the above statement is false, then there exist $u_n$ $(n\geq 1)$ such that 
\[ 1 = \int_{S^3\times [0,T]} |u_n|^2 > n\int_{S^3\times [0,T]}|d_A u_n|^2 .\]
If we take a subsequence, 
then the restrictions of $u_n$ to $S^3\times (0,T)$ converge to some $u$ weakly in $L^2_1(S^3\times (0,T))$
and strongly in $L^2(S^3\times (0,T))$.
We have $\norm{u}_{L^2}=1$ (in particular $u\neq 0$) and $d_Au=0$.
This means that $A$ is reducible over $S^3\times (0,T)$. This is a contradiction.
\end{proof}
\begin{lemma} \label{lemma: quantitative irreducibility}
Let $4<q<\infty$. For any $u\in L^q_1(S^3\times (0,T), \Lambda^+(\ad E))$,
\[ \norm{u}_{L^\infty(S^3\times (0,T))} \leq \const_q \norm{d_Au}_{L^q(S^3\times (0,T))}.\]
\end{lemma}
\begin{proof}
Note that the Sobolev embedding 
$L^q_1(S^3\times (0,T))\hookrightarrow \mathcal{C}^0(S^3\times [0,T])$ is a compact operator.
Then this lemma can be proved in the same way as in Lemma \ref{lemma: uniform irreducibility}
\end{proof}
\begin{lemma} \label{lemma: quantitative irreducibility :gauge transformation}
Let $4<q<\infty$. For any gauge transformation $g:E\to E$ and $n\in \mathbb{Z}$,
\[ \min(\norm{g-1}_{L^\infty(S^3\times (nT,(n+1)T))}, \norm{g+1}_{L^\infty(S^3\times(nT,(n+1)T))})
   \leq \const_q \norm{d_Ag}_{L^q(S^3\times (nT,(n+1)T))}.\]
Here $\const_q$ is independent of $g$ and $n$. 
\end{lemma}
\begin{proof}
From the $T$-periodicity of $A$, it is enough to prove the case of $n=0$.
Suppose the statement is false. 
Then there exists a sequence of gauge transformations $\{g_n\}_{n\geq 1}$ satisfying 
\[ \min(\norm{g_n-1}_{L^\infty(S^3\times (0,T))},\norm{g_n+1}_{L^\infty(S^3\times (0,T))}) 
   > n\norm{d_Ag_n}_{L^q(S^3\times (0,T))}.\]
If we take a subsequence, then $g_n$ converges to some $g$ weakly in $L^q_1(S^3\times (0,T))$ and
strongly in $\mathcal{C}^0(S^3\times [0,T])$.
In particular we have $d_Ag=0$. Hence $g=\pm 1$ since $A$ is irreducible.
By multiplying $\pm 1$ to $g_n$, we can assume that $g=1$.
Then there exists $u_n\in L^q_1(S^3\times (0,T), \Lambda^+(\ad E))$ $(n\gg 1)$
satisfying $g_n=e^{u_n}$ and 
$\norm{u_n}_{L^\infty(S^3\times (0,T))}\leq \const \norm{g_n-1}_{L^\infty(S^3\times (0,T))}$.
Then, by using Lemma \ref{lemma: quantitative irreducibility}, we have
\begin{equation*}
 \begin{split}
 \norm{g_n-1}_{L^\infty(S^3\times (0,T))}&\leq \const \norm{u_n}_{L^\infty(S^3\times (0,T))} \\
   &\leq \const'\norm{d_Au_n}_{L^q(S^3\times (0,T))} \leq \const'' \norm{d_Ag_n}_{L^q(S^3\times (0,T))}.
 \end{split}
\end{equation*}
This is a contradiction.
\end{proof}
\begin{lemma}\label{lemma: quantitative irreducibility :gauge transformation 2}
There exists $\varepsilon_1>0$ such that, for any gauge transformation $g:E\to E$,
if $\norm{d_Ag}_{L^\infty(X)}\leq \varepsilon_1$ then
\[ \min(\norm{g-1}_{L^\infty(X)},\norm{g+1}_{L^\infty(X)})\leq \const \norm{d_Ag}_{L^\infty(X)}.\]
\end{lemma}
\begin{proof}
From Lemma \ref{lemma: quantitative irreducibility :gauge transformation}
\[ \min(\norm{g-1}_{L^\infty(S^3\times (nT,(n+1)T))}, \norm{g+1}_{L^\infty(S^3\times(nT,(n+1)T))})
   \leq C \norm{d_Ag}_{L^\infty(X)} \leq C\cdot \varepsilon_1.\]
Suppose $\min(\norm{g-1}_{L^\infty(S^3\times (0,T))}, \norm{g+1}_{L^\infty(S^3\times(0,T))})
=\norm{g-1}_{L^\infty(S^3\times (0,T))}$.
We want to prove that for all $n\in \mathbb{Z}$
\begin{equation}\label{eq: compatibility of minimum}
 \min(\norm{g-1}_{L^\infty(S^3\times (nT,(n+1)T))}, \norm{g+1}_{L^\infty(S^3\times(nT,(n+1)T))})
   = \norm{g-1}_{L^\infty(S^3\times (nT,(n+1)T))}.
\end{equation}
We have $\norm{g-1}_{L^\infty(S^3\times (0,T))} \leq C\cdot \varepsilon_1 \ll 1$.
From $|d_Ag|\leq \varepsilon_1$, $\norm{g-1}_{L^\infty(S^3\times (T,2T))}\leq (C+T)\varepsilon_1$,
and hence $\norm{g+1}_{L^\infty(S^3\times (T,2T))}\geq 2-(C+T)\varepsilon_1$.
We choose $\varepsilon_1>0$ so that $(C+T)\varepsilon_1<1$. 
Then (\ref{eq: compatibility of minimum}) holds for $n=1$.
In the same way, by using induction, we can prove that (\ref{eq: compatibility of minimum}) holds
for all $n\in \mathbb{Z}$.
Then Lemma \ref{lemma: quantitative irreducibility :gauge transformation} implies 
$\norm{g-1}_{L^\infty(X)}\leq C \norm{d_Ag}_{L^\infty(X)}$.
\end{proof}
Let $N>0$ be a large positive integer which will be fixed later, and set $R :=NT$.
Let $\varphi$ be a smooth function on $\mathbb{R}$ such that $0\leq \varphi\leq 1$, 
$\varphi =1$ on $[0,R]$, $\varphi = 0$ over $t\geq 2R$ and $t\leq -R$, and
$|\varphi'|, |\varphi''|\leq 2/R$.
Then for any $u\in \Omega^0(\ad E)$ (not necessarily compact supported), 
\[ \int_{S^3\times [0,R]} |d_Au|^2 
\leq \int_{S^3\times \mathbb{R}} |d_A(\varphi u)|^2 = \int_{S^3\times \mathbb{R}} (\Delta_A(\varphi u), \varphi u).\]
Here $\Delta_A:=\nabla_A^*\nabla_A=-*d_A*d_A$ on $\Omega^0(\ad E)$.
We have $\Delta_A(\varphi u) = \varphi \Delta_Au + \Delta\varphi\cdot u + *(*d\varphi\wedge d_Au -d\varphi\wedge *d_Au)$.
Then $\Delta_A(\varphi u) = \Delta_A u$ over $S^3\times [0,R]$ and
\[ |\Delta_A(\varphi u)|\leq (2/R)|u| + (4/R)|d_Au| + |\Delta_Au|.\]
Hence 
\[  \int_{S^3\times [0,R]} |d_Au|^2 \leq (2/R)\int_{t\in [-R,0]\cup [R,2R]} |u|^2 
+ (4/R)\int_{t\in [-R,0]\cup [R,2R]}|u||d_Au| + \int_{S^3\times [-R,2R]}|\Delta_Au||u| .\]
From Lemma \ref{lemma: uniform irreducibility}, 
\[\int_{t\in [-R,0]\cup [R,2R]} |u|^2 \leq C_2\int_{t\in [-R,0]\cup [R,2R]} |d_Au|^2, \]
\[\int_{t\in [-R,0]\cup [R,2R]}|u||d_Au|  
\leq \sqrt{\int_{t\in [-R,0]\cup [R,2R]} |u|^2} \sqrt{ \int_{t\in [-R,0]\cup[R,2R]} |d_Au|^2}
\leq \sqrt{C_2} \int_{t\in [-R,0]\cup[R,2R]} |d_Au|^2 .\]
Hence 
\[  \int_{S^3\times [0,R]} |d_Au|^2 \leq \frac{2C_2 + 4\sqrt{C_2}}{R} \int_{t\in [-R,0]\cup[R,2R]} |d_Au|^2 
   + \int_{S^3\times [-R,2R]}|\Delta_Au||u| .\]
For a function (or a section of some Riemannian vector bundle) $f$ on $S^3\times \mathbb{R}$ and $p\in [1,\infty]$, we set 
\[ \norm{f}_{\ell^\infty L^p} := \sup_{n\in \mathbb{Z}} \norm{f}_{L^p(S^3\times (nR, (n+1)R))}.\]
Then the above implies 
\[ \int_{S^3\times [0,R]} |d_Au|^2 \leq \frac{4C_2 + 8\sqrt{C_2}}{R}\norm{d_Au}_{\ell^\infty L^2}^2 
   +  3\norm{|\Delta_Au|\cdot |u|}_{\ell^\infty L^1}.\]
In the same way, for any $n\in \mathbb{Z}$,
\[ \int_{S^3\times [nR,(n+1)R]} |d_Au|^2 \leq \frac{4C_2 + 8\sqrt{C_2}}{R}\norm{d_Au}_{\ell^\infty L^2}^2 
   +  3\norm{|\Delta_Au|\cdot |u|}_{\ell^\infty L^1}.\]
Then we have 
\[   \norm{d_Au}_{\ell^\infty L^2}^2   \leq \frac{4C_2 + 8\sqrt{C_2}}{R}\norm{d_Au}_{\ell^\infty L^2}^2 
   +  3\norm{|\Delta_Au|\cdot |u|}_{\ell^\infty L^1}.\]
We fix $N>0$ so that $(4C_2 + 8\sqrt{C_2})/R\leq 1/2$ (recall: $R=NT$).
If $\norm{d_A u}_{\ell^\infty L^2} <\infty$, then we get 
\[ \norm{d_Au}_{\ell^\infty L^2}^2 \leq 6\norm{|\Delta_Au|\cdot |u|}_{\ell^\infty L^1}.\]
From H\"{o}lder's inequality and Lemma \ref{lemma: uniform irreducibility}, 
\[ \norm{|\Delta_Au|\cdot |u|}_{\ell^\infty L^1} \leq \norm{\Delta_Au}_{\ell^\infty L^2}\norm{u}_{\ell^\infty L^2} 
\leq  \sqrt{C_2} \norm{\Delta_Au}_{\ell^\infty L^2} \norm{d_Au}_{\ell^\infty L^2} .\]
Hence $\norm{d_Au}_{\ell^\infty L^2} \leq 6\sqrt{C_2}\norm{\Delta_Au}_{\ell^\infty L^2}$,
and $\norm{u}_{\ell^\infty L^2} \leq \sqrt{C_2}\norm{d_Au}_{\ell^\infty L^2} 
\leq 6C_2\norm{\Delta_Au}_{\ell^\infty L^2}$.
Then we get the following conclusion.
\begin{lemma} \label{lemma: estimate by irreducibility}
There exists a constant $C_3>0$ such that, for any $u\in \Omega^0(\ad E)$ with $\norm{d_A u}_{\ell^\infty L^2} <\infty$,
we have
\[ \norm{u}_{\ell^\infty L^2} + \norm{d_A u}_{\ell^\infty L^2} \leq C_3 \norm{\Delta_A u}_{\ell^\infty L^2}.\]
\end{lemma}
The following result gives the ``partial Coulomb gauge slice" in our situation.
\begin{proposition} \label{prop: Coulomb gauge condition}
There exists $\varepsilon_2>0$ satisfying the following. 
For any $a$ and $b$ in $\Omega^1(\ad E)$ satisfying 
$d_A^*a = d_A^*b=0$ and $\norm{a}_{L^\infty}, \norm{b}_{L^\infty}\leq \varepsilon_2$,
if there is a gauge transformation $g$ of $E$ satisfying $g(A+a) = A+b$
then $a=b$ and $g=\pm 1$. 
\end{proposition}
\begin{proof}
Since $g(A+a) = A+b$, we have 
$d_Ag=ga -b g$. Then we have $|d_Ag|\leq 2\varepsilon_2$.
We choose $\varepsilon_2>0$ so that $2\varepsilon_2\leq \varepsilon_1$.
($\varepsilon_1$ is the positive constant introduced 
in Lemma \ref{lemma: quantitative irreducibility :gauge transformation 2}.)
From Lemma \ref{lemma: quantitative irreducibility :gauge transformation 2}, 
by multiplying $\pm 1$ to $g$, 
we can suppose $\norm{g-1}_{L^\infty} \leq \const\cdot \varepsilon_2 \ll 1$.
Then there exists $u\in \Omega^0(\ad E)$ 
satisfying $g= e^u$ and $\norm{u}_{L^\infty}\leq \const\cdot \varepsilon_2$. 
We have 
\[ d_Ae^u = d_Au + (d_Au\cdot u + ud_Au)/2! + (d_Au \cdot u^2 + ud_Au\cdot u + u^2d_Au)/3! + \cdots.\]
Since $|u|\leq \const \cdot \varepsilon_2\ll 1$,
\[ |d_Ae^u|\geq |d_Au|(2-e^{|u|}) \geq |d_Au|/2 .\]
Hence $|d_Au|\leq 2|d_Ag|\leq 4 \varepsilon_2$. 
In particular, $\norm{d_A u}_{\ell^\infty L^2} <\infty$.
In the same way we get $|d_Ag|\leq 2|d_Au|$, and hence
\begin{equation} \label{eq: bound on the norm of d_A g by the norm of Delta_A u}
 \norm{d_Ag}_{\ell^\infty L^2}\leq 2\norm{d_Au}_{\ell^\infty L^2} \leq 2C_3\norm{\Delta_Au}_{\ell^\infty L^2}.
\end{equation}
Here we have used Lemma \ref{lemma: estimate by irreducibility}.
Since $d_A^*a = d_A^*b =0$ and $d_Ag=ga-bg$, we have 
\[ \Delta_A g = -*d_A*d_Ag = -*(d_Ag\wedge *a + *b \wedge d_Ag) .\]
Therefore 
$\norm{\Delta_A g}_{\ell^\infty L^2}\leq (\norm{a}_{L^\infty}+\norm{b}_{L^\infty})\norm{d_Ag}_{\ell^\infty L^2} <\infty$.
Moreover, by using the above (\ref{eq: bound on the norm of d_A g by the norm of Delta_A u})
and $\norm{a}_{L^\infty}, \norm{b}_{L^\infty}\leq \varepsilon_2$, we get
\begin{equation}\label{eq: in the proof of Coulomb gauge condition}
 \norm{\Delta_Ag}_{\ell^\infty L^2}
   \leq 4C_3\varepsilon_2 \norm{\Delta_Au}_{\ell^\infty L^2}.
\end{equation}

A direct calculation shows $|\Delta_Au^n|\leq n(n-1)|u|^{n-2}|d_Au|^2 +n|u|^{n-1}|\Delta_Au|$.
Hence 
\begin{equation} \label{eq: the difference between Delta_A g and Delta_A u}
 |\Delta_A(e^u-u)|\leq e^{|u|}|d_Au|^2+(e^{|u|}-1)|\Delta_Au| 
 \leq C \varepsilon_2 (|d_Au|+|\Delta_Au|).
\end{equation}
Here we have used $|u|, |d_Au|\leq \const \cdot\varepsilon_2\ll 1$.
Hence $(1-C \varepsilon_2)|\Delta_A u|\leq C \varepsilon_2|d_Au|+|\Delta_Ag|$, 
and $(1-C\varepsilon_2)\norm{\Delta_A u}_{\ell^\infty L^2}\leq C\varepsilon_2\norm{d_Au}_{\ell^\infty L^2}
+\norm{\Delta_A g}_{\ell^\infty L^2} <\infty$. 
We choose $\varepsilon_2>0$ so that $(1-C\varepsilon_2)>0$.
Then $\norm{\Delta_A u}_{\ell^\infty L^2}<\infty$.

The above (\ref{eq: the difference between Delta_A g and Delta_A u}) implies 
\[ \norm{\Delta_Ag-\Delta_Au}_{\ell^\infty L^2}\leq 
    C\varepsilon_2(\norm{d_Au}_{\ell^\infty L^2} + \norm{\Delta_Au}_{\ell^\infty L^2}).\]
Using Lemma \ref{lemma: estimate by irreducibility}, we get 
\[ \norm{\Delta_Ag-\Delta_Au}_{\ell^\infty L^2}
  \leq C' \varepsilon_2\norm{\Delta_Au}_{\ell^\infty L^2}.\]
Then the inequality (\ref{eq: in the proof of Coulomb gauge condition}) gives 
\[ (1- 4C_3\varepsilon_2)\norm{\Delta_Au}_{\ell^\infty L^2}
 \leq C'\varepsilon_2\norm{\Delta_Au}_{\ell^\infty L^2}.\]
If we choose $\varepsilon_2 >0$ so small that $(1-4C_3\varepsilon_2)>C'\varepsilon_2$, then 
this estimate gives $\Delta_Au=0$. (Here we have used $\norm{\Delta_A u}_{\ell^\infty L^2}<\infty$.)
Then we get (from Lemma \ref{lemma: estimate by irreducibility}) $u=0$.
This shows $g=1$ and $a=b$.
\end{proof}
The following ``$L^\infty$-estimate'' will be used in the next section.
For its proof, see Proposition \ref{prop: L^infty-estimate, appendix}
in Appendix \ref{appendix: Green kernel}.
\begin{proposition} \label{L^infty-estimate}
Let $\xi$ be a $\mathcal{C}^2$-section of $\Lambda^+(\ad E)$ over $S^3\times \mathbb{R}$, 
and set $\eta := (\nabla_A^*\nabla_A+S/3)\xi$.
If $\norm{\xi}_{L^\infty}, \norm{\eta}_{L^\infty}<\infty$, then 
\[ \norm{\xi}_{L^\infty}\leq (24/S)\norm{\eta}_{L^\infty}.\]
\end{proposition}

\section{Proof of the lower bound: deformation theory} \label{section: Proof of the lower bound}
The argument in this section is a Yang-Mills analogue of the deformation theory 
developed in Tsukamoto \cite{Tsukamoto-deformation}.
Let $d$ be a positive real number.
As in Section \ref{section: Analytic preliminaries for the lower bound}, 
let $T>0$ be a positive real number,
$\underbar{E}$ be a principal $SU(2)$-bundle over $S^3\times (\mathbb{R}/T\mathbb{Z})$, and $\underbar{A}$ be an ASD connection
on $\underbar{E}$. Suppose that $\underbar{A}$ is not flat and
\begin{equation} \label{eq: strict bound on the curvature}
\norm{F(\underbar{A})}_{L^\infty} < d.
\end{equation}
Set $E:=\pi^*\underbar{E}$ and $A:=\pi^*\underbar{A}$ where $\pi:S^3\times \mathbb{R}\to S^3\times (\mathbb{R}/T\mathbb{Z})$
is the natural projection.
Some constants introduced below depend on $(\underbar{E}, \underbar{A})$.
But we don't explicitly write their dependence on it because we consider that $(\underbar{E},\underbar{A})$ is fixed.

We define the Banach space $H^1_A$ by setting 
\[ H^1_A:=\{a\in \Omega^1(\ad E)|\, (d_A^*+d_A^+)a =0, \, \norm{a}_{L^\infty} <\infty\}.\]
$(H^1_A, \norm{\cdot}_{L^\infty})$ becomes an infinite dimensional Banach space.
The additive group $T\mathbb{Z} = \{nT\in \mathbb{R}|\, n\in \mathbb{Z}\}$ acts on $H^1_A$ as follows.
From the definition of $E$ and $A$, we have $(T^*E, T^*A) = (E, A)$ where 
$T:S^3\times\mathbb{R}\to S^3\times \mathbb{R}, (\theta, t)\mapsto (\theta, t+T)$.
Hence for any $a\in H^1_A$, we have $T^*a\in H^1_A$ and $\norm{T^*a}_{L^\infty} = \norm{a}_{L^\infty}$.

Fix $0<\alpha<1$. We want to define the H\"{o}lder space 
$\mathcal{C}^{k, \alpha}(\Lambda^+(\ad E))$ for $k\geq 0$.
Let $\{U_\lambda\}_{\lambda=1}^\Lambda$, $\{U'_\lambda\}_{\lambda=1}^\Lambda$, 
$\{U''_\lambda\}_{\lambda=1}^\Lambda$ be finite open coverings of $S^3\times (\mathbb{R}/T\mathbb{Z})$
satisfying the following conditions.

\noindent 
(i) $\bar{U}_\lambda \subset U'_\lambda$ and $\bar{U}'_\lambda\subset U''_\lambda$.
$U_\lambda$, $U'_\lambda$ and $U''_\lambda$ are connected, and their boundaries are smooth.
Each $U''_\lambda$ is a coordinate chart, i.e., a diffeomorphism between $U''_\lambda$ and 
an open set in $\mathbb{R}^4$ is given for each $\lambda$.

\noindent 
(ii) The covering map $\pi:S^3\times \mathbb{R}\to S^3\times (\mathbb{R}/T\mathbb{Z})$ can be trivialized over each $U''_\lambda$, i.e., we have a disjoint union 
$\pi^{-1}(U''_\lambda) = \bigsqcup_{n\in \mathbb{Z}} U''_{n\lambda}$ such that
$\pi:U''_{n\lambda}\to U''_\lambda$ is diffeomorphic.
We set $U_{n\lambda} := U''_{n\lambda}\cap \pi^{-1}(U_\lambda)$ and 
$U'_{n\lambda} := U''_{n\lambda}\cap \pi^{-1}(U'_\lambda)$.
We have $\pi^{-1}(U_\lambda) = \bigsqcup_{n\in \mathbb{Z}}U_{n\lambda}$ and
$\pi^{-1}(U'_\lambda) = \bigsqcup_{n\in \mathbb{Z}}U'_{n\lambda}$.

\noindent
(iii) A trivialization of the principal $SU(2)$-bundle $\underbar{E}$ over each 
$U''_\lambda$ is given.

From the conditions (ii) and (iii), 
we have a coordinate system and a trivialization of $E$ over each $U''_{n\lambda}$.
Let $u$ be a section of $\Lambda^i(\ad E)$ $(0\leq i\leq 4)$ over $S^3\times \mathbb{R}$. 
Then $u|_{U''_{n\lambda}}$ can be seen as a vector-valued function over $U''_{n\lambda}$.
Hence we can consider the H\"{o}lder norm $\norm{u}_{\mathcal{C}^{k, \alpha}(\bar{U}_{n\lambda})}$ 
of $u$ as a vector-valued function over $\bar{U}_{n\lambda}$
(cf. Gilbarg-Trudinger \cite[Chapter 4]{Gilbarg-Trudinger}).
We define the H\"{o}lder norm $\norm{u}_{\mathcal{C}^{k, \alpha}}$ by setting 
\[ \norm{u}_{\mathcal{C}^{k, \alpha}} := \sup_{n\in \mathbb{Z}, 1\leq \lambda\leq \Lambda} 
   \norm{u}_{\mathcal{C}^{k,\alpha}(\bar{U}_{n\lambda})}.\]
For $a\in H^1_A$, we have $\norm{a}_{\mathcal{C}^{k,\alpha}} \leq \const_k \norm{a}_{L^\infty} <\infty$ for every
$k=0,1,2,\cdots$ by the elliptic regularity.
We define the Banach space $\mathcal{C}^{k, \alpha}(\Lambda^+(\ad E))$ as the space of 
sections $u$ of $\Lambda^+(\ad E)$ over $S^3\times \mathbb{R}$ satisfying $\norm{u}_{\mathcal{C}^{k,\alpha}} <\infty$.

Consider the following map:
\begin{equation*} 
 \Phi: H^1_A\times \mathcal{C}^{2,\alpha}(\Lambda^+(\ad E)) \to \mathcal{C}^{0,\alpha}(\Lambda^+(\ad E)), \quad 
 (a, \phi) \mapsto F^+(A+a+d_A^*\phi).
\end{equation*}
This is a smooth map between the Banach spaces.
Since $F^+(A+a) = (a\wedge a)^+$,
\begin{equation} \label{eq: F^+(A+a+d_A^*phi}
 F^+(A+a+d_A^*\phi) = (a\wedge a)^+ + d_A^+d_A^*\phi + [a\wedge d_A^*\phi]^+ + (d_A^*\phi\wedge d_A^*\phi)^+.
\end{equation}
The derivative of $\Phi$ with respect to the second variable $\phi$ at the origin $(0,0)$ is given by 
\begin{equation} \label{eq: implicit function theorem}
 \partial_2 \Phi_{(0,0)} = d_A^+ d_A^* = \frac{1}{2}(\nabla_A^*\nabla_A + S/3) :
 \mathcal{C}^{2,\alpha}(\Lambda^+(\ad E)) \to \mathcal{C}^{0,\alpha}(\Lambda^+(\ad E)).
\end{equation}
Here we have used the Weitzenb\"{o}ck formula (see (\ref{eq: Weitzenbock formula})).
\begin{proposition} \label{proposition: derivative is isomorphic}
The map $(\nabla_A^*\nabla_A+S/3):\mathcal{C}^{2,\alpha}(\Lambda^+(\ad E)) 
\to \mathcal{C}^{0,\alpha}(\Lambda^+(\ad E))$
is isomorphic.
\end{proposition}
\begin{proof}
The injectivity follows from the $L^\infty$-estimate of Proposition \ref{L^infty-estimate}.
So the problem is the surjectivity. First we prove the following lemma.
\begin{lemma} \label{lemma: partial surjectivity}
Suppose that $\eta\in\mathcal{C}^{0,\alpha}(\Lambda^+(\ad E))$ is compact-supported.
Then there exists $\phi\in \mathcal{C}^{2,\alpha}(\Lambda^+(\ad E))$ satisfying 
$(\nabla_A^*\nabla_A+S/3)\phi = \eta$ and 
$\norm{\phi}_{\mathcal{C}^{2,\alpha}}\leq \const \cdot \norm{\eta}_{\mathcal{C}^{0,\alpha}}$.
\end{lemma}
\begin{proof}
Set $L^2_1:=\{\xi\in L^2(\Lambda^+(\ad E))|\, \nabla_A\xi\in L^2\}$.
For $\xi_1, \xi_2\in L^2_1$, 
set $(\xi_1, \xi_2)_{S/3} := (S/3)(\xi_1, \xi_2)_{L^2}+(\nabla_A\xi_1, \nabla_A\xi_2)_{L^2}$.
Since $S$ is a positive constant, this inner product defines a norm equivalent to the standard $L^2_1$-norm.
$\eta$ defines a bounded linear functional 
$(\cdot, \eta)_{L^2}: L^2_1\to \mathbb{R}$, $\xi\mapsto (\xi, \eta)_{L^2}$.
From the Riesz representation theorem, there uniquely exists $\phi\in L^2_1$
satisfying $(\xi, \phi)_{S/3} = (\xi, \eta)_{L^2}$ for any $\xi\in L^2_1$.
This implies that $(\nabla_A^*\nabla_A+S/3)\phi = \eta$ in the sense of distributions.
Moreover we have $\norm{\phi}_{L^2_1}\leq \const \norm{\eta}_{L^2}$.
From the elliptic regularity (see Gilbarg-Trudinger \cite[Chapter 9]{Gilbarg-Trudinger})
and the Sobolev embedding $L^2_1\hookrightarrow L^4$,
\begin{equation*}
 \begin{split}
 \norm{\phi}_{L^4_2(U_{n\lambda})}&\leq 
 \const_{\lambda}(\norm{\phi}_{L^4(U'_{n\lambda})} + \norm{\eta}_{L^4(U'_{n\lambda})}) ,\\
 &\leq \const_{\lambda}(\norm{\phi}_{L^2_1(U'_{n\lambda})}+ \norm{\eta}_{L^4(U'_{n\lambda})}),\\
 &\leq \const_\lambda (\norm{\eta}_{L^2} + \norm{\eta}_{L^4}).
 \end{split}
\end{equation*}
Here $\const_\lambda$ are constants depending on $\lambda = 1,2,\cdots,\Lambda$.
The important point is that they are independent of $n\in \mathbb{Z}$.
This is because we have the $T\mathbb{Z}$-symmetry of the equation.
From the Sobolev embedding $L^4_2\hookrightarrow L^\infty$, we have 
\[ \norm{\phi}_{L^\infty}\leq \const\cdot \sup_{n,\lambda}\norm{\phi}_{L^4_2(U_{n\lambda})} \leq
   \const(\norm{\eta}_{L^2}+\norm{\eta}_{L^4}) < \infty.\]
Using the Schauder interior estimate (see Gilbarg-Trudinger \cite[Chapter 6]{Gilbarg-Trudinger}),
we get 
\[ \norm{\phi}_{\mathcal{C}^{2,\alpha}(\bar{U}_{n\lambda})}\leq 
   \const_\lambda (\norm{\phi}_{L^\infty} + \norm{\eta}_{\mathcal{C}^{0,\alpha}(\bar{U}'_{n\lambda})}).\]
From Proposition \ref{L^infty-estimate}, we get $\norm{\phi}_{L^\infty}\leq (24/S)\norm{\eta}_{L^\infty}$.
It is easy to see that 
\begin{equation}\label{eq: C^{0,alpha}-norm over U' leq U}
 \sup_{n,\lambda}\norm{\eta}_{\mathcal{C}^{0,\alpha}(\bar{U}'_{n\lambda})}
 \leq \const \norm{\eta}_{\mathcal{C}^{0,\alpha}}.
\end{equation}
(Recall 
$\norm{\eta}_{\mathcal{C}^{0,\alpha}} 
=\sup_{n,\lambda}\norm{\eta}_{\mathcal{C}^{0,\alpha}(\bar{U}_{n\lambda})}$.)
Hence $\norm{\phi}_{\mathcal{C}^{2,\alpha}}\leq 
\const (\norm{\eta}_{L^\infty} + \norm{\eta}_{\mathcal{C}^{0,\alpha}})\leq \const \norm{\eta}_{\mathcal{C}^{0,\alpha}}$.
\end{proof}
Let $\eta\in \mathcal{C}^{0,\alpha}(\Lambda^+(\ad E))$ (not necessarily compact-supported).
Let $\varphi_k$ $(k=1,2,\cdots)$ be cut-off functions such that $0\leq \varphi_k\leq 1$, 
$\varphi_k=1$ over $|t|\leq k$ and $\varphi_k=0$ over $|t|\geq k+1$.
Set $\eta_k :=\varphi_k\eta$.
From the above Lemma \ref{lemma: partial surjectivity}, 
there exists $\phi_k \in \mathcal{C}^{2,\alpha}(\Lambda^+(\ad E))$ satisfying 
$(\nabla_A^*\nabla_A+S/3)\phi_k = \eta_k$.
From the $L^\infty$-estimate (Proposition \ref{L^infty-estimate}), we get 
\[ \norm{\phi_k}_{L^\infty}\leq (24/S)\norm{\eta_k}_{L^\infty}\leq (24/S)\norm{\eta}_{L^\infty}.\]
From the Schauder interior estimate, we get 
\[ \norm{\phi_k}_{\mathcal{C}^{2,\alpha}(\bar{U}_{n\lambda})}\leq 
 \const_\lambda\cdot (\norm{\phi_k}_{L^\infty(U'_{n\lambda})}
 +\norm{\eta_k}_{\mathcal{C}^{0,\alpha}(\bar{U}'_{n\lambda})})
\leq \const (\norm{\eta}_{L^\infty} + \norm{\eta_k}_{\mathcal{C}^{0,\alpha}(\bar{U}'_{n\lambda})}).\]
We have $\eta_k=\eta$ over each $U'_{n\lambda}$ for $k\gg 1$.
Hence $\norm{\phi_k}_{\mathcal{C}^{2,\alpha}(\bar{U}_{n\lambda})}$ $(k\geq 1)$ is bounded
for each $(n,\lambda)$.
Therefore, if we take a subsequence, $\phi_k$ converges to 
a $\mathcal{C}^2$-section $\phi$ of $\Lambda^+(\ad E)$
in the $\mathcal{C}^2$-topology over every compact subset.
$\phi$ satisfies $(\nabla_A^*\nabla_A+S/3)\phi = \eta$ and 
$\norm{\phi}_{L^\infty}\leq (24/S)\norm{\eta}_{L^\infty}$.
The Schauder interior estimate gives 
\[ \norm{\phi}_{\mathcal{C}^{2,\alpha}(\bar{U}_{n\lambda})}\leq 
\const_\lambda(\norm{\phi}_{L^\infty}+\norm{\eta}_{\mathcal{C}^{0,\alpha}(\bar{U}'_{n\lambda})}).\]
By (\ref{eq: C^{0,alpha}-norm over U' leq U}),
we get $\norm{\phi}_{\mathcal{C}^{2,\alpha}}\leq \const \norm{\eta}_{\mathcal{C}^{0,\alpha}}<\infty$.
\end{proof}
Since the map (\ref{eq: implicit function theorem}) is isomorphic,
the implicit function theorem implies that there exist $\delta_2>0$ and $\delta_3>0$ such that 
for any $a\in H^1_A$ with $\norm{a}_{L^\infty}\leq \delta_2$ there uniquely exists 
$\phi_a\in \mathcal{C}^{2,\alpha}(\Lambda^+(\ad E))$ with $\norm{\phi_a}_{\mathcal{C}^{2,\alpha}}\leq \delta_3$
satisfying $F^+(A+a+d_A^*\phi_a) =0$, i.e., 
\begin{equation} \label{eq: ASD equation for phi_a}
 d_A^+d_A^*\phi_a + [a\wedge d_A^*\phi_a]^+ + (d_A^*\phi_a\wedge d_A^*\phi_a)^+ = -(a\wedge a)^+.
\end{equation}
Here the ``uniqueness" means that if $\phi \in \mathcal{C}^{2,\alpha}(\Lambda^+(\ad E))$ with 
$\norm{\phi}_{\mathcal{C}^{2,\alpha}}\leq \delta_3$ satisfies $F^+(A+a+d_A^*\phi) = 0$ then 
$\phi = \phi_a$.
From the elliptic regularity, $\phi_a$ is smooth.
We have $\phi_0 =0$ and 
\begin{equation}\label{eq: consequence of implicit function theorem}
\norm{\phi_a}_{\mathcal{C}^{2,\alpha}}\leq \const \norm{a}_{L^\infty}, \quad
\norm{\phi_a-\phi_b}_{\mathcal{C}^{2,\alpha}}\leq \const \norm{a-b}_{L^\infty},
\end{equation}
for any $a, b\in H^1_A$ with $\norm{a}_{L^\infty}, \norm{b}_{L^\infty}\leq \delta_2$.
The map $a\mapsto \phi_a$ is $T$-equivariant, i.e., 
$\phi_{T^*a}=T^*\phi_a$ where $T:S^3\times \mathbb{R}\to S^3\times \mathbb{R}, (\theta, t)\mapsto (\theta, t+T)$.

We have $F(A+a+d_A^*\phi_a) = F(A+a) + d_Ad_A^*\phi_a+[a\wedge d_A^*\phi_a] + d_A^*\phi_a\wedge d_A^*\phi_a$.
From (\ref{eq: strict bound on the curvature}), if we choose $\delta_2>0$ sufficiently small, 
\begin{equation}\label{eq: L^infty bound on the curvature of deformed connection}
\norm{F(A+a+d_A^*\phi_a)}_{L^\infty}\leq \norm{F(A)}_{L^\infty} +\const \cdot \delta_2 \leq d.
\end{equation}
Moreover we can choose $\delta_2>0$ so that, for any $a\in H^1_A$ with $\norm{a}_{L^\infty}\leq \delta_2$,  
\begin{equation} \label{eq: L^infty norm of a+d_A^*phi_a}
 \norm{a+d_A^*\phi_a}_{L^\infty}\leq  \const\cdot \delta_2  \leq   \varepsilon_2 ,
\end{equation}
where $\varepsilon_2$ is the positive constant introduced in 
Proposition \ref{prop: Coulomb gauge condition}.
\begin{lemma} \label{lemma: injectivity}
We can take the above constant $\delta_2>0$ sufficiently small so that, 
if $a, b\in H^1_A$ with 
$\norm{a}_{L^\infty}, \norm{b}_{L^\infty}\leq \delta_2$ satisfy $a+d_A^*\phi_a = b+d_A^*\phi_b$,
then $a=b$.
\end{lemma}
\begin{proof}
By (\ref{eq: ASD equation for phi_a}),
\begin{equation}
 \begin{split}
 &\frac{1}{2}(\nabla_A^*\nabla_A+S/3)(\phi_a-\phi_b) = d_A^+d_A^*(\phi_a-\phi_b) \\
     &= (b\wedge (b-a))^++((b-a)\wedge a)^+ + 
     [b\wedge (d_A^*\phi_b-d_A^*\phi_a)]^+ + [(b-a)\wedge d_A^*\phi_a]^+ \\
    &+ (d_A^*\phi_b\wedge (d_A^*\phi_b-d_A^*\phi_a))^+ 
    +((d_A^*\phi_b-d_A^*\phi_a)\wedge d_A^*\phi_a)^+    
 \end{split}
\end{equation}
Its $\mathcal{C}^{0,\alpha}$-norm is bounded by 
\begin{equation*}
 \begin{split}
 \const (\norm{a}_{\mathcal{C}^{0,\alpha}}+ &\norm{b}_{\mathcal{C}^{0,\alpha}} + 
 \norm{d_A^*\phi_a}_{\mathcal{C}^{0,\alpha}}) \norm{a-b}_{\mathcal{C}^{0,\alpha}}  \\
 &+ \const (\norm{b}_{\mathcal{C}^{0,\alpha}}+\norm{d_A^*\phi_a}_{\mathcal{C}^{0,\alpha}}
 +\norm{d_A^*\phi_b}_{\mathcal{C}^{0,\alpha}})
    \norm{d_A^*\phi_a-d_A^*\phi_b}_{\mathcal{C}^{0,\alpha}}. 
 \end{split}
\end{equation*}
From (\ref{eq: consequence of implicit function theorem}), this is bounded by $\const \cdot \delta_2\norm{a-b}_{L^\infty}$.
Then Proposition \ref{proposition: derivative is isomorphic} implies
\[ \norm{\phi_a-\phi_b}_{\mathcal{C}^{2,\alpha}}\leq \const \cdot \delta_2\norm{a-b}_{L^\infty}.\]
Hence, if $a+d_A^*\phi_a=b+d_A^*\phi_b$ then 
\[ \norm{a-b}_{L^\infty} = \norm{d_A^*\phi_a-d_A^*\phi_b}_{L^\infty} \leq \const\cdot \delta_2\norm{a-b}_{L^\infty}.\]
If $\delta_2$ is sufficiently small, then this implies $a=b$.
\end{proof}
For $r>0$, we set $B_r(H^1_A):=\{a\in H^1_A|\, \norm{a}_{L^\infty}\leq r\}$.
\begin{lemma}\label{lemma: continuity under the topology of compact uniform convergence}
Let $\{a_n\}_{n\geq 1}\subset B_{\delta_2}(H^1_A)$ and suppose that this sequence converges to 
$a\in B_{\delta_2}(H^1_A)$ in the topology of uniform convergence over compact subsets, 
i.e., for any compact set $K\subset S^3\times \mathbb{R}$,
$\norm{a_n-a}_{L^\infty(K)}\to 0$ as $n\to \infty$.
Then $d_A^*\phi_{a_n}$ converges to $d_A^*\phi_a$ in the $\mathcal{C}^\infty$-topology 
over every compact subset in $S^3\times \mathbb{R}$.
\end{lemma}
\begin{proof}
It is enough to prove that there exists a subsequence (also denoted by $\{a_n\}$) such that 
$d_A^*\phi_{a_n}$ converges to $d_A^*\phi_a$ in the topology of $\mathcal{C}^\infty$-convergence 
over compact subsets in $S^3\times \mathbb{R}$.
From the elliptic regularity, $a_n$ converges to $a$ in the $\mathcal{C}^\infty$-topology over every 
compact subset.
Hence, for each $k\geq 0$ and each compact subset $K$ in $X$, 
the $\mathcal{C}^k$-norms of $\phi_{a_n}$ over $K$ $(n\geq 1)$ are bounded
by the equation (\ref{eq: ASD equation for phi_a}) and $\norm{\phi_{a_n}}_{\mathcal{C}^{2,\alpha}}\leq \delta_3$.
Then a subsequence of $\phi_{a_n}$ 
converges to some $\phi$ in the $\mathcal{C}^\infty$-topology over every compact subset.
We have $\norm{\phi}_{\mathcal{C}^{2,\alpha}}\leq \delta_3$ and $F^+(A+a+d_A^*\phi) =0$.
Then the uniqueness of $\phi_a$ implies $\phi =\phi_a$.
\end{proof}

Consider the following map 
(cf. the description of $\moduli_d$ in Remark \ref{remark: another description of moduli_d}):
\begin{equation}\label{eq: deformation map}
 B_{\delta_2}(H^1_A) \to \moduli_d, \quad a\mapsto [E,A+a+d_A^*\phi_a].
\end{equation}
Note that we have $|F(A+a+d_A^*\phi_a)|\leq d$ 
(see (\ref{eq: L^infty bound on the curvature of deformed connection})), and hence this map is well-defined.
$B_{\delta_2}(H^1_A)$ is equipped with the topology of uniform convergence over compact subsets.
($B_{\delta_2}(H^1_A)$ becomes compact and metrizable.)
The map (\ref{eq: deformation map}) is continuous 
by Lemma \ref{lemma: continuity under the topology of compact uniform convergence}.
$T\mathbb{Z}$ naturally acts on $B_{\delta_2}(H^1_A)$, and the map (\ref{eq: deformation map}) is 
$T\mathbb{Z}$-equivariant.
($\moduli_d$ is equipped with the action of $T\mathbb{Z}$ induced by the action of $\mathbb{R}$.)
\begin{lemma} \label{lemma: deformation map is injective}
The map (\ref{eq: deformation map}) is injective for sufficiently small $\delta_2>0$.
\end{lemma}
\begin{proof}
Let $a, b\in B_{\delta_2}(H^1_A)$, and suppose that there exists a gauge transformation $g:E\to E$
satisfying $g(A+a+d_A^*\phi_a) = A+b+d_A^*\phi_b$.
We have $d_A^*(a+d_A^*\phi_a) = d_A^*(b+d_A^*\phi_b) =0$ and 
$\norm{a+d_A^*\phi_a}_{L^\infty}, \norm{b+d_A^*\phi_b}_{L^\infty}\leq \varepsilon_2$ 
(see (\ref{eq: L^infty norm of a+d_A^*phi_a})).
Then Proposition \ref{prop: Coulomb gauge condition} implies $a+d_A^*\phi_a = b+d_A^*\phi_b$.
Then we have $a=b$ by Lemma \ref{lemma: injectivity}.
\end{proof}
Therefore the map (\ref{eq: deformation map}) becomes a $T\mathbb{Z}$-equivariant topological embedding.
Hence 
\begin{equation} \label{eq: lower bound on the local mean dimension by a Banach space}
 \dim_{[E,A]}(\moduli_d :T\mathbb{Z}) \geq \dim_0(B_{\delta_2}(H^1_A):T\mathbb{Z}).
\end{equation}
The right-hand-side is the local mean dimension of $(B_{\delta_2}(H^1_A), T\mathbb{Z})$ at the origin.
We define a distance on $B_{\delta_2}(H^1_A)$ by 
\[ \dist(a,b):=\sum_{n\geq 0}2^{-n}\norm{a-b}_{L^\infty(|t|\leq (n+1)T)} 
   \quad (a,b\in B_{\delta_2}(H^1_{A})).\]
   Set $\Omega_n:=\{0,T,2T,\cdots,(n-1)T\} \subset T\mathbb{Z}$ $(n\geq 1)$.
$\{\Omega_n\}_{n\geq 1}$ is an amenable sequence in $T\mathbb{Z}$.
For $a,b\in B_{\delta_2}(H^1_A)$,
\begin{equation} \label{eq: dist_{Omega_n} geq L^infty-norm}
 \dist_{\Omega_n}(a,b) \geq \norm{a-b}_{L^\infty(0\leq t\leq nT)}.
\end{equation}

For each $n\geq 1$, let 
$\pi_n:S^3\times (\mathbb{R}/nT\mathbb{Z}) \to S^3\times (\mathbb{R}/T\mathbb{Z})$ be the natural 
$n$-hold covering, and set $E_n := \pi_n^*(\underbar{E})$ and $A_n:=\pi_n^*(\underbar{A})$.
We denote $H^1_{A_n}$ as the space of $a\in \Omega^1(\ad E_n)$ over $S^3\times (\mathbb{R}/nT\mathbb{Z})$
satisfying $(d_{A_n}^+ +d_{A_n}^*)a=0$.
We can identify $H^1_{A_n}$ with the subspace of $H^1_A$ consisting of $nT$-invariant elements.
The index formula gives $\dim H^1_{A_n} = 8nc_2(\underbar{E})$.
(We have $H^0_{A_n}=H^2_{A_n}=0$.)
From (\ref{eq: dist_{Omega_n} geq L^infty-norm}), 
for $a,b\in B_{\delta_2}(H^1_{A_n}) := \{u\in H^1_{A_n}|\, \norm{u}_{L^\infty(X)}\leq \delta_2\}$
\begin{equation} \label{eq: dist_{Omega_n} geq L^infty-norm for the element of H^1_{A_n}}
  \dist_{\Omega_n}(a,b) \geq \norm{a-b}_{L^\infty(X)}.
\end{equation}

Let $0<r<\delta_2$. 
Since $\dist_{T\mathbb{Z}}(a,b)\leq 2\norm{a-b}_{L^\infty}$, we have 
$B_{r/2}(H^1_A)\subset B_r(0;B_{\delta_2}(H^1_A))_{T\mathbb{Z}}$.
Here $B_r(0;B_{\delta_2}(H^1_A))_{T\mathbb{Z}}$ is the closed $r$-ball 
centered at $0$ in $B_{\delta_2}(H^1_A)$
with respect to the distance $\dist_{T\mathbb{Z}}(\cdot,\cdot)$.
From (\ref{eq: dist_{Omega_n} geq L^infty-norm for the element of H^1_{A_n}}) and 
Lemma \ref{lemma: widim of Banach ball}, for $\varepsilon<r/2$
\begin{equation*}
 \begin{split}
   \widim_\varepsilon(B_r(0;B_{\delta_2}(H^1_A))_{T\mathbb{Z}},\dist_{\Omega_n})&\geq 
   \widim_\varepsilon(B_{r/2}(H^1_A),\dist_{\Omega_n}) \\
   &\geq \widim_\varepsilon(B_{r/2}(H^1_{A_n}),\norm{\cdot}_{L^\infty})
   = \dim H^1_{A_n} = 8nc_2(\underbar{E}).
 \end{split}
\end{equation*}
Hence, for $\varepsilon <r/2$,
\begin{equation*}
 \begin{split}
 \widim_\varepsilon(B_r(0;B_{\delta_2}(H^1_A))_{T\mathbb{Z}}&\subset B_{\delta_2}(H^1_A):T\mathbb{Z}) \\
 &\geq \limsup_{n\to \infty}\left(\frac{1}{n}\widim_\varepsilon(B_r(0;B_{\delta_2}(H^1_A))_{T\mathbb{Z}},\dist_{\Omega_n}) \right)
 \geq 8c_2(\underbar{E}).
 \end{split}
\end{equation*}
Let $\varepsilon\to 0$. Then 
\[ \dim(B_r(0;B_{\delta_2}(H^1_A))_{T\mathbb{Z}}\subset B_{\delta_2}(H^1_A):T\mathbb{Z}) \geq 8c_2(\underbar{E}).\]
Let $r\to 0$. We get $\dim_0(B_{\delta_2}(H^1_A):T\mathbb{Z})\geq 8c_2(\underbar{E})$.
From (\ref{eq: lower bound on the local mean dimension by a Banach space}) and 
Proposition \ref{prop: mean dimension of flow and shift},
\[ \dim_{[E,A]}(\moduli_d:\mathbb{R}) = \dim_{[E,A]}(\moduli_d:T\mathbb{Z})/T 
   \geq 8c_2(\underbar{E})/T = 8\rho(A).\]
Therefore we get the conclusion:
\begin{theorem}
If $\bm{A}$ is a periodic ASD connection on $\bm{E}$ satisfying $\norm{F(\bm{A})}_{L^\infty(X)} <d$,
then 
\[ \dim_{[\bm{A}]}(\moduli_d:\mathbb{R}) = 8\rho(\bm{A}).\]
\end{theorem}
\begin{proof}
The upper-bound $\dim_{[\bm{A}]}(\moduli_d:\mathbb{R})\leq 8\rho(\bm{A})$
was already proved in Section \ref{subsection: upper bound on the local mean dimension}.
If $\bm{A}$ is not flat, then the above argument shows that we also have the lower-bound
$\dim_{[\bm{A}]}(\moduli_d:\mathbb{R})\geq 8\rho(\bm{A})$.
If $\bm{A}$ is flat, then $\dim_{[\bm{A}]}(\moduli_d:\mathbb{R})\geq 0 = 8\rho(\bm{A})$.
Hence $\dim_{[\bm{A}]}(\moduli_d:\mathbb{R}) = 8\rho(\bm{A})$.
\end{proof}
We have completed all the proofs of Theorem \ref{thm: main theorem} and
Theorem \ref{thm: main theorem on the local mean dimension}.

\appendix

\section{Green kernel} \label{appendix: Green kernel}
In this appendix, we prepare some basic facts on a Green kernel over $S^3\times \mathbb{R}$.
Let $a>0$ be a positive constant. 
Some constants introduced in this appendix depend on $a$,
but we don't explicitly write their dependence on $a$ for simplicity of the explanation.
In the main body of the paper we have $a= S/3$ ($S$ is the scalar curvature of $S^3\times \mathbb{R}$), and 
its value is fixed throughout the argument.
Hence we don't need to care about the dependence on $a=S/3$.
\subsection{$(\Delta+a)$ on functions} \label{subsection: (Delta+1) on functions}
Let $\Delta := \nabla^*\nabla$ be the Laplacian on functions over $S^3\times \mathbb{R}$.
(Notice that the sign convention of our Laplacian $\Delta = \nabla^*\nabla$ is ``geometric''.
For example, we have $\Delta = -\sum_{i=1}^4 \partial^2/\partial x_i^2$ on the Euclidean space $\mathbb{R}^4$.)
Let $g(x, y)$ be the Green kernel of $\Delta+a$;
\[ (\Delta_y +a)g(x, y) = \delta_x(y).\]
This equation means that 
\[ \phi(x) = \int_{S^3\times \mathbb{R}}g(x,y)(\Delta_y+a)\phi(y)d\vol(y),\]
for compact-supported smooth functions $\phi$.
The existence of $g(x,y)$ is essentially standard (\cite[Chapter 4]{Aubin}).
We briefly explain how to construct it.
We fix $x \in S^3\times \mathbb{R}$ and construct a
function $g_{x}(y)$ satisfying $(\Delta+a)g_{x} = \delta_{x}$.
As in \cite[Chapter 4, Section 2]{Aubin}, by using a local coordinate around $x$, we can construct (by hand) 
a compact-supported function $g_{0,x}(y)$ satisfying 
\[ (\Delta+a)g_{0,x} = \delta_x - g_{1,x},\]
where $g_{1,x}$ is a compact supported continuous function.
Moreover $g_{0,x}$ is smooth outside $\{x\}$ and it satisfies
\[ \const_1/d(x, y)^2 \leq g_{0,x}(y) \leq \const_2/d(x,y)^2,\]
for some positive constants $\const_1$ and $\const_2$ in some small neighborhood of $x$.
Here $d(x,y)$ is the distance between $x$ and $y$.
Since $(\Delta+a):L^2_2\to L^2$ is isomorphic, there exists 
$g_{2,x}\in L^2_2$ satisfying $(\Delta+a) g_{2,x} = g_{1,x}$.
($g_{2,x}$ is of class $\mathcal{C}^1$.)
Then $g_{x} := g_{0,x}+g_{2,x}$ satisfies $(\Delta+a)g_{x}=\delta_{x}$, and
$g(x, y) := g_{x}(y)$ becomes the Green kernel. 
$g(x,y)$ is smooth outside the diagonal.
Since $S^3\times \mathbb{R}=SU(2) \times \mathbb{R}$ is a Lie group and its 
Riemannian metric is two-sided invariant, we have 
$g(x,y) = g(zx,zy) = g(xz,yz)$ for $x,y,z\in S^3\times \mathbb{R}$.
$g(x,y)$ satisfies
\begin{equation}\label{eq: singularity of the Green function on the diagonal}
 c_1/d(x,y)^2\leq g(x,y) \leq c_2/d(x,y)^2 \quad (d(x,y)\leq \delta), 
\end{equation}
for some positive constants $c_1$, $c_2$, $\delta$. 
\begin{lemma}\label{lemma: positivity of the Green function}
$g(x,y)>0$ for $x\neq y$.
\end{lemma}
\begin{proof}
Fix $x = (\theta_0,t_0)\in S^3\times \mathbb{R}$. 
We have $(\Delta+a)g_x=0$ outside $\{x\}$, and hence (by elliptic regularity)
\[ |g_x(\theta, t)|\leq \const \norm{g_x}_{L^2(S^3\times [t-1,t+1])} \quad (|t-t_0| > 1).\]
Since the right-hand-side goes to zero as $|t|\to \infty$, $g_x$ vanishes at infinity.
Let $R>0$ be a large positive number and set 
$\Omega:= S^3\times [-R,R]\setminus B_\delta(x)$. ($\delta$ is a positive constant 
in (\ref{eq: singularity of the Green function on the diagonal}).)
Since $g_x(y)  \geq c_1/d(x,y)^2>0$ on $\partial B_\delta(x)$, we have 
$g_x\geq -\sup_{t=\pm R} |g_x(\theta, t)|$ on $\partial \Omega$.
Since $(\Delta+a)g_x=0$ on $\Omega$, we can apply the weak maximum (minimum) principle to $g_x$
(Gilbarg-Trudinger \cite[Chapter 3, Section 1]{Gilbarg-Trudinger}) and get 
\[ g_x(y) \geq -\sup_{t=\pm R} |g_x(\theta, t)| \quad (y\in \Omega).\]
The right-hand-side goes to zero as $R\to \infty$. 
Hence we have $g_x(y) \geq 0$ for $y\neq x$.
Since $g_x$ is not constant, the strong maximum principle (\cite[Chapter 3, Section 2]{Gilbarg-Trudinger})
implies that $g_x$ cannot achieve zero. Therefore $g_x(y)>0$ for $y\neq x$.
\end{proof}
\begin{lemma}\label{lemma: exponential decay of the Green function}
There exists $c_3>0$ such that 
\[ 0< g(x,y) \leq c_3 e^{-\sqrt{a}d(x,y)} \quad (d(x,y)\geq 1).\]
In particular, 
\[ \int_{S^3\times \mathbb{R}}g(x,y)d\vol(y) <\infty.\]
The value of this integral is independent of $x\in S^3\times \mathbb{R}$ because of the symmetry of $g(x,y)$.
\end{lemma}
\begin{proof}
We fix $x_0 = (\theta_0,0)\in S^3\times\mathbb{R}$.
Since $S^3\times \mathbb{R}$ is homogeneous, it is enough to show that $g_{x_0}(y) = g(x_0,y)$ satisfies
\[  g_{x_0}(y)\leq \const\cdot e^{-\sqrt{a}|t|} \quad 
(y=(\theta, t)\in S^3\times\mathbb{R} \text{ and $|t|\geq 1$}) .\]
Let $C:=\sup_{|t|=1}g_{x_0}(\theta, t)>0$, and set $u := Ce^{\sqrt{a}(1-|t|)}-g_{x_0}(y)$ $(|t|\geq 1)$.
We have $u\geq 0$ at $t=\pm 1$ and $(\Delta+a)u=0$ $(|t|\geq 1)$. $u$ goes to zero at infinity.
(See the proof of Lemma \ref{lemma: positivity of the Green function}.)
Hence we can apply the weak minimum principle (see the proof of Lemma \ref{lemma: positivity of the Green function})
to $u$ and get $u\geq 0$ for $|t|\geq 1$.
Thus $g_{x_0}(y)\leq Ce^{\sqrt{a}(1-|t|)}$ $(|t|\geq 1)$.
\end{proof}
The following technical lemma will be used in the next subsection.
\begin{lemma} \label{lemma: Green kernel representation}
Let $f$ be a smooth function over $S^3\times \mathbb{R}$.
Suppose that there exist non-negative functions $f_1, f_2 \in L^2$, $f_3\in L^1$ and 
$f_4,f_5,f_6\in L^\infty$ such that 
$|f|\leq f_1+f_4$, $|\nabla f| \leq f_2+f_5$ and 
$|\Delta f+af|\leq f_3+f_6$.
Then we have 
\[ f(x) = \int_{S^3\times \mathbb{R}}g(x,y)(\Delta_y+a)f(y)d\vol(y).\]
\end{lemma}
\begin{proof}
We fix $x\in S^3\times \mathbb{R}$.
Let $\rho_n$ $(n\geq 1)$ be cut-off functions satisfying $0\leq \rho_n\leq 1$,
$\rho_n = 1$ over $|t|\leq n$ and $\rho_n=0$ over $|t|\geq n+1$.
Moreover $|\nabla \rho_n|, |\Delta\rho_n|\leq \const$ (independent of $n\geq 1$).
Set $f_n := \rho_n f$.
We have 
\[ f_n(x) = \int g(x,y)(\Delta_y+a)f_n(y) d\vol(y) .\]
\[ (\Delta+a)f_n = \Delta\rho_n \cdot f -2\langle \nabla \rho_n, \nabla f\rangle + \rho_n(\Delta+a)f.\]
Note that $g_x(y) = g(x,y)$ is smooth outside $\{x\}$ and exponentially decreases as $y$ goes to infinity.
Hence for $n\gg 1$,
\[ \int g_x |\Delta\rho_n\cdot f|d\vol \leq C\sqrt{\int_{\supp (d\rho_n)}f_1^2\, d\vol} 
   + C\int_{\supp(d\rho_n)} g_xf_4\, d\vol(y).\]
Since $\supp(d\rho_n)\subset \{t\in [-n-1,-n]\cup [n,n+1]\}$ and $f_1\in L^2$ and $f_4\in L^\infty$,
the right-hand-side goes to zero as $n\to \infty$.
In the same way, we get 
\[ \int g_x |\langle \nabla \rho_n,\nabla f\rangle| d\vol \to 0 \quad (n\to \infty).\]
We have $g_x|\rho_n(\Delta+a)f|\leq g_x|\Delta f+af|$, and 
\begin{equation*}
 \begin{split}
 \int g_x(y)|\Delta f+af|d\vol \leq &\int_{d(x,y)\leq 1}g_x(y)|\Delta f+af|d\vol 
   + \left(\sup_{d(x,y)>1} g_x(y)\right) \int_{d(x,y)>1}f_3\, d\vol \\
   & + \int_{d(x,y)>1}g_x f_6 \, d\vol <\infty.
 \end{split}
\end{equation*}
Hence Lebesgue's theorem implies
\[ \lim_{n\to \infty} \int g_x \rho_n(\Delta+a)f\, d\vol = \int g_x (\Delta+a)f\, d\vol.\]
Therefore we get 
\[ f(x) = \int g_x (\Delta+a)f\, d\vol.\]
\end{proof}

\subsection{$(\nabla^*\nabla+a)$ on sections}
Let $E$ be a real vector bundle over $S^3\times \mathbb{R}$ with a fiberwise metric and 
a connection $\nabla$ compatible with the metric. 
\begin{lemma} \label{lemma: Green kernel bound on |phi|}
Let $\phi$ be a smooth section of $E$ such that $\norm{\phi}_{L^2}$, $\norm{\nabla\phi}_{L^2}$
and $\norm{\nabla^*\nabla\phi+a\phi}_{L^\infty}$ are finite.
Then $\phi$ satisfies
\[ |\phi(x)|\leq \int_{S^3\times \mathbb{R}}g(x,y)|\nabla^*\nabla\phi(y)+a\phi(y)|d\vol(y).\]
\end{lemma}
\begin{proof}
The following argument is essentially due to Donaldson \cite[p. 184]{Donaldson}.
Let $\underline{\mathbb{R}}$ be the product line bundle over $S^3\times \mathbb{R}$ 
with the product metric and the product connection.
Set $\phi_n := (\phi, 1/n)$ (a section of $E\oplus \underline{\mathbb{R}}$).
Then $|\phi_n|\geq 1/n$ and hence $\phi_n\neq 0$ at all points.
We want to apply Lemma \ref{lemma: Green kernel representation} to $|\phi_n|$.
$|\phi_n|\leq |\phi|+1/n$ where $|\phi|\in L^2$ and $1/n\in L^\infty$.
$\nabla\phi_n = (\nabla \phi, 0)$ and $\nabla^*\nabla\phi_n=(\nabla^*\nabla\phi,0)$.
We have the Kato inequality $|\nabla|\phi_n||\leq |\nabla \phi_n|$.
Hence $\nabla|\phi_n|\in L^2$.
From $\Delta |\phi_n|^2/2= (\nabla^*\nabla \phi_n, \phi_n) -|\nabla \phi_n|^2$,
\begin{equation}\label{eq: (Delta+1)|phi_n|}
 (\Delta+a) |\phi_n| = 
 (\nabla^*\nabla \phi_n + a\phi_n, \phi_n/|\phi_n|) - \frac{|\nabla\phi_n|^2-|\nabla|\phi_n||^2}{|\phi_n|}.
\end{equation} 
Hence (by using $|\phi_n|\geq 1/n$ and $|\nabla |\phi_n||\leq |\nabla \phi_n|$)
\[ |(\Delta+a) |\phi_n|| \leq |\nabla^*\nabla\phi_n +a\phi_n| + n|\nabla \phi_n|^2
\leq |\nabla^*\nabla\phi+a\phi|+a/n + n|\nabla\phi|^2.\]
$|\nabla^*\nabla\phi+a\phi|+a/n \in L^\infty$ and $n|\nabla\phi|^2\in L^1$.
Therefore we can apply Lemma \ref{lemma: Green kernel representation} to $|\phi_n|$ and get 
\[ |\phi_n(x)|=\int g(x,y)(\Delta_y+a)|\phi_n(y)|d\vol(y).\]
From (\ref{eq: (Delta+1)|phi_n|}) and the Kato inequality $|\nabla|\phi_n||\leq |\nabla \phi_n|$,
\[(\Delta_y+a)|\phi_n(y)|\leq |\nabla^*\nabla\phi_n+a\phi_n|\leq |\nabla^*\nabla\phi+a\phi|+a/n. \]
Thus
\[ |\phi_n(x)|\leq \int g(x,y)|\nabla^*\nabla\phi(y)+a\phi(y)|d\vol(y) 
  + \frac{a}{n}\int g(x,y)d\vol(y).\]
Let $n\to \infty$. Then we get the desired bound.
\end{proof}
\begin{proposition}  \label{prop: L^infty-estimate, appendix}
Let $\phi$ be a section of $E$ of class $\mathcal{C}^2$, and suppose that 
$\phi$ and $\eta:=(\nabla^*\nabla+a)\phi$ are contained in $L^\infty$.
Then 
\[ \norm{\phi}_{L^\infty}\leq (8/a)\norm{\eta}_{L^\infty}.\]
\end{proposition}
\begin{proof}
There exists a point $(\theta_1,t_1)\in S^3\times \mathbb{R}$ where 
$|\phi(\theta_1,t_1)|\geq \norm{\phi}_{L^\infty} /2$.
We have 
\[ \Delta|\phi|^2=2(\nabla^*\nabla\phi,\phi)-2|\nabla\phi|^2 = 2(\eta,\phi)-2a|\phi|^2-2|\nabla\phi|^2.\]
Set $M:=\norm{\phi}_{L^\infty}\norm{\eta}_{L^\infty}$.
Then
\[ (\Delta+2a)|\phi|^2 \leq 2(\eta,\phi)\leq 2M.\]
Define a function $f$ on $S^3\times \mathbb{R}$ by 
$f(\theta,t):= (2M/a)\cosh \sqrt{a}(t-t_1) = (M/a)(e^{\sqrt{a}(t-t_1)}+e^{\sqrt{a}(-t+t_1)})$.
Then $(\Delta+a)f=0$, and hence $(\Delta+2a)f=a f\geq 2M$.
Therefore 
\[ (\Delta+2a)(f-|\phi|^2)\geq 0.\]
Since $|\phi|$ is bounded and $f$ goes to $+\infty$ at infinity, we have $f-|\phi|^2>0$
for $|t|\gg 1$.
Then the weak minimum principle (\cite[Chapter 3, Section 1]{Gilbarg-Trudinger}) implies
$f(\theta_1, t_1)-|\phi(\theta_1,t_1)|^2\geq 0$.
This means that 
$\norm{\phi}_{L^\infty}^2/4 \leq |\phi(\theta_1,t_1)|^2\leq (2M/a) 
= (2/a)\norm{\phi}_{L^\infty}\norm{\eta}_{L^\infty}$.
Thus $\norm{\phi}_{L^\infty}\leq (8/a)\norm{\eta}_{L^\infty}$.
\end{proof}
\begin{lemma} \label{lemma: solve the equation in a special case}
Let $\eta$ be a compact-supported smooth section of $E$.
Then there exists a smooth section $\phi$ of $E$ satisfying $(\nabla^*\nabla+a)\phi=\eta$
and 
\[ |\phi(x)|\leq \int_{S^3\times \mathbb{R}} g(x,y)|\eta(y)|d\vol(y).\]
\end{lemma}
\begin{proof}
Set $L^2_1(E):=\{\xi\in L^2(E)|\, \nabla\xi\in L^2\}$ and 
$(\xi_1,\xi_2)_{a}:=(\nabla\xi_1,\nabla\xi_2)_{L^2}+a(\xi_1,\xi_2)_{L^2}$ for 
$\xi_1, \xi_2\in L^2_1(E)$. 
(Since $a>0$, this inner product defines a norm equivalent to the standard $L^2_1$-norm.)
$\eta$ defines the bounded functional 
\[ (\cdot, \eta)_{L^2}:L^2_1(E)\to \mathbb{R}, \quad \xi\mapsto (\xi,\eta)_{L^2}.\]
From the Riesz representation theorem, there uniquely exists $\phi\in L^2_1(E)$ 
satisfying $(\xi, \phi)_{a}=(\xi, \eta)_{L^2}$ for any $\xi\in L^2_1(E)$.
Then we have $(\nabla^*\nabla+a)\phi = \eta$ in the sense of distribution.
From the elliptic regularity, $\phi$ is smooth.
$\phi$ and $\nabla\phi$ are in $L^2$, and $(\nabla^*\nabla+a)\phi = \eta$ is in $L^\infty$.
Hence we can apply Lemma \ref{lemma: Green kernel bound on |phi|} to $\phi$ and get 
\[ |\phi(x)|\leq \int g(x,y)|\nabla^*\nabla\phi(y)+a\phi(y)|d\vol(y) = \int g(x,y)|\eta(y)|d\vol(y).\]
\end{proof}
\begin{proposition}  \label{prop: solve the equation in a general case}
Let $\eta$ be a smooth section of $E$ satisfying $\norm{\eta}_{L^\infty}<\infty$.
Then there exists a smooth section $\phi$ of $E$ satisfying 
$(\nabla^*\nabla+a)\phi =\eta$ and 
\begin{equation}\label{eq: Green kernel bound on phi in general case}
  |\phi(x)|\leq \int_{S^3\times \mathbb{R}}g(x,y)|\eta(y)|d\vol(y).
\end{equation}
(Hence $\phi$ is in $L^\infty$.)
In particular, if $\eta$ vanishes at infinity, then $\phi$ also vanishes at infinity.
Moreover, if a smooth section $\phi'\in L^\infty(E)$ satisfies $(\nabla^*\nabla+a)\phi' = \eta$
($\eta$ does not necessarily vanishes at infinity),
then $\phi' = \phi$.
\end{proposition}
\begin{proof}
Let $\rho_n$ $(n\geq 1)$ be the cut-off functions introduced in the proof of
Lemma \ref{lemma: Green kernel representation}, 
and set $\eta_n :=\rho_n\eta$. From Lemma \ref{lemma: solve the equation in a special case},
there exists a smooth section $\phi_n$ satisfying 
$(\nabla^*\nabla+a)\phi_n=\eta_n$ and 
\begin{equation} \label{eq: Green kernel bound on phi_n}
 |\phi_n(x)|\leq \int g(x,y)|\eta_n(y)|d\vol(y) \leq \int g(x,y) |\eta(y)|d\vol(y).
\end{equation}
Hence $\{\phi_n\}_{n\geq 1}$ is uniformly bounded.
Then by using the Schauder interior estimate (\cite[Chapter 6]{Gilbarg-Trudinger}),
for any compact set $K\subset S^3\times\mathbb{R}$,
the $\mathcal{C}^{2,\alpha}$-norms of $\phi_n$ over $K$ are bounded $(0<\alpha<1)$.
Hence there exists a subsequence $\{\phi_{n_k}\}_{k\geq 1}$ and a section $\phi$ of $E$ such that 
$\phi_{n_k}\to \phi$ in the $\mathcal{C}^2$-topology over every compact subset in $S^3\times \mathbb{R}$.
Then $\phi$ satisfies $(\nabla^*\nabla+a)\phi = \eta$.
$\phi$ is smooth by the elliptic regularity, 
and it satisfies (\ref{eq: Green kernel bound on phi in general case})
from (\ref{eq: Green kernel bound on phi_n}).

Suppose $\eta$ vanishes at infinity. 
Set $K:= \int g(x,y)d\vol(y) <\infty$ (independent of $x$).
For any $\varepsilon>0$, there exists a compact set $\Omega_1\subset S^3\times \mathbb{R}$ such that
$|\eta|\leq \varepsilon/(2K)$ on the complement of $\Omega_1$.
There exists a compact set $\Omega_2\supset \Omega_1$ such that for any $x\not\in\Omega_2$
\[ \norm{\eta}_{L^\infty} \int_{\Omega_1}g(x,y) d\vol(y) \leq \varepsilon/2.\]
Then from (\ref{eq: Green kernel bound on phi in general case}), for $x\not\in\Omega_2$, 
\[  |\phi(x)|\leq \int_{\Omega_1}g(x,y)|\eta(y)|d\vol(y) + \int_{\Omega_1^c}g(x,y)|\eta(y)|d\vol(y)
    \leq \varepsilon/2 +\varepsilon/2 = \varepsilon.\]
This shows that $\phi$ vanishes at infinity. 

Suppose that smooth $\phi'\in L^\infty(E)$ satisfies $(\nabla^*\nabla+a)\phi' = \eta$.
We have $(\nabla^*\nabla+a)(\phi-\phi')=0$, and $\phi-\phi'$ is contained in $L^\infty$.
Then the $L^\infty$-estimate in Proposition \ref{prop: L^infty-estimate, appendix} implies 
$\phi-\phi'=0$.
\end{proof}

\vspace{10mm}

\address{ Shinichiroh Matsuo \endgraf
Graduate School of Mathematical Sciences, University of Tokyo, 
3-8-1 Komaba Meguro-ku, Tokyo 153-8914, Japan}

\textit{E-mail address}: \texttt{exotic@ms.u-tokyo.ac.jp}

\vspace{0.5cm}

\address{ Masaki Tsukamoto \endgraf
Department of Mathematics, Kyoto University, Kyoto 606-8502, Japan}

\textit{E-mail address}: \texttt{tukamoto@math.kyoto-u.ac.jp}

\end{document}